\documentclass[12pt]{amsart}
\usepackage[a4paper,margin=26mm]{geometry}

\usepackage[T1]{fontenc}
\usepackage[utf8]{inputenc}
\usepackage{amsmath,amssymb}            % 'not essential, but very useful'
\usepackage{mathtools}
\usepackage{hyperref}
\usepackage{ae}
\usepackage{comment}
\usepackage{bm}
\usepackage{graphicx}
\usepackage{xcolor}
\newcommand{\N}{\mathbb{N}}
\newcommand{\Z}{\mathbb{Z}}

\newcommand{\R}{\mathbb{R}}
\newcommand{\C}{\mathbb{C}}
\newcommand{\U}{\mathbb{U}}

\DeclareMathOperator{\e}{\mathrm{e}}
\DeclareMathOperator{\s}{\mathrm{s}}
\DeclareMathOperator{\rep}{\mathrm{r}}

\newcommand{\floor}[1]{\left\lfloor #1 \right\rfloor}
\newcommand{\ceil}[1]{\left\lceil #1 \right\rceil}
\newcommand{\frp}[1]{\left\{ #1 \right\}}
\newcommand{\abs}[1]{\left| #1 \right|}
\newcommand{\dv}{\, \mid \,}

\newcommand{\norm}[1]{\left\| #1 \right\|}

\newcommand{\fourier}[1]{\widehat{#1}}
\newcommand{\conjugate}[1]{\overline{#1}}

\newcommand{\one}{\mathbf{1}}

\newtheorem{theorem}{Theorem}
\newtheorem{proposition}{Proposition}

\newtheorem{lemma}{Lemma}

\newtheorem{definition}{Definition}

\newtheorem{remark}{Remark}

\numberwithin{equation}{section}

\begin{document}
\title{Digital functions along Squares of Prime Numbers}

% 'Each author has his or her own set of coordinates.'
\author{Michael Drmota}
\email{michael.drmota@tuwien.ac.at}
\address{Institut f\"ur Diskrete Mathematik und Geometrie
TU Wien\\
Wiedner Hauptstr. 8--10\\
1040 Wien\\
Austria.}
\author{Jo\"el Rivat}
\email{joel.rivat@univ-amu.fr}
\address{Universit\'e d'Aix-Marseille\\
Institut Universitaire de France\\
Institut de Math\'ematiques de Marseille\\
CNRS UMR 7373\\
163, avenue de Luminy, Case 907\\
13288 MARSEILLE Cedex 9\\
France.}
%

% 'MSC classification, keywords and grant acknowledgements'
\date{\today} 
\thanks{The first author is supported by the
Austrian Science Foundation FWF, Grant I554
This work was supported by the joint ANR-FWF Grant 4945-N and ANR Grant 20-CE91-0006}

\begin{abstract}
  In the last 20 years the Gelfond conjectures concerning the well
  distribution of the sum-of-digits function along prime numbers and
  along squares have been solved
  and these results, which are strongly connected with the Sarnak
  conjecture,
  were generalized to $q$-multiplicative functions
  and automatic sequences.  
  In this paper we study a combination of both challenges and prove a
  Prime Number Theorem for $q$-multiplicative functions along squares
  which can be rewritten into a well distribution result along squares
  of primes.
\end{abstract}

\maketitle

\begin{center}
  \begin{minipage}[c]{0.8\linewidth}
    % table of contents depth
     \setcounter{tocdepth}{1}
    \tableofcontents  
  \end{minipage}  
\end{center}

\section{Introduction}

It is a classical and still modern problem in Analytic Number Theory to 
decide whether a given (number-theoretic) complex valued sequence $f(n)$ is 
{\it M\"obius disjoint}, that is,
\begin{displaymath}
  \sum_{n\le x} \mu(n) f(n) = o(x) \qquad (x\to\infty),
\end{displaymath}
where $\mu(n)$ denotes the M\"obius function (defined by $\mu(1) = 1$,
$\mu(p_1\cdots p_k) = (-1)^k$ for a product of different prime numbers and
$\mu(n) = 0$ otherwise). It is very well known that the prime number
theorem is equivalent to the statement that $f(n)=1$ is
M\"obius disjoint, and that the Dirichlet prime number theorem is
equivalent to the property that periodic sequences $f(n)$ are M\"obius
disjoint.

Recently M\"obius Disjointness received a lot of attention due to the
{\it Sarnak Conjecture} \cite{bourgain-sarnak-ziegler-2013}
saying that all bounded deterministic
sequences $f(n)$ are M\"obius disjoint (actually the conjecture is
slightly more general and is formulated for sequences of the form
$f(n) = F(T^n x_0)$ for zero entropy dynamical systems $(X,T)$ on a
compact metric space $X$, with $x_0\in X$, and with a continuous
mapping $F:X\to \mathbb{C}$).  Deterministic sequences $f(n)$ can be
defined by the property that for all $\varepsilon > 0$, the set of
$k$-tuples
\begin{displaymath}
  \{(f (n + 0), \ldots , f (n + k - 1)) : n \ge 0\} \subseteq \mathbb{C}^k  
\end{displaymath}
can be covered by $\exp(o(k))$ many balls of radius $\varepsilon$, as
$k$ goes to infinity.  If the range of $f(n)$ is finite then this
means that the subword complexity of $f(n)$ is sub-exponential.

The Sarnak conjecture has been verified for several classes of
deterministic sequences, see
\cite{ferenczi-kulaga-lemanczyk-2018}
including nil-sequences \cite{green-tao-nilsequences-2012}
(the simplest examples of nil-sequences are Weyl sequences 
$f(n) = \exp(2\pi i \alpha n)$ with $\alpha\in \R$) 
and automatic sequences \cite{muellner-2017} 
(automatic sequences can be defined in severals ways
\cite{allouche-shallit-2003}, for example, as the image of a fixed point of constant
length morphisms on a finite alphabet; one of the most prominent example is the
Thue-Morse sequence $t(n) = (-1)^{s_2(n)} \bmod 2$, where $s_2(n)$
denotes the binary sum-of-digits function).

Similarly to M\"obius Disjointness, a sequence $f(n)$ is said to satisfy a 
{\it Prime Number Theorem (PNT)} if
\begin{displaymath}
  \sum_{n\le x} \Lambda(n) f(n) =  c x + o(x) \qquad (x\to\infty)  
\end{displaymath}
for some constant $c\in\C$, where $\Lambda(n)$ denotes the von Mangoldt
$\Lambda$-function (defined by $\Lambda(n) = \log p$ if $n = p^k$ for
a prime power $p^k$ and $\Lambda(n) = 0$ otherwise). PNT's 
are notoriously more difficult to obtain.
Indeed, there are many classes of deterministic sequences that do not
satisfy a PNT (although they are M\"obius disjoint),
for example automatic sequences, where only the logarithmic letter
densities exist \cite{adamczewski-drmota-muellner-2022}; for other
examples see \cite{ferenczi-1997,mauduit-rivat-RS-primes}.
Nevertheless it is
known that nil-sequences \cite{green-tao-nilsequences-2012} and
primitive automatic sequences \cite{muellner-2017} satisfy a PNT.  
It is also conjectured that horocyle flows and
parabolic systems satisfy a PNT
\cite{drmota-lemanczyk-mullner-rivat-2024}.

Several variations and extensions of the Sarnak Conjecture
have been proposed
(see
\cite{ferenczi-kulaga-lemanczyk-2018,drmota-lemanczyk-mullner-rivat-2024}).
The so-called
{\it polynomial Sarnak conjecture} 
(formulated by Eisner \cite{eisner-2015}) 
claims that $f(P(n))$ should be M\"obius disjoint for non-negative
integer valued polynomials $P(x)$. Actually this conjecture is false
in general even if one concentrates on minimal topological dynamical
systems \cite{Hu-Li-Sh-Ye,lemanczyk-2009}. However, it seems
to be true in several interesting instances.  Nil-sequences provide an
{\it easy} example: if $f(n)$ is a nil-sequence
and $P(x)$ is a non-negative integer valued polynomial,
then $f(P(n))$ is again a nil-sequence.
In particular we have a prime number
theorem for $f(P(n))$ in this case, too. Furthermore, since automatic
sequences along arithmetic progressions are still automatic, the
polynomial Sarnak conjecture holds in this case (although this is a
quite trivial remark). However, this becomes much more difficult for
polynomial subsequences of higher degree. A lot of effort has been
invested to study subsequences
along squares $a(n^2)$ of automatic sequences $a(n)$.
It is known that the density properties of $a(n^2)$ are the
same as that of $a(n)$
\cite{mauduit-rivat-2009,adamczewski-drmota-muellner-2022}.
Moreover, it has been shown that the sequence $t(n^2)$ has full
complexity \cite{moshe-2007}. Furthermore $t(n^2)$ is a {\it normal
sequence} \cite{drmota-mauduit-rivat-normality-TM}, that is, all
subwords appear with the expected frequency. This is a quite
remarkable property since the Thue-Morse sequence $t(n)$ (as all
automatic sequences) has at most linear complexity (and is, thus,
deterministic).  In \cite{drmota-mauduit-rivat-spiegelhofer} it was
shown that $t(n^2) = (-1)^{s_2(n^2)} = \exp(i \pi s_2(n^2))$ and more generally
$a(n^2) = \exp( 2\pi i \gamma s_q(n^2))$, where $s_q(n)$ denotes the
$q$-adic sum-of-digits function, is M\"obius disjoint:
  \begin{displaymath}
    \sum_{n\le x} \mu(n) \exp(2\pi i\gamma s_q(n^2)) = o(x)
    \qquad (x\to\infty).  
  \end{displaymath}
This also provides an
explicit example of a sequence having maximal topological entropy and
being M\"obius disjoint.

So one might conjecture that the polynomial
Sarnak conjecture holds for all automatic sequences, that is,
the sequence $a(P(n))$ is M\"obius disjoint for all automatic
sequences $a(n)$ and 
all non-negative integer valued polynomial $P(x)$:
\begin{displaymath}
  \sum_{n\le x} \mu(n) a(P(n)) = o(x)
  \qquad (x\to\infty).  
\end{displaymath}
We are far away from proving such a general statement but the results
so far support this conjecture.
We may even conjecture more than the
M\"obius disjointness for $f(P(n))$,
namely that for all primitive automatic sequences $a(n)$
and for all non-negative integer valued polynomials $P(x)$ the sequence 
$a(P(n))$ might satisfy a PNT:
\begin{displaymath}
  \sum_{n\le x} \Lambda(n) a(P(n)) = c x + o(x)
  \qquad (x\to\infty)
\end{displaymath}
for some constant $c\in \C$.
This can be also seen in the context of
the Gelfond problems
\cite{gelfond-1968}
on the (special automatic) sequences $s_q(n) \bmod m$,
that,
after some partial results
\cite{fouvry-mauduit-1996-1,fouvry-mauduit-1996-2,
  dartyge-tenenbaum-2005-Fourier},
have been solved for prime numbers \cite{mauduit-rivat-2010},
squares \cite{mauduit-rivat-2009} and cubes \cite{spiegelhofer-cubes}.

The main purpose of the present paper is to support the polynomial
version of this conjecture by proving
the first Prime Number Theorem for non-linear polynomials (actually with
$P(x) = x^2$).

\begin{theorem}\label{ThMain}
  Suppose that $\gamma \in \mathbb{Q}$ with $(q-1)\gamma \not\in \mathbb{Z}$.
  Then there exists a number $q_0 > 0$ that
  depends only on the denominator of $\gamma$ so that
  for all prime numbers $q\ge q_0$ we have
  \begin{displaymath}
    \sum_{n\le x} \Lambda(n) \exp(2\pi i\gamma s_q(n^2))
    =
    O\left( x^{1-\frac {c}{\log q}\norm{(q-1)\gamma}^2} \right)
  \end{displaymath}
  for some absolute positive constant $c$.
\end{theorem}
In particular, by setting $\gamma = 1/m$ this implies that the automatic
sequence $s_q(n) \bmod m$ satisfies a PNT along squares
for sufficiently large prime bases $q$.
This means that 
\begin{equation}\label{eq:MainResultForPrimes}
  \# \{ p\in \mathbb{P}:  p\le x, \  s_q(p^2) \equiv a \bmod m \}
  =
  \frac{\pi(x)}{m}
  +
  O\left( x^{1-\frac{c'}{\log q}\norm{(q-1)\gamma}^2} \right)
\end{equation}
holds for all integers $a$ and for all $m$ that are coprime to $q-1$ 
(for sufficiently large prime bases $q$),
where $c'$ is an absolute positive constant.

\medskip
\subsection*{Proof strategy and plan of the paper}~

In order to prove our results, the challenge is to overcome the
difficulties inherant to the prime numbers combined with the
difficulties arising from the squares.

Applying a combinatorial identity (e.g. Vaughan's identity)
the problem is reduced to type I and type II sums
involving a function still containing squares.

The problem is then again reduced to the estimate of double exponential sums
and Fourier analytic methods.
This leads to the analysis of so called diagonal and non diagonal terms.

The main difficulty arise 
from the diagonal terms which produce
diophantine constraints than cannot be directly handled.
In the type II sums (Section~\ref{section:typeIIsums}),
the new idea is to reinterpret these sums by Fourier inversion, which
leads to convolutions of Vaaler kernels than need a proper study.
Exploiting their properties we can start a second Fourier analysis
which permits us to conclude.
In the type I sums (Section~\ref{section:typeIsums}),
we combine a covering argument with the large sieve.

Due to the technical difficulty of the proof, we need a wide
collection of preliminary results on approximation, convolution,
Fourier analysis of $q$-multiplicative functions and exponential sums
(Section~\ref{section:preliminaries}).

\medskip
\subsection*{Notation and basic definitions}~

We denote by $\N$ the set of non negative integers, 
by $\U$ the set of complex numbers of modulus~$1$,
by $\mathbb{P}$ the set of prime numbers.

For $x\in\R$ we denote 
by $\norm{x}$ the distance of $x$ to the nearest integer,
and we set $\e(x) = \exp(2i\pi x)$.
For $Q>0$ and $x\in\R$, we set $\e_Q(x) = \exp(2i\pi x/Q)$.
For $\ell\in\Z$ and $x\in\R$, we set $\e^\ell(x)=\e(\ell x)$.
For $z\in\C$, we recall that the arithmetic functions
$\sigma_z(n) = \sum_{d\dv n} d^z$ (sum of powers of divisors), and
\begin{math}
  \tau(n) = \sigma_0(n) = \sum_{d\dv n} 1
\end{math}
(number of divisors),
are multiplicative.
We denote by $\omega(n)$ the number of distinct prime factors of $n$.
If $f$ and $g$ are two functions with $g$ taking strictly positive
values such that $f/g$ is bounded, we write $f\ll g$ (or $f=O(g)$).

In all this paper $q$ denotes an integer greater or equal to $2$ and
for any positive integer $n$, 
\begin{equation}
  \label{eq:representation-in-base-q}
  n = \sum_{j\geq 0} \varepsilon_{j}(n) \, q^j
  \mbox{ with }
  \varepsilon_{j}(n)\in\{0,\ldots,q-1\} \mbox{ for all } j\in\N
\end{equation}
is the representation of $n$ in base $q$.

\medskip
\subsection*{$q$-additive and $q$-multiplicative functions}~

The notion of $q$-additive function has been introduced independently
by Bellman and Shapiro in \cite{bellman-shapiro-1948} and by Gelfond in
\cite{gelfond-1968}.
\begin{definition}
  A function $h:\N \to \R$ is $q$-additive 
  (resp. strongly $q$-additive) 
  if for all $(a,b)\in\N\times\{0,\ldots,q-1\}$,
  we have 
  \begin{displaymath}
    h(aq+b) = h(aq)+h(b)
  \end{displaymath}
  (resp. $h(aq+b) = h(a)+h(b)$).
\end{definition}
It follows that any $q$-additive function $h$ verifies $h(0)=0$. 
If $h$ is a strongly $q$-additive function then $h$ is uniquely
determined by the values $h(1)$,\ldots, $h(q-1)$ and for any positive
integer $n$ written in base $q$ as \eqref{eq:representation-in-base-q},
we have
\begin{displaymath}
  h \left( \sum_{j\geq 0} \varepsilon_{j}(n) \, q^j \right)
  =
  \sum_{j\geq 0} h(\varepsilon_{j}(n))
  .
\end{displaymath}
The most classical example of $q$-additive function is the $q$-ary
sum-of-digits function defined by 
\begin{math}
  \s_q(n) = \sum_{j\geq 0} \varepsilon_{j}(n)  .
\end{math}

In a similar way we can define the notions of $q$-multiplicative
function and strongly $q$-multiplicative function:
\begin{definition}\label{definition:q-multiplicative-function}
  A function $f:\N \to \U$ is $q$-multiplicative
  (resp. strongly $q$-multiplicative) 
  if for all $(a,b)\in\N\times\{0,\ldots,q-1\}$,
  we have 
  \begin{math}
    f(aq+b) = f(aq) \, f(b)
  \end{math}
  (resp. $f(aq+b) = f(a) \, f(b)$).
\end{definition}
If $h$ is a $q$-additive (resp. strongly $q$-additive) function then
$f = \e(h)$ is $q$-multiplicative (resp. strongly $q$-multiplicative).
Conversely if $f = \e(h)$ is a $q$-multiplicative (resp. strongly
$q$-multiplicative) function from $\N$ to $\U$ then $h$ is
$q$-additive (resp. strongly $q$-additive) modulo $1$.

\begin{definition}\label{definition:proper}
  A strongly $q$-multiplicative function is called {\bf proper} if 
  it is not of the form $f(n) = \e(\gamma n)$ 
  with $(q-1)\gamma \in \mathbb{Z}$.
\end{definition}

In particular, if $f(n) = \e(\gamma s_q(n))$ then $f(n)$ is proper 
if and only if $(q-1)\gamma \not\in \mathbb{Z}$.

\section{Statement of the results}

We suppose that $f(n)$ is a proper strongly $q$-multiplicative function.
For such a function we define the Fourier transform $F_\lambda(t)$ by
\begin{equation}\label{eq:definition-F-lambda}
  F_\lambda(t)
  =
  \frac 1{q^\lambda}
  \sum_{0\le u < q^\lambda} f(u) \e\left(  \frac {-tu}{q^\lambda} \right)
\end{equation}
(see Section~\ref{subsecFourier}) and constants $c(f)$ and $\eta(f)$ by
\begin{equation}\label{eq:definition-c-f}
  c(f)
  =
  \frac{-\log \left( \max_{t\in\R} |F_1(t) F_1(qt)| \right)}{2 \log q}
\end{equation}
and 
\begin{displaymath}\label{eq:definition-eta-f}
  \eta(f) = \frac 1{\log q} 
  \log 
  \left( 
    \max_{0\le t\le 1} \sum_{0\le r < q} 
    \left|F_1\left( q t + r \right) \right|
  \right)
  .
\end{displaymath}
We note that $0< c(f)\le \eta(f)$,
see sections
\ref{subsection:Fourier-Transform-q-multiplicative-function}
and \ref{sec:L1estimates}.

The main result of the paper is the following theorem:

\begin{theorem}\label{Thmainresult}
  Let  $q\ge 2$ be a prime number and
  suppose that $f(n)$ is a proper strongly $q$-multiplicative function
  such that $\eta(f) \le 1/2000$.

  Then, uniformly for $\theta \in \mathbb{R}$, we have 
  \begin{displaymath}
    \sum_{n\le x} 
    \Lambda(n) f(n^2) \e(\theta n) = 
    O\left( (\log x)^{\frac 52}
      x^{1- (d_1 c(f) - d_2 \eta(f)^2) }
    \right)
  \end{displaymath}
  for some explicit absolute constants $d_1,d_2 > 0$.
\end{theorem}

As the following proof shows we can choose these constants $d_1,d_2$
explicitly:
\begin{displaymath}
%  d_1 = \frac{8}{51}, \ d_2 = \frac 15, \
  d_1 = \frac 1{40}, \ d_2 = \frac{16}5.
\end{displaymath}

In the special case of $f(n) = \e(\gamma s_q(n))$ (where
$(q-1)\gamma \in \mathbb{Z}$) we know proper bounds for
$c_q = c(\e(\gamma s_q(n)))$ and $\eta_q = \eta(\e(\gamma s_q(n)))$
(see Sections~\ref{subsecFourier}):
\begin{displaymath}
c_q \ge
  \frac{\pi^2(q-1)}{12(q+1)\log q} \norm{(q-1)\gamma}^2
\end{displaymath}
and by Lemme 14 of \cite{mauduit-rivat-2010} we have
\begin{displaymath}
\eta_q \le \frac 1{\log q} 
\left(
\frac 2 {q \sin \frac{\pi}{2q}} + \frac 2\pi \log \frac{2q}{\pi}
\right).
\end{displaymath}
This means that
\begin{displaymath}
c_q \ge d_3 \frac{\norm{(q-1)\gamma}^2}{\log q}
\end{displaymath}
and
\begin{displaymath}
\eta_q \le d_4 \frac{\log\log q}{\log q}
\end{displaymath}
for properly chosen constants $d_3, d_4 > 0$.
Thus,
if $q$ is sufficiently large and if we have the condition
\begin{equation}\label{eqqcondition}
  \norm{(q-1)\gamma}^2
  \ge
  \frac{2 d_2 d_4^2}{d_1d_3} \frac{(\log\log q)^2}{\log q},
\end{equation}
we certainly have
\begin{displaymath}
  d_1 c_q - d_2 \eta_q^2
  \ge \frac{d_1}2 c_q
  \ge \frac{d_1d_3}2 \frac{\norm{(q-1)\gamma}^2}{\log q}
  .
\end{displaymath}
This condition is very likely to hold if the prime number $q$ is
sufficiently large since 
$\norm{(q-1)\gamma}$ is usually not too small. In particular if 
$\gamma = a/b$ is a rational number
with $(q-1)\gamma \not \in \mathbb{Z}$ then 
we have $\norm{(q-1)\gamma} \ge 1/b$ so that 
(\ref{eqqcondition}) is certainly true for sufficiently large
prime numbers $q \ge q_0$ 
(depending on the denominator $b$ of $\gamma$).
Thus, Theorem~\ref{ThMain} is a direct corollary
of Theorem~\ref{Thmainresult}.
Moreover, \eqref{eq:MainResultForPrimes} follows from
Theorem~\ref{ThMain} by partial summation.

\section{Preliminaries}\label{section:preliminaries}

In this section we collect several preliminary results that will be
used for the proof of Theorem~\ref{Thmainresult}, mainly on
approximation, convolution, Fourier analysis of $q$-multiplicative
functions, and on exponential sums.

\medskip
\subsection{Detection of digits}~

For $a\in\Z$ and $\kappa\in\N$ we denote by $\rep_\kappa(a)$ 
the unique integer $r\in\{0,\ldots,q^{\kappa}-1\}$
such that $a\equiv r \bmod q^{\kappa}$. 
More generally for integers $0\leq \kappa_1 \leq \kappa_2$ 
we denote by $\rep_{\kappa_1,\kappa_2}(a)$
the unique integer $u\in\{0,\ldots,q^{\kappa_2-\kappa_1}-1\}$
such that $a = k q^{\kappa_2} + u q^{\kappa_1} +v$ for some
$v\in\{0,\ldots,q^{\kappa_1}-1\}$ and $k\in\Z$.
We notice that we have 
\begin{math}
  \rep_{\kappa_1,\kappa_2}(a)
  =
  \floor{\frac{\rep_{\kappa_2}(a)}{q^{\kappa_1}}}
\end{math}
and for any $u\in\{0,\ldots,q^{\kappa_2-\kappa_1}-1 \}$, 
\begin{equation}\label{eq:digit-detection}
  \rep_{\kappa_1,\kappa_2}(a) = u
  \Longleftrightarrow
  \frac{a}{q^{\kappa_2}} 
  \in 
  \left[ 
    \frac{u}{q^{\kappa_2-\kappa_1}}, \frac{u+1}{q^{\kappa_2-\kappa_1}}
  \right)
  + \Z.
\end{equation}
For $a\geq 0$, $\rep_\kappa(a)$ is the integer obtained from the
$\kappa$ least significant digits of $a$, while 
$\rep_{\kappa_1,\kappa_2}(a)$ is the integer obtained
using the digits of $a$ of index in $\{\kappa_1,\ldots,\kappa_2-1\}$.

\medskip
\subsection{Beurling--Selberg approximation}~

For $\alpha\in\R$ with $0\leq \alpha<1$ we denote by $\chi^*_\alpha$ the
characteristic function of the interval
$[-\alpha/2,\alpha/2[$ modulo $1$:
\begin{equation}
  \label{eq:definition-chi-star}
  \chi^*_\alpha(x)
  =
  \one_{[-\alpha/2,\alpha/2[+\Z}(x)
  =
  \floor{x+\frac{\alpha}{2}}-\floor{x-\frac{\alpha}{2}}
  ,
\end{equation}
with real valued Fourier coefficients
\begin{equation}
  \label{eq:definition-fourier-chi-star}
  \fourier{\chi^*_\alpha}(0) = \alpha,
  \quad
  \forall h\neq 0,\
  \fourier{\chi^*_\alpha}(h)
  =
  \int_{-1/2}^{1/2} \chi^*_\alpha(t) \e(-ht) dt
  =
  \frac{\sin \pi h \alpha}{\pi h}
  .
\end{equation}
Applying Theorem 19 of Vaaler \cite{vaaler-1985},
for any integer $H\geq 1$, there exist real valued
trigonometric polynomials $\chi_{\alpha,H}^*(x)$ and $B_{\alpha,H}^*(x)$ 
such that for all $x\in\R$
\begin{equation}\label{eq:vaaler-approximation-chi-star}
  \abs{ \chi^*_\alpha(x) - \chi_{\alpha,H}^*(x) }
  \leq
  B_{\alpha,H}^*(x)
  ,
\end{equation}
where
\begin{equation}\label{eq:definition-chi-H-star}
  \chi_{\alpha,H}^*(x)
  =
  \sum_{\abs{h}\leq H}  \fourier{\chi_{\alpha,H}^*}(h) \e(h x),\
\end{equation}
\begin{align}\label{eq:definition-B-H-star}
  B_{\alpha,H}^*(x)
  &=
  \frac{
    \sin^2\left(\pi (H+1) \left(x-\frac{\alpha}{2}\right)\right)
  }{2(H+1)^2 \sin^2 \left(\pi \left(x-\frac{\alpha}{2}\right)\right)}
  +
    \frac{
      \sin^2\left(\pi (H+1) \left(x+\frac{\alpha}{2}\right)\right)
    }{2 (H+1)^2 \sin^2 \left(\pi \left(x+\frac{\alpha}{2}\right)\right)}
  \\ \nonumber
  &=
  \sum_{\abs{h}\leq H} \fourier{B_{\alpha,H}^*}(h) \e(h x)
\end{align}
with real valued Fourier coefficients
\begin{equation}\label{eq:vaaler-coef-chi-star-H}
  \fourier{\chi_{\alpha,H}^*}(h) 
  = 
  \begin{cases}
    \fourier{\chi^*_\alpha}(h)
    \left( 
      \pi \tfrac{\abs{h}}{H+1} \left(1-\tfrac{\abs{h}}{H+1}\right) 
      \cot \pi \tfrac{\abs{h}}{H+1}
      +
      \tfrac{\abs{h}}{H+1}
    \right)
    & \text{ for } \abs{h} \leq H,
    \\
    0
    & \text{ for } \abs{h} \geq H+1,
  \end{cases}
\end{equation}
and
\begin{equation}\label{eq:vaaler-coef-B-star-H}
  \fourier{B_{\alpha,H}^*}(h)
  = 
  \begin{cases}
    \tfrac{1}{H+1} \left(1-\tfrac{\abs{h}}{H+1}\right) 
    \cos(\pi h \alpha)
    & \text{ for } \abs{h} \leq H,
    \\
    0
    & \text{ for } \abs{h} \geq H+1,
  \end{cases}
\end{equation}
satisfying
\begin{equation}\label{eq:vaaler-coef-majoration-star}
  \fourier{\chi_{\alpha,H}^*}(0) = \alpha,\ 
  \abs{ \fourier{\chi_{\alpha,H}^*}(h) }
  \leq
  \abs{ \fourier{\chi_{\alpha}^*}(h) }
  \leq
  \min\left(\alpha,\tfrac{1}{\pi\abs{h}}\right),\
  \abs{  \fourier{B_{\alpha,H}^*}(h)}
  \leq
  \fourier{B_{\alpha,H}^*}(0) = \tfrac{1}{H+1}.
\end{equation}

For $\alpha\in\R$ with $0\leq \alpha<1$ we denote by $\chi_\alpha$ the
characteristic function of the interval
$[0,\alpha[$ modulo $1$:
\begin{equation}\label{eq:definition-chi}
  \chi_\alpha(x)
  =
  \one_{[0,\alpha)+\Z} (x)
  =
  \floor{x} - \floor{x-\alpha}
  =
  \chi^*_\alpha\left(x-\frac{\alpha}{2}\right)
  .
\end{equation}
For any integer $H\geq 1$,
the real valued trigonometric polynomials
$\chi_{\alpha,H}(x) = \chi_{\alpha,H}^*\left(x-\frac{\alpha}{2}\right)$
and
$B_{\alpha,H}(x) = B_{\alpha,H}^*\left(x-\frac{\alpha}{2}\right)$ 
are such that for all $x\in\R$
\begin{equation}\label{eq:vaaler-approximation-chi}
  \abs{ \chi_\alpha(x) - \chi_{\alpha,H}(x) }
  \leq
  B_{\alpha,H}(x),
\end{equation}
where
\begin{equation}\label{eq:definition-A-B}
  \chi_{\alpha,H}(x)
  =
  \sum_{\abs{h}\leq H}  \fourier{\chi_{\alpha,H}}(h)  \e(h x),\
  B_{\alpha,H}(x) = \sum_{\abs{h}\leq H}   \fourier{B_{\alpha,H}}(h) \e(h x)
\end{equation}
with Fourier coefficients
\begin{equation}\label{eq:vaaler-coef-chi-B}
  \fourier{\chi_{\alpha,H}}(h) 
  =
  \fourier{\chi_{\alpha,H}^*}(h)
  \e\left(\tfrac{-h\alpha}{2}\right)
  ,\quad
  \fourier{B_{\alpha,H}}(h) = 
  \fourier{B_{\alpha,H}^*}(h)
  \e\left(\tfrac{-h\alpha}{2} \right),
\end{equation}
satisfying
\begin{equation}\label{eq:vaaler-coef-majoration}
  \fourier{\chi_{\alpha,H}}(0)  = \alpha,\ 
  \abs{ \fourier{\chi_{\alpha,H}}(h) }
  \leq
  \abs{ \fourier{\chi_{\alpha}}(h) }
  \leq
  \min\left(\alpha,\tfrac{1}{\pi\abs{h}}\right),\
  \abs{  \fourier{B_{\alpha,H}}(h)}
  \leq
  \fourier{B_{\alpha,H}}(0) = \tfrac{1}{H+1}.
\end{equation}

\medskip
\subsection{Convolutions and further properties of Beurling--Selberg
  approximation}~

For $0<\alpha\leq 1/2$, the function $\chi_\alpha^* * \chi_\alpha^*$ 
satisfies 
\begin{displaymath}
  \forall x\in\R,\
  \chi_\alpha^* * \chi_\alpha^*(x)
  =
  \int_{-1/2}^{1/2} \chi_\alpha^*(x-t)\chi_\alpha^*(t) dt
  =
  \alpha \max\left(1-\frac{\norm{x}}{\alpha},0\right)
  ,
\end{displaymath}
so that
\begin{equation}\label{eq:chi-star-convolution-chi-star-of-zero}
  \chi_\alpha^* * \chi_\alpha^*(0)
  =
  \int_{-1/2}^{1/2} {\chi_\alpha^*}^2(x) dx
  =\alpha
\end{equation}
and for $\norm{x} \geq \alpha$, 
\begin{equation}\label{eq:chi-star-convolution-chi-star-non-zero}
  \chi_\alpha^* * \chi_\alpha^*\left(x\right)  = 0
  .
\end{equation}
The Fourier expansion of
\begin{math}
  \chi_{\alpha}^* * \chi_{\alpha}^*
\end{math}
is
\begin{equation}
  \label{eq:chi-star-convolution-chi-star-fourier-expansion}
  \chi_{\alpha}^* * \chi_{\alpha}^*\left(x\right)
  =
  \sum_{h\in\Z}
  \abs{\fourier{\chi_{\alpha}^*}(h)}^2
  \e\left( h x \right)
  .
\end{equation}
\begin{lemma}\label{lemma:chi-convolution-chi-average}
  For $U\in\N$ with $U \geq 2$, $\alpha=1/U$
  and $a\in\Z$ we have
  \begin{displaymath}
    \sum_{k\in\Z}
    \abs{\fourier{\chi_{\alpha}^*}\left(kU+a\right)}^2
    =
    \frac{1}{U^2} 
    .
  \end{displaymath}    
\end{lemma}
\begin{proof}
  We have
  \begin{align*}
    \sum_{k\in\Z}
    \abs{\fourier{\chi_{\alpha}^*}\left(kU+a\right)}^2
    &=
    \sum_{h\in\Z}
    \abs{\fourier{\chi_{\alpha}^*}(h)}^2
    \frac{1}{U} \sum_{0\leq u < U} \e\left(\frac{(h-a)u}{U}\right)
    \\
    &=
      \frac{1}{U} \sum_{0\leq u < U}
      \e\left(\frac{-a u}{U}\right)
      \sum_{h\in\Z}
      \abs{\fourier{\chi_{\alpha}^*}(h)}^2
      \e\left(\frac{h u}{U}\right)
    \\
    &=
      \frac{1}{U} \sum_{0\leq u < U}
      \e\left(\frac{-a u}{U}\right)
      \chi_{\alpha}^* * \chi_{\alpha}^*\left(\frac{u}{U}\right)
      .
  \end{align*}
  By \eqref{eq:chi-star-convolution-chi-star-non-zero},
  for $u\in\{1,\ldots,U-1\}$
  we have
  \begin{math}
    \chi_{\alpha}^* * \chi_{\alpha}^*\left(\frac{u}{U}\right)=0,
  \end{math}
  hence the sum is reduced to $u=0$,
  and by \eqref{eq:chi-star-convolution-chi-star-of-zero}
  we get the result.
\end{proof}

In the context of a multidimensional approximation
by~\eqref{eq:vaaler-approximation-chi-star},
products of $\fourier{\chi_{\alpha,H}^*}$ will arise,
which may lead to convolutions of $\chi_{\alpha,H}^*$ functions.
Since $\chi_{\alpha,H}^*$ approximates
the characteristic function $\chi_\alpha^*$
of the interval $[-\alpha/2,\alpha/2)$ modulo $1$,
it follows that
\begin{math}
  \chi_{\alpha,H}^* * \chi_{\alpha,H}^*
\end{math}
is expected to be close to
\begin{math}
  \chi_\alpha^* * \chi_\alpha^*
\end{math}
and therefore more or less concentrated around the origin.
We will show this property in
Lemma~\ref{lemma:chi_H-convolution-chi_H}. 
Up to an admissible error term,
this permits to take advantage of a compact support alternatively on
both sides of Fourier.

\begin{lemma}\label{lemma:chi-convolution-B_H}
  For $(U,H)\in\N^2$ with $2 \leq U \leq H+1$
  and $\alpha=1/U$, we have
  \begin{displaymath}
    \sum_{0\leq u < U}
    \chi_\alpha^* * B_{\alpha,H}^*\left(\frac{u}{U}\right)
    =
    \frac{1}{H+1} 
    .
  \end{displaymath}  
\end{lemma}
\begin{proof}
  By \eqref{eq:definition-chi-star} and
  \eqref{eq:vaaler-coef-majoration-star} we have
  \begin{align*}
    \sum_{0\leq u < U}
    \chi_\alpha^* * B_{\alpha,H}^*\left(\frac{u}{U}\right)
    &
      =
      \sum_{0\leq u < U}
      \int_{-\frac{\alpha}{2}}^{\frac{\alpha}{2}}
      B_{\alpha,H}^*\left( \frac{u}{U}-t \right) dt
      =
      \sum_{0\leq u < U}
      \int_{\frac{u-\frac12}{U}}^{\frac{u+\frac12}{U}}
      B_{\alpha,H}^*\left( t' \right) dt'
    \\
    &
      =
      \int_0^1 B_{\alpha,H}^*\left( t' \right) dt'
      =
      \fourier{B_{\alpha,H}^*}(0)
      = \frac{1}{H+1} 
      .
  \end{align*}
  Alternatively we can prove this equality by using the
  Fourier expansion. By~\eqref{eq:definition-fourier-chi-star} we have
  $\fourier{\chi_\alpha^*}(0) = \alpha = 1/U$ and
  $\fourier{\chi_\alpha^*}(k U) = 0$ for $k\neq 0$:
  \begin{align*}
    \sum_{0\leq u < U}
    \chi_\alpha^* * B_{\alpha,H}^* \left(\frac{u}{U}\right)
    &
      =
      \sum_{0\leq u < U}
      \sum_{\abs{h}\leq H}
      \fourier{\chi_\alpha^*}(h) \, \fourier{B_{\alpha,H}^*}(h)
      \e\left( \frac{h u}{U} \right)
    \\
    &
      =
      U
      \sum_{\abs{h}\leq H}
      \one_{h\equiv 0 \bmod U}
      \fourier{\chi_\alpha^*}(h) \, \fourier{B_{\alpha,H}^*}(h)
    \\
    &
      =
      U
      \fourier{\chi_\alpha^*}(0) \, \fourier{B_{\alpha,H}^*}(0)
      +
      U
      \sum_{1\leq \abs{k} < K}
      \fourier{\chi_\alpha^*}(k U) \,
      \fourier{B_{\alpha,H}^*}(k  U)
    \\
    &
      = \fourier{B_{\alpha,H}^*}(0)
      = \frac{1}{H+1} 
      .
  \end{align*}
\end{proof}

\begin{lemma}\label{lemma:B_H-convolution-B_H}
  For $(U,H)\in\N^2$ with $2 \leq U \leq H+1$
  and $\alpha=1/U$, we have
  \begin{displaymath}
    \sum_{0\leq u < U}
    B_{\alpha,H}^* * B_{\alpha,H}^* \left(\frac{u}{U}\right)
    \leq
    \frac{1}{H+1}
    .
  \end{displaymath}
\end{lemma}
\begin{proof}
  Remembering by \eqref{eq:vaaler-coef-B-star-H} that
  \begin{displaymath}
    \fourier{B_{\alpha,H}^*}(h)
    = 
    \tfrac{1}{H+1} \left(1-\tfrac{\abs{h}}{H+1}\right) 
    \cos(\pi h \alpha)
    ,
  \end{displaymath}
  writing $H+1 = K U + R$
  with integers $K\geq 1$ and $0\leq R \leq U-1$,
  we have
  \begin{align*}
    \sum_{0\leq u < U}
    B_{\alpha,H}^* * B_{\alpha,H}^* \left(\frac{u}{U}\right)
    &
      =
      \sum_{0\leq u < U}
      \sum_{\abs{h}\leq H+1}
      \left(\fourier{B_{\alpha,H}^*}(h)\right)^2
      \e\left( \frac{h u}{U} \right)
    \\
    &
      =
      \frac{U}{(H+1)^2}
      \sum_{\abs{h}\leq H+1}
      \one_{h\equiv 0 \bmod U}
      \left(1-\tfrac{\abs{h}}{H+1}\right)^2
      \cos^2\left(\pi \frac{h}{U} \right)  
    \\
    &
      =
      \frac{U}{(H+1)^2}
      \sum_{\abs{k} \leq K}
      \left(1-\tfrac{\abs{k} U}{H+1}\right)^2
      \cos^2\left(\pi \abs{k} \right)  
    \\
    &
      =
      \frac{U}{(H+1)^4}
      \sum_{\abs{k} \leq K} \left(H+1-\abs{k} U\right)^2  
  \end{align*}
  and 
  \begin{align*}
    \sum_{\abs{k} \leq K}
    \left(H+1-\abs{k} U\right)^2  
    &=
      \sum_{\abs{k} \leq K}
      \left(
      R^2
      +2 R (K-\abs{k}) U
      + (K-\abs{k})^2 U^2
      \right)
    \\
    &=
      (2K+1) R^2
      +2 R K^2 U
      + \frac{2K^3+K}{3} U^2
  \end{align*}
  so that
  \begin{multline*}
    U \sum_{\abs{k} \leq K}
    \left(H+1-\abs{k} U\right)^2  
    =
    R^2 U
    + 2 R^2 K U
    + 2 R K^2 U^2
    + \frac{2K^3+K}{3} U^3
    \\
    =
    \left( K U +R \right)^3
    - R^2 \left(R - U \right)
    - R^2 K U
    - R K^2 U^2
    - \frac{K}{3} (K^2-1) U^3 
    ,
  \end{multline*}
  and finally
  \begin{displaymath}
    \sum_{0\leq u < U}
    B_{\alpha,H}^* * B_{\alpha,H}^* \left(\frac{u}{U}\right)
    \leq
    \frac{\left( K U +R \right)^3}{(H+1)^4}
    =
    \frac{1}{H+1}
    ,
  \end{displaymath}
  as expected.
\end{proof}

\begin{lemma}\label{lemma:chi_H-convolution-chi_H}
  For $(U,H)\in\N^2$ with $2 \leq U \leq H+1$
  and $\alpha=1/U$, we have
  \begin{multline*}
    \sum_{0\leq u < U}
    \abs{
      \chi_{\alpha,H}^* * \chi_{\alpha,H}^*\left(\frac{u}{U}\right)
      -
      \chi_\alpha^* * \chi_\alpha^*\left(\frac{u}{U}\right)
    }
    \\
    =
    \abs{ \chi_{\alpha,H}^* * \chi_{\alpha,H}^*(0) - \alpha }
    +
    \sum_{1\le u < U}
    \abs{  \chi_{\alpha,H}^* * \chi_{\alpha,H}^*\left(\frac{u}{U}\right)}
    \leq \frac{3}{H+1} 
    .
  \end{multline*}
\end{lemma}
\begin{remark}
  This shows that
  \begin{math}
    \chi_{\alpha,H}^* * \chi_{\alpha,H}^*(0) 
  \end{math}
  is close to $\alpha = 1/U$
  and that the values of
  $\chi_{\alpha,H}^* * \chi_{\alpha,H}^*(u/U)$ are (almost) negligible
  for $u\in\{1,\ldots,U-1\}$.
\end{remark}
\begin{proof}
  We write 
  \begin{multline*}
    \chi_\alpha^* * \chi_\alpha^* - \chi_{\alpha,H}^** \chi_{\alpha,H}^*
    \\
    =
    \left( \chi_\alpha^* -\chi_{\alpha,H}^*\right) * \chi_\alpha^*
    +
    \chi_\alpha^** \left( \chi_\alpha^* -\chi_{\alpha,H}^*\right) 
    -
    \left( \chi_\alpha^* -\chi_{\alpha,H}^*\right)
    * \left( \chi_\alpha^* -\chi_{\alpha,H}^*\right) 
    ,
  \end{multline*}
  and by \eqref{eq:vaaler-approximation-chi-star} we have
  \begin{equation}\label{eq:approximation-chi-star-convolution-chi-star}
    \abs{
      \chi_\alpha^* * \chi_\alpha^* - \chi_{\alpha,H}^**  \chi_{\alpha,H}^*
    }
    \leq
    2 \, \chi_\alpha^* * B_{\alpha,H}^* + B_{\alpha,H}^* * B_{\alpha,H}^*
    .
  \end{equation}
  With \eqref{eq:chi-star-convolution-chi-star-of-zero},
  \eqref{eq:chi-star-convolution-chi-star-non-zero}
  and \eqref{eq:approximation-chi-star-convolution-chi-star},
  Lemma~\ref{lemma:chi-convolution-B_H}
  and Lemma~\ref{lemma:B_H-convolution-B_H}
  we deduce
  \begin{multline*}
    \abs{  \chi_{\alpha,H}^* * \chi_{\alpha,H}^*(0) - \alpha}
    +
    \sum_{1\leq u < U}
    \abs{ \chi_{\alpha,H}^* * \chi_{\alpha,H}^*\left(\frac{u}{U}\right)}
    \\
    =
    \sum_{0\leq u < U}
    \abs{
      \chi_{\alpha,H}^* * \chi_{\alpha,H}^*\left(\frac{u}{U}\right)
      -
      \chi_\alpha^* * \chi_\alpha^*\left(\frac{u}{U}\right)
    }
    \\
    \leq
    2
    \sum_{0\leq u < U}
    \chi_\alpha^* * B_{\alpha,H}^*\left(\frac{u}{U}\right)
    +
    \sum_{0\leq u < U}
    B_{\alpha,H}^* * B_{\alpha,H}^* \left(\frac{u}{U}\right) 
    \leq
    \frac{3}{H+1}
    ,
  \end{multline*}
  as expected.
\end{proof}

For $0<\alpha\leq 1/2$ and $\ell \in \Z$, 
we also need to consider the function
\begin{math}
  \chi_\alpha^* * (\chi_\alpha^* \e^\ell).
\end{math}
Since $\chi_\alpha^*$ is even and
\begin{math}
  \left(\chi_\alpha^*\right)^2 = \chi_\alpha^*
  ,
\end{math}
we have
\begin{displaymath}
  \chi_\alpha^* * (\chi_\alpha^* \e^\ell)(0)
  = 
  \int_{-1/2}^{1/2} \chi_\alpha^*(0-t)\chi_\alpha^*(t) \e^\ell(t)\,  dt 
  = 
  \int_{-1/2}^{1/2} \chi_\alpha^*(t) \e(\ell t)\,  dt 
  ,
\end{displaymath}
so that, by~\eqref{eq:definition-fourier-chi-star},
\begin{equation}\label{eq:chi-star-convolution-chi-star-of-zero-2}
  \chi_\alpha^* * (\chi_\alpha^* \e^\ell)(0)
  =
  \fourier{\chi_\alpha^*}(-\ell)
  =
  \fourier{\chi_\alpha^*}(\ell)
  =
  \begin{cases}
    \frac{\sin(\pi \alpha \ell)}{\pi \ell} \text{ if } \ell \neq 0,\\
    \alpha \text{ if } \ell = 0.
  \end{cases}
\end{equation}
We have
\begin{displaymath}
  \abs{\chi_\alpha^* * (\chi_\alpha^* \e^\ell)}
  \leq
  \abs{\chi_\alpha^*} *   \abs{\chi_\alpha^*}
  =
  \chi_\alpha^* * \chi_\alpha^*
  ,
\end{displaymath}
hence, by~\eqref{eq:chi-star-convolution-chi-star-non-zero},
for $\norm{x} \geq \alpha$, 
\begin{equation}\label{eq:chi-star-convolution-chi-star-non-zero-2}
  \chi_\alpha^* * (\chi_\alpha^* \e^\ell) \left(x\right)  = 0
  .
\end{equation}
For $h\in\Z$, we have
\begin{displaymath}
  \fourier{\chi_\alpha^* \e^\ell}(h)
  =
  \int_{-1/2}^{1/2} \chi_\alpha^*(t) \e^\ell(t) \e(-h t)\,  dt
  =
  \int_{-1/2}^{1/2} \chi_\alpha^*(t) \e((\ell-h) t)\,  dt
  =
  \fourier{\chi_{\alpha}^*}(h-\ell).
\end{displaymath}
The Fourier expansion of
\begin{math}
  \chi_\alpha^* * (\chi_\alpha^* \e^\ell)
\end{math}
is
\begin{equation}
  \label{eq:chi-star-convolution-chi-star-fourier-expansion-2}
  \chi_\alpha^* * (\chi_\alpha^* \e^\ell) (x)
  =
  \sum_{h\in\Z}
  \fourier{\chi_{\alpha}^*}(h) \fourier{\chi_{\alpha}^*}(h-\ell) 
  \e( h x )
  .
\end{equation}

We also need an extension of Lemma~\ref{lemma:chi_H-convolution-chi_H}.

\begin{lemma}\label{lemma:chi_H-convolution-chi_H-2}
  For $(U,H)\in\N^2$ with $2 \leq U \leq H+1$
  and $\alpha=1/U$, we have
  \begin{multline*}
    \sum_{0\leq u < U}
    \abs{
      \chi_{\alpha,H}^* * (\chi_{\alpha,H}^*\e^\ell)
     \left(\frac{u}{U}\right)
      -
      \chi_{\alpha}^* * (\chi_{\alpha}^*\e^\ell)
    \left(\frac{u}{U}\right)
    }
    \\
    =
    \abs{ 
      \chi_{\alpha,H}^* * (\chi_{\alpha,H}^*\e^\ell)(0)
      - \fourier{\chi_\alpha^*}(\ell) }
    +
    \sum_{1\le u < U}
    \abs{  
      \chi_{\alpha,H}^* * (\chi_{\alpha,H}^*\e^\ell)
      \left(\frac{u}{U}\right)}
    \leq \frac{3}{H+1} 
    .
  \end{multline*}
\end{lemma}
\begin{proof}
  The equality follows
  from \eqref{eq:chi-star-convolution-chi-star-of-zero-2}
  and \eqref{eq:chi-star-convolution-chi-star-non-zero-2}.
  As in the proof of Lemma~\ref{lemma:chi_H-convolution-chi_H} we write 
  \begin{multline*}
    \chi_\alpha^* * (\chi_\alpha^*\e^\ell)
     - 
    \chi_{\alpha,H}^* * (\chi_{\alpha,H}^*\e^\ell) 
    \\
    =
    \left( \chi_\alpha^* -\chi_{\alpha,H}^*\right) * (\chi_\alpha^*\e^\ell) 
    +
    \chi_\alpha^* * \left( (\chi_\alpha^* -\chi_{\alpha,H}^*)\e^\ell \right) 
    -
    \left( \chi_\alpha^* -\chi_{\alpha,H}^*\right)
    * \left( (\chi_\alpha^* -\chi_{\alpha,H}^*)\e^\ell \right) 
    .
  \end{multline*}
  Since $\abs{\chi_\alpha^* -\chi_{\alpha,H}^*} \le B_{\alpha,H}^*$ and
  and $\chi_\alpha^*$ and $B_{\alpha,H}^*$ are non-negative we have
  \begin{align*}
  \abs{\left( \chi_\alpha^* -\chi_{\alpha,H}^*\right) * (\chi_\alpha^*\e^\ell)}
    \le
    B_{\alpha,H}^* * \chi_\alpha^* 
    =
    \chi_\alpha^* * B_{\alpha,H}^* 
  \end{align*}
  and similarly 
  \begin{align*}
  \abs{\chi_\alpha^* * \left( (\chi_\alpha^* -\chi_{\alpha,H}^*)\e^\ell \right)}
  \le \chi_\alpha^* * B_{\alpha,H}^*
  \end{align*}
  and
  \begin{align*}
  \abs{\left( \chi_\alpha^* -\chi_{\alpha,H}^*\right)
    * \left( (\chi_\alpha^* -\chi_{\alpha,H}^*)\e^\ell \right)}
  \le B_{\alpha,H}^* * B_{\alpha,H}^*.
  \end{align*}
  Thus, 
  as in the proof of Lemma~\ref{lemma:chi_H-convolution-chi_H}
  the expected upper bound follows from
  Lemma~\ref{lemma:chi-convolution-B_H}
  and Lemma~\ref{lemma:B_H-convolution-B_H}.
\end{proof}

\medskip
\subsection{Fourier analysis}\label{subsecFourier}~

\begin{lemma}\label{lemma:vaaler-expansion}
  For $f: \Z \to \U$ and $0 \leq \kappa_1 \leq \kappa_2$,
  we define $f_{\kappa_1,\kappa_2}$ and $F_{\kappa_1,\kappa_2}$ by
  \begin{equation}\label{eq:definition-f_kappa1-kappa2}
    \forall a\in\Z,\
    f_{\kappa_1,\kappa_2}(a)
    =
    f\left( q^{\kappa_1} r_{\kappa_1,\kappa_2}(a) \right)
  \end{equation}
  and
  \begin{equation}\label{eq:definition-F_kappa1-kappa2}
    \forall t\in\R,\
    F_{\kappa_1,\kappa_2}(t)
    =
    \frac{1}{q^{\kappa_2-\kappa_1}}
    \sum_{0\leq u < q^{\kappa_2-\kappa_1}}
    f_{\kappa_1,\kappa_2}(q^{\kappa_1} u) \,
    \e\left( \frac{- t u}{q^{\kappa_2-\kappa_1}} \right)
    .
  \end{equation}
  Writing
  \begin{equation}
    \label{eq:definition-alpha-kappa1-kappa2}
    \alpha=q^{\kappa_1-\kappa_2}
    ,
  \end{equation}
  for all integers $K\geq 1$ and $H$ defined by
  \begin{equation}
    \label{eq:condition-H-K}
    H = K q^{\kappa_2-\kappa_1} -1,
  \end{equation}
  defining
  \begin{equation}
    \label{eq:definition-f-kappa1-kappa2-H}
    \forall a\in\Z,\
    f_{\kappa_1,\kappa_2,H}(a)
    :=
    q^{\kappa_2-\kappa_1}
    \sum_{\abs{h}\leq H}
    \fourier{\chi_{\alpha,H}}(h)
    F_{\kappa_1,\kappa_2}(h) \,
    \e\left( \frac{h a}{q^{\kappa_2}} \right)
    ,
  \end{equation}
  we have for all $a \in \Z$,
  \begin{equation}\label{eq:vaaler-approximation-f_kappa1-kappa2}
    \abs{
      f_{\kappa_1,\kappa_2}(a)
      -
      f_{\kappa_1,\kappa_2,H}(a)
    }
    \leq
    q^{\kappa_2-\kappa_1}
    \sum_{\abs{k}< K}
    \fourier{B_{\alpha,H}}\left(k q^{\kappa_2-\kappa_1}\right)
    \e\left( \frac{k a}{q^{\kappa_1}} \right)
    \leq 1
    .
  \end{equation}
  Furthermore we have
  \begin{equation}
    \label{eq:periodicity-F_kappa1-kappa2}
    \forall t\in\R,\
    F_{\kappa_1,\kappa_2}\left( t+q^{\kappa_2-\kappa_1} \right)
    =
    F_{\kappa_1,\kappa_2}(t)
    ,
  \end{equation}
  \begin{equation}
    \label{eq:quadratic-mean-F_kappa1-kappa2}
    \forall t\in\R,\
    \sum_{0\leq \ell < q^{\kappa_2-\kappa_1}}
    \abs{F_{\kappa_1,\kappa_2}(t+\ell)}^2
    = 1
    ,
  \end{equation}
  \begin{equation}
    \label{eq:L2-mean-chi_H-F_kappa1-kappa2}
    \forall t\in\R,\
    \sum_{\abs{h}\leq H}
    \abs{\fourier{\chi_{\alpha,H}^*}(h)}^2
    \abs{F_{\kappa_1,\kappa_2}(t+h)}^2 
    \leq \frac{1}{q^{2(\kappa_2-\kappa_1)}}
    .
  \end{equation}
  For $0<\delta \leq \frac12$,
  if $t_1,\ldots,t_N$ is a finite sequence of real numbers which are 
  $\delta$ well-spaced modulo $1$,
  then
  \begin{equation}
    \label{eq:L2-mean-large-sieve}
    \sum_{n=1}^N
    \abs{F_{\kappa_1,\kappa_2}\left( q^{\kappa_2-\kappa_1} t_n \right)}^2 
    <
    1 + \frac{1}{\delta q^{\kappa_2-\kappa_1}}
    .
  \end{equation}
  Assuming \eqref{eq:condition-H-K} we have
  \begin{multline}
    \label{eq:L1-mean-chi_H-F_kappa1-kappa2}
    \forall t\in\R,\
    \sum_{\abs{h}\leq H}
    \abs{\fourier{\chi_{\alpha,H}^*}(h)}
    \abs{F_{\kappa_1,\kappa_2}(t+h)}
    \\
    \leq
    q^{\kappa_1-\kappa_2}
    \left( 2  + \frac{2}{\pi} + \frac{2}{\pi} \log K \right)
    \sum_{0\leq \ell < q^{\kappa_2-\kappa_1}}
    \abs{F_{\kappa_1,\kappa_2}(t+\ell)}
    .
  \end{multline}
\end{lemma}

\begin{proof}
  By definition, for $0\leq u < q^{\kappa_2-\kappa_1}$ we have
  $r_{\kappa_1,\kappa_2}\left( q^{\kappa_1} u \right) = u$,
  hence
  \begin{align*}
    f_{\kappa_1,\kappa_2}(a)
    =
    f\left( q^{\kappa_1} r_{\kappa_1,\kappa_2}(a) \right)
    &
      =
      \sum_{0\leq u < q^{\kappa_2-\kappa_1}}
      f\left(q^{\kappa_1} u \right) \,
      \one_{r_{\kappa_1,\kappa_2}(a)=u}
    \\
    &
      =
      \sum_{0\leq u < q^{\kappa_2-\kappa_1}}
      f\left(q^{\kappa_1} r_{\kappa_1,\kappa_2}(q^{\kappa_1} u) \right) \,
      \one_{r_{\kappa_1,\kappa_2}(a)=u}
    \\
    &
      =
      \sum_{0\leq u < q^{\kappa_2-\kappa_1}}
      f_{\kappa_1,\kappa_2}\left( q^{\kappa_1} u \right) \,
      \one_{r_{\kappa_1,\kappa_2}(a)=u}
      ,
  \end{align*}
  hence,
  by \eqref{eq:definition-alpha-kappa1-kappa2},
  \eqref{eq:digit-detection} and \eqref{eq:definition-chi}
  we can write
  \begin{displaymath}
    f_{\kappa_1,\kappa_2}(a)
    =
    \sum_{0\leq u < q^{\kappa_2-\kappa_1}}
    f_{\kappa_1,\kappa_2}\left( q^{\kappa_1} u \right) \,
    \chi_{\alpha}\left(
      \frac{a}{q^{\kappa_2}}  - \frac{u}{q^{\kappa_2-\kappa_1}}
    \right)
    .
  \end{displaymath}
  By \eqref{eq:vaaler-approximation-chi} and \eqref{eq:definition-A-B}
  we get the approximation
  \begin{multline*}
    \abs{
      f_{\kappa_1,\kappa_2}(a)
      -
      \sum_{\abs{h}\leq H}  \fourier{\chi_{\alpha,H}}(h)
      \sum_{0\leq u < q^{\kappa_2-\kappa_1}}
      f_{\kappa_1,\kappa_2}\left( q^{\kappa_1} u \right) \,
      \e\left(
        \frac{h a}{q^{\kappa_2}} - \frac{h u}{q^{\kappa_2-\kappa_1}}
      \right)
    }
    \\
    \leq
    \sum_{\abs{h}\leq H}  \fourier{B_{\alpha,H}}(h)
    \sum_{0\leq u < q^{\kappa_2-\kappa_1}}
    \e\left(
      \frac{ha}{q^{\kappa_2}} - \frac{h u}{q^{\kappa_2-\kappa_1}}
    \right)
    .
  \end{multline*}
  By the definition of $F_{\kappa_1,\kappa_2}(h)$
  given by \eqref{eq:definition-F_kappa1-kappa2},
  we have
  \begin{displaymath}
    \sum_{0\leq u < q^{\kappa_2-\kappa_1}}
    f_{\kappa_1,\kappa_2}\left( q^{\kappa_1} u \right) \,
    \e\left(
      - \frac{h u}{q^{\kappa_2-\kappa_1}}
    \right)
    =
    q^{\kappa_2-\kappa_1} F_{\kappa_1,\kappa_2}(h)
    ,
  \end{displaymath}
  which leads to $f_{\kappa_1,\kappa_2,H}(a)$
  and gives the left hand side
  of~\eqref{eq:vaaler-approximation-f_kappa1-kappa2}.  
  Since
  \begin{displaymath}
    \sum_{0\leq u < q^{\kappa_2-\kappa_1}}
    \e\left( - \frac{h u}{q^{\kappa_2-\kappa_1}} \right)
    =
    q^{\kappa_2-\kappa_1} \one_{h \equiv 0 \bmod q^{\kappa_2-\kappa_1}}
    ,
  \end{displaymath}
  the right hand side of the inequality above is equal to
  \begin{displaymath}
    q^{\kappa_2-\kappa_1}
    \sum_{\substack{\abs{h} < H+1\\ h\equiv 0 \bmod q^{\kappa_2-\kappa_1}}}
    \fourier{B_{\alpha,H}}(h)
    \e\left( \frac{h a}{q^{\kappa_2}} \right)
    ,
  \end{displaymath}
  and
  by writing $h = k q^{\kappa_2-\kappa_1}$
  and using~\eqref{eq:condition-H-K},
  we obtain the left inequality
  of~\eqref{eq:vaaler-approximation-f_kappa1-kappa2}.

  Using \eqref{eq:vaaler-coef-chi-B},
  \eqref{eq:vaaler-coef-B-star-H}
  and \eqref{eq:condition-H-K},
  the first upper bound
  of~\eqref{eq:vaaler-approximation-f_kappa1-kappa2}
  is at most
  \begin{multline*}
    q^{\kappa_2-\kappa_1}
    \sum_{\abs{k}< K}
    \abs{\fourier{B_{\alpha,H}^*}\left(k q^{\kappa_2-\kappa_1}\right)}
    \\
    \leq
    q^{\kappa_2-\kappa_1}
    \sum_{\abs{k}< K}
    \frac{1}{H+1}\left(1-\frac{\abs{k} q^{\kappa_2-\kappa_1}}{H+1}\right)
    =
    \sum_{\abs{k}< K}
    \frac{1}{K} \left(1-\frac{\abs{k}}{K}\right)
    =
    1
    .
  \end{multline*}

  For $t\in\R$, by the definition of $F_{\kappa_1,\kappa_2}$
  given by \eqref{eq:definition-F_kappa1-kappa2},
  \eqref{eq:periodicity-F_kappa1-kappa2} is obvious, and
  the left hand side of \eqref{eq:quadratic-mean-F_kappa1-kappa2}
  is equal to
  \begin{multline*}
    \frac{1}{q^{2(\kappa_2-\kappa_1)}}
    \sum_{0\leq u < q^{\kappa_2-\kappa_1}}
    \sum_{0\leq v < q^{\kappa_2-\kappa_1}}
    f_{\kappa_1,\kappa_2}(q^{\kappa_1} u) \,
    \conjugate{f_{\kappa_1,\kappa_2}(q^{\kappa_1} v)} \,
    \sum_{0\leq \ell < q^{\kappa_2-\kappa_1}}
    \e\left( \frac{- (t+\ell) (u-v)}{q^{\kappa_2-\kappa_1}} \right)
    \\
    =
    \frac{1}{q^{\kappa_2-\kappa_1}}
    \sum_{0\leq u < q^{\kappa_2-\kappa_1}}
    \abs{f_{\kappa_1,\kappa_2}(q^{\kappa_1} u)}^2
    =
    \frac{1}{q^{\kappa_2-\kappa_1}}
    \sum_{0\leq u < q^{\kappa_2-\kappa_1}}  1
    = 1,
  \end{multline*}
  which proves \eqref{eq:quadratic-mean-F_kappa1-kappa2}.
  
  For $t\in\R$, by using \eqref{eq:vaaler-coef-majoration-star},
  we have
  \begin{multline*}
    \sum_{\abs{h}\leq H}
    \abs{\fourier{\chi_{\alpha,H}^*}(h)}^2
    \abs{F_{\kappa_1,\kappa_2}(t+h)}^2
    \\
    \leq
    \sum_{h\in\Z}
    \abs{\fourier{\chi_{\alpha}^*}(h)}^2
    \abs{F_{\kappa_1,\kappa_2}(t+h)}^2
    \hspace{20mm}
    \\
    =
    \sum_{k\in\Z}
    \sum_{0\leq \ell < q^{\kappa_2-\kappa_1}}
    \abs{\fourier{\chi_{\alpha}^*}(k q^{\kappa_2-\kappa_1}+\ell)}^2
    \abs{F_{\kappa_1,\kappa_2}(t+k q^{\kappa_2-\kappa_1}+\ell)}^2 
    .
  \end{multline*}  
  By the periodicity of $F_{\kappa_1,\kappa_2}$ modulo $q^{\kappa_2-\kappa_1}$
  given by \eqref{eq:periodicity-F_kappa1-kappa2},
  this is equal to
  \begin{displaymath}
    \sum_{0\leq \ell < q^{\kappa_2-\kappa_1}}
    \abs{F_{\kappa_1,\kappa_2}(t+\ell)}^2 
    \sum_{k\in\Z}
    \abs{\fourier{\chi_{\alpha}^*}(k q^{\kappa_2-\kappa_1}+\ell)}^2
    .
  \end{displaymath}
  By Lemma~\ref{lemma:chi-convolution-chi-average}
  and \eqref{eq:quadratic-mean-F_kappa1-kappa2}
  this is 
  \begin{displaymath}
    q^{-2(\kappa_2-\kappa_1)} 
    \sum_{0\leq \ell < q^{\kappa_2-\kappa_1}}
    \abs{F_{\kappa_1,\kappa_2}(t+\ell)}^2 
    =
    q^{-2(\kappa_2-\kappa_1)} 
    ,
  \end{displaymath}
  which completes the proof of
  \eqref{eq:L2-mean-chi_H-F_kappa1-kappa2}.

  By the definition of $F_{\kappa_1,\kappa_2}$
  given by \eqref{eq:definition-F_kappa1-kappa2}
  and the large sieve inequality
  \cite[Theorem~7.7]{iwaniec-kowalski-2004}, we have 
  \begin{multline*}
    \sum_{n=1}^N
    \abs{F_{\kappa_1,\kappa_2}\left( q^{\kappa_2-\kappa_1} t_n \right)}^2 
    =
    \sum_{n=1}^N
    \abs{
      \sum_{0\leq u < q^{\kappa_2-\kappa_1}}
      \frac{f_{\kappa_1,\kappa_2}(q^{\kappa_1} u)}{q^{\kappa_2-\kappa_1}}
      \,
      \e\left( - t_n u \right)
    }^2
    \\
    \leq
    \left(
      q^{\kappa_2-\kappa_1} - 1 + \frac{1}{\delta}
    \right)
    \sum_{0\leq u<q^{\kappa_2-\kappa_1}}
    \frac{\abs{f_{\kappa_1,\kappa_2}(q^{\kappa_1} u)}^2}{q^{2(\kappa_2-\kappa_1)}}
    <
    1 + \frac{1}{\delta q^{\kappa_2-\kappa_1}}
    ,
  \end{multline*}
  and we get \eqref{eq:L2-mean-large-sieve}.

  For $t\in\R$, we have
  \begin{multline*}
    \sum_{\abs{h}\leq H}
    \abs{\fourier{\chi_{\alpha,H}^*}(h)}
    \abs{F_{\kappa_1,\kappa_2}(t+h)}
    \\
    =
    \sum_{-K \leq k < K}
    \sum_{0\leq \ell < q^{\kappa_2-\kappa_1}}
    \abs{\fourier{\chi_{\alpha,H}^*}(k q^{\kappa_2-\kappa_1}+\ell)}
    \abs{F_{\kappa_1,\kappa_2}(t+k q^{\kappa_2-\kappa_1}+\ell)}
    .
  \end{multline*}  
  By the periodicity of $F_{\kappa_1,\kappa_2}$ modulo $q^{\kappa_2-\kappa_1}$
  given by \eqref{eq:periodicity-F_kappa1-kappa2},
  this is equal to
  \begin{displaymath}
    \sum_{0\leq \ell < q^{\kappa_2-\kappa_1}}
    \abs{F_{\kappa_1,\kappa_2}(t+\ell)}
    \sum_{-K \leq k < K}
    \abs{\fourier{\chi_{\alpha,H}^*}(k q^{\kappa_2-\kappa_1}+\ell)}
    ,
  \end{displaymath}
  and by \eqref{eq:vaaler-coef-majoration-star},
  isolating $k=-1$ and $k=0$,
  \begin{align*}
    \sum_{-K \leq k < K}
    \abs{\fourier{\chi_{\alpha,H}^*}(k q^{\kappa_2-\kappa_1}+\ell)}
    &
      \leq
      2 q^{\kappa_1-\kappa_2}
      +
      \sum_{-K \leq k \leq -2} \frac{1}{\pi \abs{k q^{\kappa_2-\kappa_1}+\ell}}
      +
      \sum_{1 \leq k < K} \frac{1}{\pi (k q^{\kappa_2-\kappa_1}+\ell)}
    \\
    &
      \leq
      2 q^{\kappa_1-\kappa_2}
      +
      \sum_{-K \leq k \leq -2} \frac{1}{\pi \abs{k+1} q^{\kappa_2-\kappa_1}}
      +
      \sum_{1 \leq k < K} \frac{1}{\pi k q^{\kappa_2-\kappa_1}}
    \\
    &
      \quad
      =
      2 q^{\kappa_1-\kappa_2}
      \left(
      1
      +
      \sum_{1 \leq k < K} \frac{1}{\pi k}
      \right)
      \leq
      2 q^{\kappa_1-\kappa_2}
      \left( 1 + \frac{1}{\pi} + \frac{1}{\pi} \log K \right)
  \end{align*}
  which completes the proof of
  \eqref{eq:L1-mean-chi_H-F_kappa1-kappa2}.
\end{proof}

\medskip
\subsection{Fourier Transform of $q$-multiplicative function}~
\label{subsection:Fourier-Transform-q-multiplicative-function}

\begin{lemma}\label{Leproductreprensentation}
  If $f$ is $q$-multiplicative
  (Definition~\ref{definition:q-multiplicative-function}),
  $0\leq \kappa_1 \leq \kappa_2 \leq \kappa_3$ are integers,
  $f_{\kappa_1,\kappa_2}$ and  $f_{\kappa_2,\kappa_3}$ are defined
  by~\eqref{eq:definition-f_kappa1-kappa2},
  $F_{\kappa_1,\kappa_2}$ and $F_{\kappa_2,\kappa_3}$ are defined
  by~\eqref{eq:definition-F_kappa1-kappa2},
  then 
  \begin{equation}
    \label{eq:FT-full-product-formula}
    \forall t \in\R,\ 
    F_{\kappa_1,\kappa_2}(t)
    =
    \prod_{\kappa=\kappa_1}^{\kappa_2-1}
    F_{\kappa,\kappa+1} \left( \frac{t}{q^{\kappa_2-(\kappa+1)}} \right)
    ,
  \end{equation}
  and
  \begin{equation}
    \label{eq:FT-product-formula}
    \forall t \in\R,\ 
    F_{\kappa_1,\kappa_3}(t)
    =
    F_{\kappa_1,\kappa_2}\left(\frac{t}{q^{\kappa_3-\kappa_2}}\right)
    F_{\kappa_2,\kappa_3}(t)
    .
  \end{equation}
\end{lemma}
\begin{proof}
  By definition, for $0\leq u < q^{\kappa_2-\kappa_1}$ we have
  $r_{\kappa_1,\kappa_2}\left( q^{\kappa_1} u \right) = u$,
  hence
  by~\eqref{eq:definition-F_kappa1-kappa2}
  and \eqref{eq:definition-f_kappa1-kappa2}
  we have
  \begin{displaymath}
    F_{\kappa_1,\kappa_2}(t)
    =
    \frac{1}{q^{\kappa_2-\kappa_1}}
    \sum_{0\leq u < q^{\kappa_2-\kappa_1}}
    f\left( q^{\kappa_1} u \right)
    \e\left( \frac{- t u}{q^{\kappa_2-\kappa_1}} \right)
    ,
  \end{displaymath} 
  thus, if $f$ is $q$-multiplicative,
  \begin{align*}
    F_{\kappa_1,\kappa_2}(t)
      =
      \prod_{j=0}^{\kappa_2-\kappa_1-1}
      \frac{1}{q}
      \sum_{0\leq u_j < q}
      f\left( q^{\kappa_1+j} u_j \right)
      \e\left( \frac{- t u_j}{q^{\kappa_2-\kappa_1-j}} \right)
      =
      \prod_{\kappa=\kappa_1}^{\kappa_2-1}
      F_{\kappa,\kappa+1} \left( \frac{t}{q^{\kappa_2-(\kappa+1)}} \right)
      ,
  \end{align*}
  which is \eqref{eq:FT-full-product-formula},
  and therefore
  \begin{align*}
    F_{\kappa_1,\kappa_2}\left(\frac{t}{q^{\kappa_3-\kappa_2}}\right)
    F_{\kappa_2,\kappa_3}(t)
    &
      =
      \left(
      \prod_{\kappa=\kappa_1}^{\kappa_2-1}
      F_{\kappa,\kappa+1} \left( \frac{t}{q^{\kappa_3-(\kappa+1)}} \right)
      \right)
      \left(
      \prod_{\kappa=\kappa_2}^{\kappa_3-1}
      F_{\kappa,\kappa+1} \left( \frac{t}{q^{\kappa_3-(\kappa+1)}} \right)
      \right)
      ,
  \end{align*}
  which gives \eqref{eq:FT-product-formula}.
\end{proof}

If $f$ is strongly $q$-multiplicative then $F_{\kappa_1,\kappa_2}(t)$ depends only
on the difference $\lambda = \kappa_2-\kappa_1$, that is, we have
\begin{displaymath}
  F_{\kappa_1,\kappa_2}(t) = F_{0,\lambda}(t)
  = \frac 1{q^\lambda} 
  \sum_{0\le u < q^\lambda} f(u) \e \left( \frac{-tu}{q^{\lambda}} \right).
\end{displaymath}
In this case,
by \eqref{eq:definition-F-lambda},
we can simply write $F_\lambda(t)$ instead of $F_{0,\lambda}(t)$.

We recall the definition of $c(f)$ given in \eqref{eq:definition-c-f}:
\begin{displaymath}
  c(f)
  =
  \frac{-\log \left( \max_{t\in\R} |F_1(t) F_1(qt)| \right)}{2 \log q}.
\end{displaymath}
This means that 
\begin{displaymath}
  \max_{t\in\R} |F_1(t) F_1(qt)| = q^{-2c(f)},
\end{displaymath}
where, by \cite[Section 5.2]{martin-mauduit-rivat-2014}, we have $c(f)>0$.
By applying Lemma~\ref{Leproductreprensentation} we, thus obtain
\begin{displaymath}
  \max_{t\in\R} |F_\lambda(t)| \le q^{2c(f)\lfloor \lambda/ 2 \rfloor } 
  \le q^{c(f)} q^{- c(f) \lambda}
  ,
\end{displaymath}
or in the more general notation
\begin{equation}\label{eqinfinitynorm}
  \max_{t\in\R} |F_{\kappa_1,\kappa_2}(t)| 
  \le
  q^{c(f)} q^{- c(f) (\kappa_2-\kappa_1)}. 
\end{equation}

In the special case of the sum-of-digits function we can be more explicit.

\begin{lemma}\label{lemma:Linf-norm-of-F-for-sum-of-digits}
  If $f(n) = \e\left(\gamma \s_q(n)\right)$ where
  $\gamma\in\R$ and $\s_q(n)$ denotes the sum of digits function in
  base $q$, for integers $0\leq \kappa_1\leq \kappa_2$, we have
  \begin{displaymath}
    \forall t\in\R,\
    \abs{F_{\kappa_1,\kappa_2}(t)} 
    \leq
    \exp\left(
      \frac{\pi^2}{48}
      - (\kappa_2-\kappa_1)
      \frac{\pi^2(q-1)}{12(q+1)} \norm{(q-1)\gamma}^2
    \right)
    .
  \end{displaymath}
\end{lemma}
\begin{proof}
  For $0\leq u < q$ we have
  \begin{math}
    f\left( q^{\kappa} u \right) = \e\left( \gamma u \right)
    ,
  \end{math}
  hence
  \begin{displaymath}
    F_{\kappa,\kappa+1}(t)
    =
    \frac{1}{q}
    \sum_{0\leq u < q}
    f\left( q^{\kappa} u \right)
    \e\left( \frac{- t u}{q} \right)
    =
    \frac{1}{q}
    \sum_{0\leq u < q}
    \e\left(
      \left( \gamma - \frac{t}{q} \right) u \right)
    ,
  \end{displaymath}
  which is independent of $\kappa$,
  and by Lemma 7 of \cite{mauduit-rivat-2009}, it follows that
  \begin{displaymath}
    F_{\kappa,\kappa+1}(t) F_{\kappa+1,\kappa+2}(q t)
    \leq
    \exp\left(
      - \frac{\pi^2(q-1)}{6(q+1)} \norm{(q-1)\gamma}^2
    \right)    
    ,
  \end{displaymath}
  hence, grouping two consecutive terms in
  \eqref{eq:FT-full-product-formula},
  and applying this inequality,
  we get
  \begin{displaymath}
    \abs{F_{\kappa_1,\kappa_2}(t)}
    \leq
    \exp\left(
      -
      \floor{\frac{\kappa_2-\kappa_1}{2}}
      \frac{\pi^2(q-1)}{6(q+1)} \norm{(q-1)\gamma}^2
    \right)
    ,
  \end{displaymath}
  and the result follows by observing that
  \begin{math}
    \floor{\frac{\kappa_2-\kappa_1}{2}}
    \geq
    \frac{\kappa_2-\kappa_1-1}{2}.
  \end{math}
\end{proof}

Lemma~\ref{lemma:Linf-norm-of-F-for-sum-of-digits}
implies in particular that 
\begin{displaymath}
c_q = c(\e(\gamma s_q(n)))
\ge \frac{\pi^2(q-1)}{12(q+1)\log q} \norm{(q-1)\gamma}^2
  \geq
  \frac{\pi^2}{36\log q} \norm{(q-1)\gamma}^2.
\end{displaymath}
In other terms we have
\begin{displaymath}
  \max_{t\in \R}  \abs{F_{\kappa_1,\kappa_2}(t)} 
  \ll  q^{- c_q (\kappa_2-\kappa_1)}
  .
\end{displaymath}

\medskip
\subsection{$L^1$ Estimates}\label{sec:L1estimates}~

We now assume that $f$ is strongly $q$-multiplicative. 
As mentioned in the previous section we have in this case
\begin{displaymath}
  F_{\kappa_1,\kappa_2}(t) = F_{0,\kappa_2-\kappa_1}(t)
  ,
\end{displaymath}
so it remains to consider $  F_\lambda(t)$
defined by \eqref{eq:definition-F-lambda}.

We will see that our main upper bound of the sums of type II
will depend on so-called $L^1$ estimates for $F_\lambda$:
\begin{displaymath}
  L:= \max_{t\in \R} \sum_{0\le h < q^\lambda} \abs{ F_\lambda(t+h)}.
\end{displaymath}
By applying the Cauchy-Schwarz inequality and 
\eqref{eq:quadratic-mean-F_kappa1-kappa2}
we obtain
\begin{displaymath}
  L \le q^{\frac{\lambda}2}.
\end{displaymath}

We also set
\begin{displaymath}
  \varphi_q(t)
  =
  \abs{ \sum_{0 \le j < q} f(j) \e\left( - jt \right) }
  = q \abs{F_1(q t)}
\end{displaymath}
and
\begin{displaymath}
  \Psi_q(t)
  =
  \frac 1q
  \sum_{0\le r < q} \varphi_q \left( t + \frac rq \right). 
\end{displaymath}
Clearly we have
\begin{equation}
  \label{eq:strongly-multiplicative-FT-full-product-formula}
  \forall t \in\R,\
  \abs{ F_\lambda(t) }
  =
  \frac 1{q^\lambda}
  \prod_{\ell = 1}^\lambda \varphi_q\left( \frac t{q^\ell} \right)
\end{equation}
and
\begin{equation}
  \label{eq:strongly-multiplicative-FT-product-formula}
  \forall t \in\R,\
  F_\lambda(t)
  =
  F_\mu(t)   F_{\lambda-\mu}\left(\frac{t}{q^\mu} \right)
  .
\end{equation}
The definition of $\eta(f)$ given in \eqref{eq:definition-eta-f}
can be rewritten as 
\begin{equation}\label{eqdefetaq}
  \eta(f) = \frac 1{\log q} \log \left( \max_{0\le t \le 1 } \Psi_q(t) \right).
\end{equation}
With the help of this notion we have the following $L^1$ estimate for
$F_\lambda(h)$.

\begin{lemma}\label{Le1}
  For integers $0\le \delta \le \lambda$,
  $a\in \Z$ and $t\in\R$ we have 
  \begin{displaymath}
    \sum_{\substack{0\le h < q^\lambda\\ h \equiv a \bmod q^\delta}}
  \abs{F_\lambda(t+h)}
  \le
  q^{\eta(f)(\lambda-\delta)} \abs{F_\delta(t+a)},
  \end{displaymath}
  where $\eta(f)$ is given by (\ref{eqdefetaq}). 
  In particular we have
  \begin{displaymath}
    L
    =
    \max_{t\in\R} \sum_{0\le h < q^\lambda} \abs{F_\lambda(t+h)}
    \le q^{\eta(f) \lambda}.
  \end{displaymath}
\end{lemma}
\begin{proof}
  We proceed as in \cite[Lemme~16]{mauduit-rivat-2010}.
  If $0\le \delta < \lambda$ we write
  \begin{displaymath}
    \sum_{\substack{0\le h < q^\lambda\\ h \equiv a \bmod q^\delta}}
    \abs{F_\lambda(t+h)}
    =
    \sum_{0\leq r < q}
    \sum_{\substack{0\le h < q^{\lambda-1}\\ h \equiv a \bmod q^\delta}}
    \abs{F_\lambda\left(t+h+r q^{\lambda-1}\right)}
  \end{displaymath}
  and by \eqref{eq:strongly-multiplicative-FT-full-product-formula}
  \begin{align*}
    \abs{F_\lambda\left(t+h+r q^{\lambda-1}\right)}
    &=
      \frac1q \varphi_q\left(\frac{t+h+r q^{\lambda-1}}{q^\lambda}\right)
      \abs{F_{\lambda-1}\left(t+h+r q^{\lambda-1}\right)}
    \\
    &=
      \frac1q \varphi_q\left(\frac{t+h}{q^\lambda}+\frac{r}{q}\right)
      \abs{F_{\lambda-1}\left(t+h\right)}
      ,
  \end{align*}
  so that
  \begin{displaymath}
    \sum_{\substack{0\le h < q^\lambda\\ h \equiv a \bmod q^\delta}}
    \abs{F_\lambda(t+h)}
    =
    \sum_{\substack{0\le h < q^{\lambda-1}\\ h \equiv a \bmod q^\delta}}
    \Psi_q\left(\frac{t+h}{q^\lambda}\right)
    \abs{F_{\lambda-1}\left(t+h\right)}
    .
  \end{displaymath}
  By \eqref{eqdefetaq} we get
  \begin{displaymath}
    \sum_{\substack{0\le h < q^\lambda\\ h \equiv a \bmod q^\delta}}
    \abs{F_\lambda(t+h)}
    \leq
    q^{\eta(f)}
    \sum_{\substack{0\le h < q^{\lambda-1}\\ h \equiv a \bmod q^\delta}}
    \abs{F_{\lambda-1}\left(t+h\right)}
    ,
  \end{displaymath}
  and the result follows by induction.
\end{proof}

We note that we always have
\begin{displaymath}
c(f) \le \eta(f).
\end{displaymath}
This follows from the following inequality
\begin{displaymath}
  1 = \sum_{0\le h < q^\lambda} |F_\lambda(h)|^2 
  \le
  \left( \max_{0\le h < q^\lambda} |F_\lambda(h)|  \right)
  \sum_{0\le h < q^\lambda} |F_\lambda(h)| 
  \le q^{c(q)} q^{(\eta(f)-c(f))\lambda}
  ,
\end{displaymath}
that is valid for all $\lambda \ge 1$.

\medskip

If $f(n) = \e(\gamma s_q(n))$, where $s_q(n)$ denotes the $q$-ary
sum-of-digits function 
it is possible to obtain upper bounds for $\eta(f)$
(that we denote by $\eta_q$) 
that improve the trivial upper bound $\eta(f)\le \frac 12$.
By Lemme 14 of \cite{mauduit-rivat-2010} we have
\begin{displaymath}
  \eta_q = \eta(\e(\gamma s_q(n))) \le \frac 1{\log q} 
  \left(
    \frac 2 {q \sin \frac{\pi}{2q}} + \frac 2\pi \log \frac{2q}{\pi}
  \right).
\end{displaymath}
or 
\begin{displaymath}
  \eta_q  \le c'' \frac{\log\log q}{\log q}
\end{displaymath}
for some constant $c''> 0$.
Note that $\eta_q \to 0$ as $q\to \infty$.

\medskip
\subsection{A mixed estimate}~

\begin{lemma}\label{Lemixedestimate}
  Suppose that
  $\lambda = \kappa_2-\kappa_1$,
  $\alpha = q^{-\lambda}$,
  $H = Kq^\lambda -1$,
  and that $f$ is $q$-multiplicative.
  Then we have for $0\le \rho \le \lambda$
  \begin{displaymath}
    \sum_{\abs{h}\leq H}
    \abs{q^\lambda \fourier{\chi_{\alpha,H}^*}(h)
    F_{\kappa_1,\kappa_2}(h)}
    \left(
      \sum_{\abs{\ell}\leq q^\rho}
      \abs{q^\lambda \fourier{\chi_{\alpha,H}^*}(h-\ell)
       F_{\kappa_1,\kappa_2}(h-\ell)}^2
    \right)^{1/2}
    \ll
    \sum_{0\leq k < q^\rho}
    \abs{F_{\kappa_2-\rho,\kappa_2}(k)}
    .
  \end{displaymath}
\end{lemma}
\begin{proof}
  We represent $h$ as
  \begin{displaymath}
    h = k + q^\rho h',
  \end{displaymath}
  where $0\le k < q^{\rho}$ and
  \begin{math}
    \abs{h'} \ll H q^{-\rho} \le K q^{\lambda-\rho}
    .
  \end{math}
  By (\ref{eq:FT-product-formula}) and periodicity we have
  \begin{displaymath}
    F_{\kappa_1,\kappa_2}(k-\ell + q^\rho h')
    =
    F_{\kappa_1,\kappa_2-\rho}\left( h' +  \frac{k-\ell}{q^{\rho}} \right)
    F_{\kappa_2-\rho,\kappa_2}(k-\ell) 
  \end{displaymath}
  which implies that
  \begin{align}
    & \sum_{\abs{h}\leq H}
      \abs{q^\lambda \fourier{\chi_{\alpha,H}^*}(h)
      F_{\kappa_1,\kappa_2}(h)}
      \left(
      \sum_{\abs{\ell}\leq q^\rho}
      \abs{q^\lambda \fourier{\chi_{\alpha,H}^*}(h-\ell)
      F_{\kappa_1,\kappa_2}(h-\ell)}^2
      \right)^{1/2}
      \nonumber \\
    & = \sum_{0\le k < q^\rho} 
      \abs{F_{\kappa_2-\rho,\kappa_2}(k)} 
      \sum_{h'} 
      \abs{q^\lambda \fourier{\chi_{\alpha,H}^*}(k+q^\rho h')
      F_{\kappa_1,\kappa_2-\rho}\left( h' +  \frac{k}{q^{\rho}} \right)} 
      \label{eqmixedestimateproof} \\
    & \qquad 
      \left(
      \sum_{\abs{\ell}\leq q^\rho}
      \abs{q^\lambda \fourier{\chi_{\alpha,H}^*}(k-\ell+q^\rho h')
      F_{\kappa_1,\kappa_2-\rho}\left( h' +  \frac{k-\ell}{q^{\rho}} \right)
      F_{\kappa_2-\rho,\kappa_2}(k-\ell)    }^2
      \right)^{1/2}
      .     
      \nonumber
  \end{align}
  By applying Cauchy-Schwarz's inequality to the sum over $h'$ 
  (with $\abs{h'} \ll H q^{-\rho} \le K q^{\lambda-\rho}$) we obtain
  \begin{align*}
    & \Biggl( 
      \sum_{h'} 
      \abs{q^\lambda \fourier{\chi_{\alpha,H}^*}(k+q^\rho h')
      F_{\kappa_1,\kappa_2-\rho}\left( h' +  \frac{k}{q^{\rho}} \right)} 
    \\
    & \qquad 
      \Biggl(
      \sum_{\abs{\ell}\leq q^\rho}
      \abs{q^\lambda \fourier{\chi_{\alpha,H}^*}(k-\ell+q^\rho h')
      F_{\kappa_1,\kappa_2-\rho}\left( h' +  \frac{k-\ell}{q^{\rho}} \right)
      F_{\kappa_2-\rho,\kappa_2}(k-\ell)    }^2
      \Biggr)^{1/2}
      \Biggr)^2 
    \\
    & \le \sum_{h'}   
      \abs{q^\lambda \fourier{\chi_{\alpha,H}^*}(k+q^\rho h')
      F_{\kappa_1,\kappa_2-\rho}\left( h' +  \frac{k}{q^{\rho}} \right)}^2
    \\
    &
      \qquad
      \times 
         \sum_{h'}
         \sum_{\abs{\ell}\leq q^\rho}
         \abs{q^\lambda \fourier{\chi_{\alpha,H}^*}(k-\ell+q^\rho h')
         F_{\kappa_1,\kappa_2-\rho}\left( h' +  \frac{k-\ell}{q^{\rho}} \right)
         F_{\kappa_2-\rho,\kappa_2}(k-\ell)    }^2
         .     
  \end{align*}
  The first sum can be easily handled by representing $h'$ as
  $h' = mq^{\lambda-\rho} + h''$ with $0\le h'' < q^{\lambda-\rho}$
  and $\abs{m}\ll K$:
  \begin{align*}
    &
      \sum_{h'}   
      \abs{q^\lambda \fourier{\chi_{\alpha,H}^*}(k+q^\rho h')
      F_{\kappa_1,\kappa_2-\rho}\left( h' +  \frac{k}{q^{\rho}} \right)}^2
    \\ 
    & =
      \sum_{0\le h'' < q^{\lambda-\rho}}
      \abs{ F_{\kappa_1,\kappa_2-\rho}\left( h'' +  \frac{k}{q^{\rho}} \right)}^2
      \sum_m
      \abs{q^\lambda \fourier{\chi_{\alpha,H}^*}(k+q^\rho h'' +m)}^2
    \\
    &\ll
      \sum_{0\le h'' < q^{\lambda-\rho}}
      \abs{ F_{\kappa_1,\kappa_2-\rho}\left( h'' +  \frac{k}{q^{\rho}} \right)}^2
      \sum_m 
      \min\left( 1, \frac{q^\lambda}{ \abs{k+q^\rho h'' +m q^\lambda } }    \right)^2 
    \\
    &\ll
      \sum_{0\le h'' < q^{\lambda-\rho}}
      \abs{ F_{\kappa_1,\kappa_2-\rho}\left( h'' +  \frac{k}{q^{\rho}} \right)}^2 
    = O(1).
  \end{align*}
  The second sum can be estimated in a similar way:
  \begin{align*}
    &\sum_{h'}
      \abs{q^\lambda \fourier{\chi_{\alpha,H}^*}(k-\ell+q^\rho h')
      F_{\kappa_1,\kappa_2-\rho}\left( h' +  \frac{k-\ell}{q^{\rho}} \right)
      }^2
    \\
    & =
      \sum_{h'}
      \abs{q^\lambda \fourier{\chi_{\alpha,H}^*}(k-\ell+q^\rho h')
      F_{\kappa_1,\kappa_2-\rho}\left( h' +  \frac{k-\ell}{q^{\rho}} \right)
      }^2
    \\
    & = 
      \sum_{0\le h'' < q^{\lambda-\rho}}
      \abs{
      F_{\kappa_1,\kappa_2-\rho}\left( h'' +  \frac{k-\ell}{q^{\rho}} \right)}^2
      \sum_m
      \abs{q^\lambda \fourier{\chi_{\alpha,H}^*}(k-\ell+q^\rho h'' +m)}^2
    \\
    &\ll
      \sum_{0\le h'' < q^{\lambda-\rho}}
      \abs{ F_{\kappa_1,\kappa_2-\rho}\left( h'' +  \frac{k-\ell}{q^{\rho}} \right)}^2
      \sum_m
      \min\left( 1, \frac{q^\lambda}{ \abs{k-\ell+q^\rho h'' +m q^\lambda } }    \right)^2 
      = O(1).
  \end{align*}
  hence
  \begin{multline*}
    \sum_{h'}
    \sum_{\abs{\ell}\leq q^\rho}
    \abs{q^\lambda \fourier{\chi_{\alpha,H}^*}(k-\ell+q^\rho h')
    F_{\kappa_1,\kappa_2-\rho}\left( h' +  \frac{k-\ell}{q^{\rho}} \right)
    F_{\kappa_2-\rho,\kappa_2}(k-\ell)    }^2
    \\
    \ll
    \sum_{\abs{\ell}\leq q^\rho}
    \abs{F_{\kappa_2-\rho,\kappa_2}(k-\ell)   }^2
    = O(1).
  \end{multline*}
  Thus, by these two estimates and by (\ref{eqmixedestimateproof})
  we obtain the proposed result.
\end{proof}

\medskip
\subsection{$L^2$ estimate along almost arithmetic progressions}~

\begin{lemma}\label{lemma:L2-along-almost-arithmetic-progressions}
  If $f$ is $q$-multiplicative,
  $\kappa_1$, $\kappa_2$ and $\alpha$ are integers
  such that
  $0\leq \kappa_1 \leq \kappa_2$ 
  and
  $0\leq \alpha <\kappa_2-\kappa_1$,
  $A$ is a real number such that
  $q^{\alpha} \leq A < q^{\alpha+1}$
  and $B$ is any real number,
  then  
  \begin{equation}\label{eq:L2-along-almost-arithmetic-progressions}
    \sum_{0\le k < q^{\kappa_2-\kappa_1}/A}
    \abs{ F_{\kappa_1,\kappa_2}\left( \floor{k A} + B\right) }^2
    \leq
    \frac{3q-2}{q-1}
    \max_{t\in\R} \abs{F_{\kappa_2-\alpha,\kappa_2}(t)}^2
    .
  \end{equation}
\end{lemma}
\begin{proof}
  Let
  \begin{math}
    t_k = \floor{k A} + B
    .
  \end{math}
  By \eqref{eq:FT-product-formula} we have
  \begin{align*}
    \sum_{0\le k < q^{\kappa_2-\kappa_1}/A}
    \abs{ F_{\kappa_1,\kappa_2}\left( t_k \right) }^2
    &=
    \sum_{0\le k < q^{\kappa_2-\kappa_1}/A}
    \abs{
      F_{\kappa_1,\kappa_2-\alpha}\left(
        \frac{t_k}{q^{\alpha}}
      \right)}^2
    \abs{F_{\kappa_2-\alpha,\kappa_2}(t_k)}^2
    \\
    &
      \leq
      \left( \max_{t\in\R} \abs{F_{\kappa_2-\alpha,\kappa_2}(t)}^2 \right)
      \sum_{0\le k < q^{\kappa_2-\kappa_1}/A}
      \abs{
      F_{\kappa_1,\kappa_2-\alpha}\left(
      \frac{t_k}{q^{\alpha}}
      \right)}^2
      .
  \end{align*}
  If $A=q^\alpha$,
  then by \eqref{eq:quadratic-mean-F_kappa1-kappa2}
  we have
  \begin{displaymath}
    \sum_{0\le k < q^{\kappa_2-\kappa_1}/A}
    \abs{
      F_{\kappa_1,\kappa_2-\alpha}\left( \frac{t_k}{q^{\alpha}} \right)
    }^2
    =
    \sum_{0\le k < q^{\kappa_2-\kappa_1-\alpha}}
    \abs{
      F_{\kappa_1,\kappa_2-\alpha}\left( k + \frac{B}{q^{\alpha}} 
      \right)
    }^2
    = 1
    ,
  \end{displaymath}
  which gives \eqref{eq:L2-along-almost-arithmetic-progressions}.

  If $A\neq q^\alpha$ then $q^\alpha < A < q^{\alpha+1}$.
  Let $L = \floor{q^{\kappa_2-\kappa_1}/A}$
  and $t'_k = \frac{t_k}{q^{\kappa_2-\kappa_1}}$.
  For $0\leq k_1 < k_2 \leq L-1$, it follows that
  \begin{displaymath}
    \frac{A-1}{q^{\kappa_2-\kappa_1}}
    <
    \frac{(k_2-k_1)A-\{k_2 A\}}{q^{\kappa_2-\kappa_1}}
    \leq
    t'_{k_2}-t'_{k_1}
    =
    \frac{\floor{k_2 A}-\floor{k_1 A}}{q^{\kappa_2-\kappa_1}}
    \leq
    \frac{\floor{(L-1) A}}{q^{\kappa_2-\kappa_1}}
    \leq
    1-\frac{A}{q^{\kappa_2-\kappa_1}}
    ,
  \end{displaymath}
  hence
  \begin{displaymath}
    \norm{t'_{k_2}-t'_{k_1}}
    >
    \frac{A-1}{q^{\kappa_2-\kappa_1}}
    \geq
    \frac{q^\alpha-1}{q^{\kappa_2-\kappa_1}}
    =
    \frac{1-q^{-\alpha}}{q^{\kappa_2-\kappa_1-\alpha}}
    .
  \end{displaymath}
  If $\alpha \geq 1$,
  the sequence $t'_0,\ldots,t'_{L-1}$ is
  \begin{math}
    \delta =
    \frac{1-q^{-\alpha}}{q^{\kappa_2-\kappa_1-\alpha}}
  \end{math}
  well-spaced modulo $1$.
  By \eqref{eq:L2-mean-large-sieve}
  we deduce
  \begin{displaymath}
    \sum_{0\le k \leq L-1}
    \abs{
      F_{\kappa_1,\kappa_2-\alpha}\left( q^{\kappa_2-\kappa_1-\alpha} t'_k \right)
    }^2    
    \leq
    1
    +
    \frac{1}{\delta \, q^{\kappa_2-\kappa_1-\alpha}}
    =
    1
    +
    \frac{1}{1-q^{-\alpha}}
  \end{displaymath}
  hence
  \begin{multline*}
    \sum_{0\le k < q^{\kappa_2-\kappa_1}/A}
    \abs{
      F_{\kappa_1,\kappa_2-\alpha}\left(
        \frac{t_k}{q^{\alpha}}
      \right)}^2
    \leq
    \sum_{0\le k \leq L}
    \abs{
      F_{\kappa_1,\kappa_2-\alpha}\left(  \frac{t_k}{q^{\alpha}} \right)
    }^2
    \\
    \leq
    \abs{
      F_{\kappa_1,\kappa_2-\alpha}\left(  \frac{t_L}{q^{\alpha}} \right)
    }^2
    +
    1
    +
    \frac{1}{1-q^{-\alpha}}
    \leq
    2
    +
    \frac{1}{1-q^{-1}}
    =
    \frac{3q-2}{q-1}
    ,
  \end{multline*}
  which gives \eqref{eq:L2-along-almost-arithmetic-progressions}.
  
  If $\alpha=0$, we have $1<A<q$.
  For $0\leq k_1 < k_2 \leq L$ we have
  $\floor{k_2 A} - \floor{k_1 A} > A-1 >0$,
  which implies that for $0\leq k \leq L$,
  the integers $\floor{k A}$ are all distinct.
  They belong to $[0,q^{\kappa_2-\kappa_1}-1]$.
  Extending the summation
  and using \eqref{eq:quadratic-mean-F_kappa1-kappa2}, it follows that
  \begin{displaymath}
    \sum_{0\le k < q^{\kappa_2-\kappa_1}/A}
    \abs{ F_{\kappa_1,\kappa_2}\left( \floor{k A} + B\right) }^2
    \leq
    \sum_{0\le \ell < q^{\kappa_2-\kappa_1}}
    \abs{ F_{\kappa_1,\kappa_2}\left( \ell + B\right) }^2
    = 1,
  \end{displaymath}
  which, since $F_{\kappa_2,\kappa_2}=1$,
  gives \eqref{eq:L2-along-almost-arithmetic-progressions}.
\end{proof}
\color{black}

\medskip
\subsection{Carry properties}~

\begin{lemma}\label{Le:carryproperty}
Suppose that $f$ is $q$-multiplicative,
and that $\mu,\nu, \rho,\widetilde{\rho}$ are non-negative integers
and set $\lambda = 2\mu + \nu + \rho + \widetilde{\rho}$.
Suppose that $\lambda < 2\mu + 2\nu$ and $q^{\mu-1}\le m < q^\mu$.
Then we have uniformly
\begin{displaymath}
  \# \left\{
    q^{\nu-1}\le n < q^\nu : 
    f_{0,\lambda}(m^2(n+r)^2) - f_{0,\lambda}(m^2n^2) 
    \ne
    f(m^2(n+r)^2) - f(m^2n^2)
  \right\}
  \ll
  q^{\nu-\widetilde{\rho}}
  .
\end{displaymath} 
\end{lemma}

\begin{proof}
The proof is very similar to that of \cite[Lemma 16]{mauduit-rivat-2009}.
We compare the digital expansions of $m^2n^2$ and
$m^2(n+r)^2 = m^2n^2 + m^2(2nr+r^2)$ and note that
$m^2(2nr+r^2)\le q^{2\mu + \nu +\rho +2}$. 
By adding $m^2(2nr+r^2)$ to $m^2n^2$ there will be certainly some
change to the digits with index $0\le j\le 2\mu + \nu +\rho +2$.
Furthermore there might be a carry propagation that effect the digits
beyond $2\mu + \nu +\rho +2$. However, there is only a carry
propagation  beyond the index
$2\mu + \nu +\rho +\widetilde{\rho}$ if the digits of $m^2n^2$ satisfy 
\begin{displaymath}
  \varepsilon_j(m^2n^2) = q-1 \qquad \mbox{for} \qquad 
  2\mu + \nu +\rho +2 \le j \le 2\mu + \nu +\rho +\widetilde{\rho}.
\end{displaymath}
This means that we have
\begin{displaymath}
  \floor{ \frac{m^2n^2}{q^{2\mu + \nu +\rho +2}} }
  = q^{\widetilde{\rho}-1} \ell -1
\end{displaymath}
for some integer $\ell> 0$. Equivalently we have
\begin{displaymath}
  q^{\widetilde{\rho}-1} \ell -1
  \le
  \frac{m^2n^2}{q^{2\mu + \nu +\rho +2}}
  < q^{\widetilde{\rho}-1} \ell
\end{displaymath}
which implies that
\begin{displaymath}
  0 < \ell \ll q^{\nu-\rho-\widetilde{\rho}}.
\end{displaymath}
Hence, the total number of $q^{\nu-1}\le n < q^\nu$ for which there is
a carry propagation  
beyond the index $2\nu +\rho +\widetilde{\rho}$ is upper bounded by
\begin{align*}
  &
    \ll
    \sum_{\ell = 1}^{C q^{\nu-\rho-\widetilde{\rho}}}
    \left(
    1 + 
    \frac 1m\sqrt{ \ell q^{\widetilde{\rho}-1} q^{2\mu+\nu + \rho +2 } }
    - \frac 1m
    \sqrt{ \left(\ell q^{\widetilde{\rho}-1}-1 \right) q^{2\mu+\nu + \rho +2 } }
    \right)
  \\
  &
    \ll
    q^{\nu-\rho-\widetilde{\rho}}
    + q^{\frac 12(\nu+\rho+\widetilde{\rho} )}
    \sum_{\ell = 1}^{C q^{\nu-\rho-\widetilde{\rho}} }
    \frac 1{ q^{\widetilde{\rho}} \sqrt{\ell} } 
  \\
  &
    \ll
    q^{\nu-\widetilde{\rho}}
  .
\end{align*}
This means that in at most $O(q^{\nu-\widetilde{\rho}})$ cases the
digits of $m^2n^2$ and $m^2(n+r)^2$ differ at some digit with index 
$j\ge 2\mu + \nu + \rho + \widetilde{\rho}$. 
This completes the proof of the lemma.
\end{proof}

As a consequence we get for all $0\le \kappa \le \nu-\rho$
and $1\le s< q^\rho$ 
\begin{align*}
  &
    \# \{ q^{\nu-1}\le n < q^\nu : 
    f_{\kappa,\lambda}((m+sq^\kappa)^2(n+r)^2) - f_{\kappa,\lambda}(m^2(n+r)^2) 
  \\
  & \qquad \qquad \qquad \qquad \qquad
    - f_{\kappa,\lambda}((m+sq^\kappa)^2n^2)+f_{\kappa,\lambda}(m^2n^2) \\ 
  & \qquad \qquad \qquad \qquad
    \ne 
    f((m+sq^\kappa)^2(n+r)^2) - f(m^2(n+r)^2) 
  \\
  & \qquad \qquad \qquad \qquad \qquad
    - f((m+sq^\kappa)^2n^2)+f(m^2n^2)
    \}
  \\
  &
    \ll
    q^{\nu-\widetilde{\rho}}
    .
\end{align*}
We just have to apply Lemma~\ref{Le:carryproperty} twice and to
observe 
that
\begin{align*}
  &
    f_{0,\lambda}((m+sq^\kappa)^2(n+r)^2) - f_{0,\lambda}(m^2(n+r)^2) 
    - f_{0,\lambda}((m+sq^\kappa)^2n^2)+f_{0,\lambda}(m^2n^2) 
  \\
  &=
    f_{\kappa,\lambda}((m+sq^\kappa)^2(n+r)^2) - f_{\kappa,\lambda}(m^2(n+r)^2) 
    - f_{\kappa,\lambda}((m+sq^\kappa)^2n^2)+f_{\kappa,\lambda}(m^2n^2) 
    .
\end{align*}

\medskip
\subsection{Exponential sums and related estimates}~

We will often make use of the following upper bound of geometric
series of ratio $\e(\xi)$ for $(L_1,L_2)\in\Z^2$,
$L_1 \leq L_2$ and $\xi\in\R$:
\begin{equation}\label{eq:estimate-geometric-series}
  \abs{\sum_{L_1<\ell \leq L_2} \e(\ell \xi)} 
  \leq 
  \min\left( L_2-L_1, \abs{\sin\pi\xi}^{-1} \right).
\end{equation}

\begin{lemma}\label{lemma:trivial-sum-of-minimums}
  For all real numbers $M>0$, $\xi\in \R\setminus \Z$, 
  $\varphi \in \R$,
  $(N_1,N_2)\in\Z^2$ with $N_1 \leq N_2$
  we have
  \begin{equation}\label{eq:trivial-sum-of-minimums}
    \sum_{N_1 < n \le N_2} 
    \min \left( M, \abs{\sin \pi (n\xi+\varphi) }^{-1} \right)
    \ll
    \left( 3+\floor{(N_2-N_1)\norm{\xi}} \right)
    \left( 
      3 M + \norm{\xi}^{-1} \log  \norm{\xi}^{-1}
    \right)
    .
  \end{equation}
\end{lemma}
\begin{proof}
  This is Lemma 6 of \cite{drmota-mauduit-rivat-two-bases}.
\end{proof}
\begin{lemma}\label{lemma:estimate-second-derivative}
  Let $N\geq 1$ be an integer, $I = [1,N]$,
  suppose that $f: I \to \R$ has two continuous derivatives on $I$,
  and that for some $\lambda_2>0$ and $c_2\geq 1$ we have
  \begin{displaymath}
    \forall x \in I,\
    \lambda_2 \leq \abs{f''(x)} \leq c_2 \lambda_2
    .
  \end{displaymath}
  Then
  \begin{displaymath}
    \abs{\sum_{n=1}^N \e\left(f(n)\right)}
    \ll
    c_2 \lambda_2^{1/2} N + \lambda_2^{-1/2}
    .
  \end{displaymath}
\end{lemma}
\begin{proof}
  This is a classical result of van der Corput.
  See e.g. \cite[Theorem 2.2]{graham-kolesnik-1991}.
\end{proof}
\begin{lemma}\label{lemma:complete-gauss-sum}
  For all $a, b, m \in \Z$ with $m\geq 1$, we have
  \begin{equation}
    \label{eq:complete-gauss-sum}
    \abs{\sum_{n=0}^{m-1} \e\left(\frac{an^2+bn}{m}\right)} 
    \leq 
    \sqrt{2 m \gcd(a,m)}.
  \end{equation}
\end{lemma}
\begin{proof}
  This is Proposition 2 of \cite{mauduit-rivat-2009}.
\end{proof}
For incomplete quadratic Gauss sums we have
\begin{lemma}\label{lemma:incomplete-gauss-sum}
  For all $a, b, m, N ,n_0 \in \Z$ with $m\geq 1$ and $N\geq 0$, 
  we have
  \begin{equation}
    \label{eq:incomplete-gauss-sum}
    \abs{\sum_{n=n_0+1}^{n_0+N} \e\left(\frac{an^2+bn}{m}\right)} 
    \leq 
    \left( \frac{N}{m} + 1 + \frac{2}{\pi} \log\frac{2m}{\pi} \right)
    \sqrt{2 m \gcd(a,m)}.
  \end{equation}
\end{lemma}
\begin{proof}
  This is Lemma 5 of \cite{mauduit-rivat-RS-squares}.
\end{proof}
If $b$ is not necessarily an integer we have a similar
but slightly worse upper bound.
\begin{lemma}[Weyl]\label{lemma:incomplete-gauss-sum-2}
  If $\alpha\in\R$ satisfies
  \begin{math}
    \abs{\alpha - \frac{a}{m}}
    \leq \frac{1}{m^2}
  \end{math}
  for some relatively prime integers $a\in\Z$ and $m\geq 1$,
  and $(\beta,\gamma)\in\R^2$, then for $N\geq 1$,
  \begin{equation}
    \label{eq:incomplete-gauss-sum-2}
    \abs{\sum_{n=n_0+1}^{n_0+N}
      \e\left(\alpha n^2 + \beta n + \gamma \right)} 
    \ll
    \frac{N}{\sqrt{m}}
    + \sqrt{N \log m}
    + \sqrt{m \log m}
    .
  \end{equation}
\end{lemma}
\begin{proof}
  This is Theorem 1 of Chapter 3 of \cite{montgomery-1994}, p. 41.
\end{proof}

\begin{lemma}\label{lemma:gcd-sum}
  For integers $m\geq 1$, $A\geq 1$, and $\gamma\in\R$ we have
  \begin{equation}\label{eq:gcd-sum}
    \frac1A \sum_{1\leq a \leq A} \left( \gcd(a,m) \right)^\gamma
    \leq
    \sigma_{\gamma-1}(m)
    .
  \end{equation}  
\end{lemma}
\begin{proof}
  We have
  \begin{displaymath}
    \sum_{1\leq a \leq A} \left( \gcd(a,m) \right)^\gamma
    =
    \sum_{\substack{d\dv m\\ d\leq A}} d^\gamma
    \sum_{\substack{1\leq a\leq A\\ \gcd(a,m)=d}} 1
    \leq
    \sum_{\substack{d\dv m\\ d\leq A}} d^\gamma
    \sum_{\substack{1\leq a\leq A\\ d \dv a}} 1
    =
    \sum_{\substack{d\dv m\\ d\leq A}} d^\gamma \floor{\frac{A}{d}}
    \leq
    A \sum_{d\dv m} d^{\gamma-1} 
    ,
  \end{displaymath}
  as expected.
\end{proof}
\begin{lemma}
  For integers $q \geq 2$ and $\lambda \geq 1$, we have
  \begin{equation}\label{eq:tau-q-lambda}
    \left(1+\lambda\right)^{\omega(q)}
    \leq
    \sigma_0\left(q^{\lambda}\right)
    =
    \tau\left(q^{\lambda}\right)
    \leq \tau(q) \lambda^{\omega(q)}    
    ,
  \end{equation}
  and for $x\in\R$, $x<0$,
  we have the following upper bound,
  independent of $\lambda$:
  \begin{equation}\label{eq:sigma-x-q-lambda}
    \sigma_x\left(q^{\lambda}\right)
    <
    \prod_{p\dv q} \frac1{1-p^x}
    .
  \end{equation}
\end{lemma}
\begin{proof}
  For $p$ prime and integers $\nu\geq 1$, $\lambda\geq 1$ we have
  \begin{displaymath}
    \left(1+\lambda\right)^{\omega\left(p^\nu\right)}
    =
    1+\lambda
    \leq
    \tau\left(p^{\nu\lambda}\right)
    = 1+\nu\lambda
    \leq (1+\nu)  \lambda
    = \tau\left(p^\nu\right) \lambda^{\omega\left(p^\nu\right)}    
  \end{displaymath}
  and \eqref{eq:tau-q-lambda} follows by multiplicativity.

  For $p$ prime, integers $\nu\geq 1$, $\lambda\geq 1$
  and $x\in\R$, $x<0$ we have $0<p^x<1$ and
  \begin{displaymath}
    \sigma_x\left(p^{\nu \lambda}\right)
    =
    \sum_{j=0}^{\nu \lambda} p^{jx}
    <
    \sum_{j=0}^\infty p^{jx}
    =
    \frac1{1-p^x},
  \end{displaymath}
  and \eqref{eq:sigma-x-q-lambda} follows by multiplicativity.
\end{proof}

The following lemma is a variant of van der corput's inequality.
\begin{lemma}\label{lemma:van-der-corput}
  For integers $N\geq 1$, $N'\geq 1$, $R\geq 1$
  and complex numbers $z_1,\ldots,z_N$ we have
  \begin{displaymath}
    \abs{ \sum_{n=1}^N z_n }^2
    \leq
    \frac{N+N'R-N'}{R} \
    \Re\left(
      \sum_{n=1}^N \abs{z_n}^2
      +
      2 \sum_{r=1}^{R-1}\left( 1-\frac{r}{R} \right)
      \sum_{n=1}^{N-N'r} z_{n+N'r}\conjugate{z_n}
    \right)
    .
  \end{displaymath}
\end{lemma}
\begin{proof}
  See, for example, Lemma 17 of \cite{mauduit-rivat-2009}.
\end{proof}
\begin{lemma}\label{lemma:double-exponential-sum-mn^2}
  For integers $M\geq 1$ and $N\geq 1$, complex numbers
  $a_1,\ldots,a_M$ and $b_1,\ldots,b_N$ of modulus at most $1$,
  $\xi_3\in \R\setminus \Z$,
  $\xi_1\in \R$, 
  we have
  \begin{equation}\label{eq:double-exponential-sum-mn^2}
    \abs{\frac{1}{MN} \sum_{m=1}^M \sum_{n=1}^N
      a_m b_n \e(\xi_3 m n^2 + \xi_1 mn)}^2
    \ll
    \norm{\xi_3}^{1/2}
    +
    \frac{\log^2 \norm{\xi_3}^{-1}}{M N^2 \norm{\xi_3}}  
    +
    \frac{1}{N} 
    +
    \frac{\log^2 \norm{\xi_3}^{-1}}{M} 
    .
\end{equation}
\end{lemma}
\begin{proof}
  By symmetry and periodicity we may assume without loss of generality
  that $0<\xi_3\leq 1/2$.
  Furthermore for $1/16 \leq\xi_3\leq 1/2$
  inequality~\eqref{eq:double-exponential-sum-mn^2} is trivially satisfied.
  Therefore we may assume $0<\xi_3<1/16$.
  By Cauchy-Schwarz we have
  \begin{displaymath}
    \abs{\sum_{m=1}^M \sum_{n=1}^N
      a_m b_n \e\left(\xi_3 m n^2 + \xi_1 mn\right)}^2
    \leq
    M
    \sum_{m=1}^M
    \abs{ \sum_{n=1}^N  b_n \e\left(\xi_3 m n^2 + \xi_1 mn\right)}^2
    .
  \end{displaymath}
  Using Lemma \ref{lemma:van-der-corput}
  (van der corput's inequality) with $1\leq R \leq N$
  and exchanging the order of summation
  this is 
  \begin{displaymath}
    \ll
    \frac{M^2N^2}{R}
    +
    \frac{M N}{R}
    \sum_{1\le r< R}
    \sum_{1\le n \leq N-r}
    \abs{
      \sum_{m=1}^M \e\left(\xi_3 m (2rn +r^2) + \xi_1 mr\right) 
    }
    ,
  \end{displaymath}
  hence by \eqref{eq:estimate-geometric-series}
  \begin{displaymath}
    \ll
    \frac{M^2N^2}{R}
    +
    \frac{M N}{R}
    \sum_{1\le r< R}
    \sum_{1\le n \leq N-r}
    \min\left(
      M,
      \abs{\sin\pi\left(\xi_3 (2rn +r^2) + \xi_1 r\right)}^{-1}
    \right)
    ,
  \end{displaymath}
  hence,
  provided $2r\xi_3\not\in\Z$ for all $r\in\{1,\ldots,R-1\}$,
  by Lemma \ref{lemma:trivial-sum-of-minimums} we get
  \begin{displaymath}
    \ll
    \frac{M^2N^2}{R}
    +
    \frac{M N}{R}
    \sum_{1\le r< R}
    (1 + N\norm{2r\xi_3})
    \left( 
      M+ \norm{2r\xi_3}^{-1} \log  \norm{2r\xi_3}^{-1}
    \right)
    .
  \end{displaymath}
  If we assume that $R \leq (4 \xi_3)^{-1}$, it follows that
  $\norm{2r\xi_3}= 2r\xi_3$ for all $r\in\{1,\ldots,R-1\}$,
  thus $0<2r\xi_3\leq 1/2$ and the condition $2r\xi_3\not\in\Z$
  is fullfilled,
  and we get
  \begin{displaymath}
    \ll
    \frac{M^2N^2}{R}
    +
    M^2 N 
    +
    M^2 N^2 R \, \xi_3
    +
    \frac{M N}{R}
    \sum_{1\le r< R}
    \frac{1}{2r\xi_3} \log \frac{1}{2r\xi_3}
    +
    \frac{M N}{R}
    \sum_{1\le r< R}
    N \log \frac{1}{2r\xi_3}
    ,
  \end{displaymath}
  hence
  \begin{displaymath}
    \ll
    \frac{M^2N^2}{R}
    +
    M^2 N 
    +
    M^2 N^2 R \, \xi_3
    +
    \frac{M N}{R \xi_3}
    (\log R)
    \log \frac{1}{2\xi_3}
    +
    M N^2 \log \frac{1}{2\xi_3}
    ,
  \end{displaymath}
  so that, taking
  \begin{displaymath}
    R
    =
    \min\left(N, \floor{\xi_3^{-1/2}} \right)
    ,
  \end{displaymath}
  since $0<\xi_3<1/16$ the condition
  \begin{math}
    1\leq R \leq \min\left(N, (4\xi_3)^{-1} \right)
  \end{math}
  is satisfied, and we get
  \begin{displaymath}
    \ll
    M^2N^2 \xi_3^{1/2}
    +
    M^2 N 
    +
    M  \xi_3^{-1} \log^2 \xi_3^{-1}
    +
    M N \xi_3^{-1/2} \log^2 \xi_3^{-1}
    +
    M N^2 \log \xi_3^{-1}
    .
  \end{displaymath}
  since
  \begin{math}
    N \xi_3^{-1/2} \leq \frac12 \left( N^2 + \xi_3^{-1} \right)
  \end{math}
  we obtain
  \begin{displaymath}
    \ll
    M^2N^2 \xi_3^{1/2}
    +
    M^2 N 
    +
    M  \xi_3^{-1} \log^2 \xi_3^{-1}
    +
    M N^2 \log^2 \xi_3^{-1}
    .
  \end{displaymath}
\end{proof}

\begin{lemma}\label{lemma:double-exponential-sum-xi_2}
  Let $M\geq 1$ and $N\geq 1$ be integers,
  $a_1,\ldots,a_M$ and $b_1,\ldots,b_N$ be complex numbers of modulus
  at most $1$, 
  and $\xi_1,\xi_2,\xi_3$ be real numbers.
  If $\xi_2\in \R\setminus \Z$, we have
  \begin{equation}\label{eq:double-exponential-sum-xi_2}
    \abs{\frac{1}{MN} \sum_{m=1}^M \sum_{n=1}^N
      a_m b_n \e\left(\xi_3 m n^2 + \xi_2 m^2 n + \xi_1 mn\right)}^2
    \ll
    \norm{\xi_2}^{1/3}
    +
    \frac{1}{M N^{1/2} \norm{\xi_2}^{1/2}}
    +
    \frac{1}{M^{1/2}}
    +
    \frac{1}{N}
    .
  \end{equation}
\end{lemma}
\begin{proof}
  By Cauchy-Schwarz we have
  \begin{multline*}
    \abs{\sum_{m=1}^M \sum_{n=1}^N
      a_m b_n \e\left( \xi_3 m n^2 + \xi_2 m^2 n + \xi_1 mn\right)}^2
    \\
    \leq
    M
    \sum_{m=1}^M
    \abs{
      \sum_{n=1}^N
      b_n \e\left( \xi_3 m n^2 + \xi_2 m^2 n + \xi_1 mn\right)
    }^2
    .
  \end{multline*}
  Using Lemma \ref{lemma:van-der-corput}
  (van der corput's inequality)
  with $1\leq R \leq N$
  and exchanging the order of the summations this is 
  \begin{displaymath}
    \ll
    \frac{M^2N^2}{R}
    +
    \frac{M N}{R}
    \sum_{1\le r< R}
    \sum_{1\le n \leq N-r}
    \abs{
      \sum_{m=1}^M
      \e\left( \xi_3 m (2rn +r^2) + \xi_2 m^2r+ \xi_1 mr \right) 
    }
    .
  \end{displaymath}
  Since $\xi_2\in\R\setminus\Z$,
  by symmetry and periodicity we may assume without loss of generality
  that $0<\xi_2\leq 1/2$.
  After applying Lemma~\ref{lemma:estimate-second-derivative}
  to the summation over $m$,
  the estimate above becomes
  \begin{displaymath}
    \ll
    \frac{M^2N^2}{R}
    +
    \frac{M N^2}{R}
    \sum_{1\le r< R}
    \left( (\xi_2 r)^{1/2} M + (\xi_2 r)^{-1/2} \right) 
    ,
  \end{displaymath}
  which is
  \begin{displaymath}
    \ll
    M^2 N^2
    \left(
      \frac{1}{R}
      +
      \xi_2^{1/2} R^{1/2}
      +
      \frac{1}{M R^{1/2} \xi_2^{1/2}}
    \right) 
    .
  \end{displaymath}
  Taking
  \begin{displaymath}
    R =
    \min\left(
      N,
      \ceil{ \max\left( (2\xi_2)^{-1/3}, (M\xi_2)^{-1} \right) }
    \right)
    ,
  \end{displaymath}
  since $0<2\xi_2\leq 1$,
  the condition $1\leq R \leq N$ is satisfied,
  and we get
  \begin{displaymath}
    \ll
    M^2 N^2
    \left(
      \frac{1}{N}
      +
      \xi_2^{1/3}
      +
      \xi_2^{1/2}
      +
      \frac{1}{M^{1/2}}
      +
      \frac{1}{M N^{1/2} \xi_2^{1/2}}
    \right) 
    .
  \end{displaymath}
  which gives inequality~\eqref{eq:double-exponential-sum-xi_2}.
\end{proof}

\begin{lemma}\label{lemma:double-exponential-sum-m^2n^2}
  Let $M\geq 1$ and $N\geq 1$ be integers,
  $a_1,\ldots,a_M$ and $b_1,\ldots,b_N$ be complex numbers of modulus
  at most $1$, 
  and $\xi_1,\xi_2,\xi_3,\xi_4$ be real numbers.
  
  If $\xi_4\in \R\setminus \Z$ we have
  \begin{multline}\label{eq:double-exponential-sum-m^2n^2}
    \abs{\frac{1}{MN} \sum_{m=1}^M \sum_{n=1}^N
      a_m b_n \e(\xi_4 m^2 n^2 + \xi_3 m n^2 + \xi_2 m^2 n + \xi_1 mn)}^4
    \\
    \ll
    \norm{\xi_4}^{2/5}
    +
    \frac{1}{N}
    +
    \frac{\log \norm{\xi_4}^{-1}}{M} 
    +
    \left(
      \frac{\norm{\xi_4}^{-1} }{M^2 N^2} 
      +
      \frac{\norm{\xi_4}^{-3/5} }{M N^2}  
      +
      \frac{\norm{\xi_4}^{-4/5} }{M^2 N} 
      +
      \frac{\norm{\xi_4}^{-2/5} }{M N}  
    \right)
    \log^3 \norm{\xi_4}^{-1}  
    .
  \end{multline}
\end{lemma}
\begin{proof}
  By Cauchy-Schwarz we have
  \begin{multline*}
    \abs{\sum_{m=1}^M \sum_{n=1}^N
      a_m b_n \e(\xi_4 m^2 n^2 + \xi_3 m n^2 + \xi_2 m^2 n + \xi_1 mn)}^2
    \\
    \leq
    M
    \sum_{m=1}^M
    \abs{
      \sum_{n=1}^N
      b_n \e(\xi_4 m^2 n^2 + \xi_3 m n^2 + \xi_2 m^2 n + \xi_1 mn)
    }^2
    .
  \end{multline*}
  Using Lemma \ref{lemma:van-der-corput}
  (van der corput's inequality)
  with $1\leq R \leq N$
  and exchanging the order of the summations this is 
  \begin{displaymath}
    \ll
    \frac{M^2N^2}{R}
    +
    \frac{M N}{R}
    \sum_{1\le r< R}
    \sum_{1\le n \leq N-r}
    \abs{
      \sum_{m=1}^M
      \e(\xi_4 m^2 (2rn +r^2) + \xi_3 m (2rn +r^2) + \xi_2 m^2r+ \xi_1 mr) 
    }
    .
  \end{displaymath}
  By Cauchy-Schwarz we have
  \begin{multline*}
    \abs{\sum_{m=1}^M \sum_{n=1}^N
      a_m b_n \e(\xi_4 m^2 n^2 + \xi_3 m n^2 + \xi_2 m^2 n + \xi_1 mn)
    }^4
    \\
    \ll
    \frac{M^4 N^4}{R^2}
    +
    \frac{M^2 N^3}{R}
    \sum_{1\le r< R}
    \sum_{1\le n \leq N-r}
    \abs{
      \sum_{m=1}^M
      \e(\xi_4 m^2 (2rn +r^2) + \xi_3 m (2rn +r^2) + \xi_2 m^2r+ \xi_1 mr) 
    }^2
  \end{multline*}
  and using Lemma \ref{lemma:van-der-corput}
  (van der corput's inequality)
  with $1\leq S \leq M$ this is 
  \begin{multline*}
    \ll
    \frac{M^4 N^4}{R^2}
    +
    \frac{M^4 N^4}{S}
    +
    \frac{M^3 N^3}{RS}
    \sum_{1\le r< R}
    \sum_{n=1}^{N-r}
    \sum_{1\le s< S} \left(1-\frac{s}{S}\right)
    \times
    \\
    \Re\left(
      \sum_{m=1}^{M-s}
      \e\left(
        \xi_4 (2sm+s^2) (2rn +r^2) + \xi_3 s (2rn +r^2)
        + \xi_2 (2sm+s^2) r+ \xi_1 rs
      \right)
    \right) 
  \end{multline*}
  hence by \eqref{eq:estimate-geometric-series}
  \begin{multline*}
    \ll
    \frac{M^4 N^4}{R^2}
    +
    \frac{M^4 N^4}{S}
    +
    \frac{M^3 N^3}{RS}
    \sum_{1\le r< R}
    \sum_{1\le s< S} 
    \sum_{n=1}^{N-r}
    \min\left(M,\abs{\sin\pi(\xi_4 2s (2rn +r^2)+\xi_2 2rs)}^{-1}\right)
    .
  \end{multline*}
  Let us assume that $\xi_4\in\R\setminus\Z$ and prove
  inequality~\eqref{eq:double-exponential-sum-m^2n^2}.  
  By symmetry and periodicity we may assume without loss of generality
  that $0<\xi_4\leq 1/2$.
  Furthermore for $8^{-5/2} \leq\xi_4\leq 1/2$
  inequality~\eqref{eq:double-exponential-sum-m^2n^2} is trivially satisfied.
  Therefore we may assume $0<\xi_4<8^{-5/2}$.
  Provided $4rs\xi_4\not\in\Z$
  for all $r\in\{1,\ldots,R-1\}$ and all $s\in\{1,\ldots,S-1\}$,
  we may apply Lemma~\ref{lemma:trivial-sum-of-minimums},
  and we get
  \begin{multline*}
    \ll
    \frac{M^4 N^4}{R^2}
    +
    \frac{M^4 N^4}{S}
    +
    \frac{M^3 N^3}{RS}
    \sum_{1\le r< R}
    \sum_{1\le s< S} 
    (1+N\norm{4rs \xi_4})
    (M+\norm{4rs \xi_4}^{-1}\log \norm{4rs \xi_4}^{-1})
  \end{multline*}
  \begin{multline*}
    \ll
    \frac{M^4 N^4}{R^2}
    +
    \frac{M^4 N^4}{S}
    +
    M^4 N^3
    +
    \frac{M^4 N^4}{RS}
    \sum_{1\le r< R}
    \sum_{1\le s< S} 
    \norm{4rs \xi_4}
    \\
    +
    \frac{M^3 N^3}{RS}
    \sum_{1\le r< R}
    \sum_{1\le s< S} 
    (
    \norm{4rs \xi_4}^{-1}\log \norm{4rs \xi_4}^{-1}
    +N \log \norm{4rs \xi_4}^{-1})
  \end{multline*}
  If we assume $RS \leq (8\xi_4)^{-1}$,
  for all $r\in\{1,\ldots,R-1\}$ and all $s\in\{1,\ldots,S-1\}$
  we have  $\norm{4rs \xi_4} = 4rs \xi_4$,
  hence $0<4rs \xi_4 \leq \frac12$,
  which ensures that the condition $4rs\xi_4\not\in\Z$ is fullfilled,
  and we get
  \begin{multline*}
    \ll
    \frac{M^4 N^4}{R^2}
    +
    \frac{M^4 N^4}{S}
    +
    M^4 N^3
    +
    M^4 N^4 RS \xi_4
    +
    \frac{M^3 N^3}{RS}
    \xi_4^{-1} \log^3 \xi_4^{-1}
    +
    M^3 N^4 \log \xi_4^{-1}
  \end{multline*}
  so that, taking
  \begin{displaymath}
    R = \min\left(N, \floor{\xi_4^{-1/5}} \right)
    ,\quad
    S = \min\left(M, \floor{\xi_4^{-2/5}} \right)
    ,
  \end{displaymath}
  since $0<\xi_4<8^{-5/2}$ the conditions
  $1\leq R\leq N$, $1\leq S \leq M$ and
  \begin{math}
    RS \leq (8\xi_4)^{-1}
  \end{math}
  are satisfied, and we get
  \begin{multline*}
    \ll
    M^4 N^4 \xi_4^{2/5}
    +
    M^4 N^3
    +
    M^3 N^4 \log \xi_4^{-1}
    +
    M^2 N^2  \xi_4^{-1} \log^3 \xi_4^{-1}
    +
    M^3 N^2  \xi_4^{-3/5} \log^3 \xi_4^{-1}
    \\
    +
    M^2 N^3  \xi_4^{-4/5} \log^3 \xi_4^{-1}
    +
    M^3 N^3  \xi_4^{-2/5} \log^3 \xi_4^{-1}
  \end{multline*}
  which is inequality~\eqref{eq:double-exponential-sum-m^2n^2}.
\end{proof}

\medskip
\subsection{Combinatorial identity}~

\begin{lemma}\label{lemma:combinatorial-identity}
  Let $q\geq 2$, $x \geq q^2$,
  $0<\beta_1<1/3$ , $1/2<\beta_2<1$.
  Let $g$ be an arithmetic function.
  We suppose that for all real numbers $M\leq x$
  and all complex numbers $a_m$, $b_n$ with $|a_m|\leq 1$, $|b_n|\leq 1$,
  we have
  \begin{align}
    \label{definition-sums-type-II}
    \abs{
      \sum_{\frac{M}{q} < m \leq M}
      \sum_{\frac{x}{qm}< n \leq \frac{x}{m}}
      a_m b_n g(mn)
    }
    &\leq U \quad
    \mbox{for} \quad x^{\beta_1} \leq M \leq x^{\beta_2} 
    \quad \mbox{(type II)},
    \\
    \hspace{8mm}
    \label{definition-sums-type-I}
    \sum_{\frac{M}{q} < m \leq M}
    \max_{\frac{x}{q m}\leq t \leq \frac{x}{m}}
    \abs{\sum_{t < n \leq \frac{x}{m}} g(mn)}
    &\leq U \quad
    \mbox{for} \quad M\leq x^{\beta_1} 
    \quad \mbox{(type I)}.
  \end{align}
  Then
  \begin{displaymath}
    \abs{\sum_{x/q < n \leq x} \Lambda(n) g(n)}
    \ll U\, (\log x)^2.
  \end{displaymath}
\end{lemma}
\begin{proof}
  This is Lemme 1 from \cite{mauduit-rivat-2010}
  based on Vaughan's identity.
\end{proof}

We will apply Lemma \ref{lemma:combinatorial-identity}
with $g(n) = f(n^2) \e(\theta n)$
where $\theta\in\R$ and
$f$ is a strongly $q$-multiplicative function
(Definition~\ref{definition:q-multiplicative-function}).

\section{Sums of type II}\label{section:typeIIsums}

\subsection{Reduction  to exponential sums estimates}~

We set
\begin{equation}\label{eq:choice-beta_1}
  \beta_1=\frac15.
\end{equation}
Clearly we have $0<\beta_1<\frac13$,
and also set $\beta_2=1-\beta_1 = \frac 45$.
For positive integers $\mu$ and $\nu$ satisfying
\begin{equation}
  \label{eq:extended-condition-mu-nu}
  \beta_1 \leq \frac{\mu}{\mu+\nu} \leq \beta_2
  ,
\end{equation}
we intend to show in a quantitative way that the type II sum
\begin{equation}\label{eq:S20-final-upperbound}
  S_{20}(\theta)
  :=
  \sum_{q^{\mu-1} \leq m < q^\mu}
  \sum_{q^{\nu-1} \leq n < q^\nu}
  a_m b_n f(m^2n^2) \e(\theta m n)
  = o\left(q^{\mu+\nu}\right)
  ,
\end{equation}
uniformly for all  $\theta\in\R$ and
all complex numbers $a_m$ and $b_n$ of modulus at most $1$
(due to the generality of $a_m$ and $b_n$
and cutting $q$-adically this is sufficient to
establish \eqref{definition-sums-type-II}).

Notice that exchanging $\mu$ and $\nu$, it is enough to show this
upper bound in the range
\begin{equation}
  \label{eq:condition-mu_nu}
  \beta_1 \leq \frac{\mu}{\mu+\nu} \leq \frac12,
\end{equation}
which implies that
\begin{displaymath}
  \frac 14 \nu \le \mu \leq \nu.
\end{displaymath}
We will now assume that a slightly modified condition holds:
\begin{equation}
  \label{eq:inequality-mu-nu}
  \frac 14 \nu -C  \le \mu \leq \nu + C.
\end{equation}
for some absolute constant $C > 0$.

From 
Let $\rho$, $\kappa_1$, $\kappa_2$, $\lambda$ be integers such that
\begin{align}
  \label{eq:condition-rho}
  0 & < \rho < \mu/2,\\
  \label{eq:definition-kappa_1}
  \kappa_1 & = \mu-\rho,\\
  \label{eq:definition-kappa_2}
  \kappa_2 &= 2\mu+\nu+\rho+ \widetilde{\rho},\\
    \label{eq:definition-lambda}
    \lambda &= \kappa_2 - \kappa_1 = \mu+\nu+2\rho+ \widetilde{\rho}
    .
\end{align}
After a first application of Cauchy-Schwarz followed by
Lemma~\ref{lemma:van-der-corput}
(van der corput's inequality) with $N'=1$ and $R = q^\rho$,
we get
\begin{displaymath} \label{eqS20est}
  \abs{S_{20}(\theta)}^2
  \ll
  q^{2\mu+2\nu-\rho}
  +
  q^{\mu+\nu-\rho}
  \sum_{1\leq r < q^\rho} \left( 1-\frac{r}{q^\rho} \right)
  \ \Re\left(S_{21}(r,\theta)\right)
  ,
\end{displaymath}
where
\begin{displaymath}
  S_{21}(r,\theta)
  :=
  \sum_{q^{\nu-1} \leq n < q^\nu-r}
  b_{n+r} \conjugate{b_n}
  \sum_{q^{\mu-1} \leq m < q^\mu} 
  f(m^2(n+r)^2) \conjugate{f}(m^2n^2) \e(\theta m r)
  .
\end{displaymath}
Since
\begin{displaymath}
  \abs{S_{20}(\theta)}^4
  \ll
  q^{4\mu+4\nu-2\rho}
  +
  q^{2\mu+2\nu-\rho} \sum_{1\leq r < q^\rho} \ \abs{S_{21}(r,\theta)}^2
  ,
\end{displaymath}
after a second application of Cauchy-Schwarz followed
by Lemma~\ref{lemma:van-der-corput}
(van der corput's inequality) with $N'=q^{\kappa_1}$
and $R=q^\rho$, we get
\begin{multline}\label{eq:S20-to-S22}
  \abs{S_{20}(\theta)}^4
  \ll
  q^{4\mu+4\nu-2\rho}  + q^{4\mu+4\nu-\rho}
  \\
  +
  q^{3\mu+3\nu-2\rho}
  \sum_{1\leq r < q^\rho}
  \sum_{1\leq s < q^\rho} \left( 1-\frac{s}{q^\rho} \right)
  \ \Re\left( S_{22}(r,s) \e(\theta r s q^{\kappa_1}) \right)
  ,
\end{multline}
where, changing the conditions $q^{\nu-1} \leq n < q^\nu-r$ and
$q^{\mu-1} \leq m < q^\mu-s q^{\kappa_1}$ into
$q^{\nu-1} \leq n < q^\nu$ and $q^{\mu-1} \leq m < q^\mu$
at the cost of an admissible error term, we take
\begin{multline*}
  S_{22}(r,s)
  =
  \sum_{q^{\mu-1} \leq m < q^\mu}
  \sum_{q^{\nu-1} \leq n < q^\nu}
  f\left(\left(m+sq^{\kappa_1}\right)^2(n+r)^2\right)
  \conjugate{f}\left(m^2(n+r)^2\right)
  \\
  \times
  \conjugate{f}\left(\left(m+sq^{\kappa_1}\right)^2n^2\right)
  f\left(m^2n^2\right)
  .
\end{multline*}
By the carry propagation property described in Lemma~\ref{Le:carryproperty},
uniformly for $1\leq r <q^\rho$ and $1\leq s <q^\rho$
we can write
\begin{equation}\label{eq:S22-to-S23}
  S_{22}(r,s)
  =
  S_{23}(r,s)
  + O\left( q^{\mu+\nu-\widetilde{\rho}}\right)
  ,
\end{equation}
where
\begin{multline*}
  S_{23}(r,s)
  =
  \sum_{q^{\mu-1} \leq m < q^\mu}
  \sum_{q^{\nu-1} \leq n < q^\nu}
  f_{\kappa_1,\kappa_2}\left(\left(m+sq^{\kappa_1}\right)^2(n+r)^2\right)
  \conjugate{f_{\kappa_1,\kappa_2}}\left(m^2(n+r)^2\right)
  \\
  \times
  \conjugate{f_{\kappa_1,\kappa_2}}\left(\left(m+sq^{\kappa_1}\right)^2n^2\right)
  f_{\kappa_1,\kappa_2}\left(m^2n^2\right)
  .
\end{multline*}

Let $K$ be an integer such that
\begin{equation}\label{eq:initial-condition-for-K}
  q^\rho \leq K \leq q^{\kappa_1-1},
\end{equation}
to be chosen later more precisely,
and
given $\lambda$ defined by~\eqref{eq:definition-lambda}
choose
\begin{align}
  \label{eq:definition-H}
  H &= K q^{\lambda} - 1,\\
  \label{eq:definition-alpha-lambda}
  \alpha &= q^{-\lambda}
    .
\end{align}
Using \eqref{eq:definition-f-kappa1-kappa2-H}, let
\begin{multline*}
  S_{24}(r,s)
  =
  \sum_{q^{\mu-1} \leq m < q^\mu}
  \sum_{q^{\nu-1} \leq n < q^\nu}
  f_{\kappa_1,\kappa_2,H}\left(\left(m+sq^{\kappa_1}\right)^2(n+r)^2\right)
  \conjugate{f_{\kappa_1,\kappa_2,H}}\left(m^2(n+r)^2\right)
  \\
  \times
  \conjugate{f_{\kappa_1,\kappa_2,H}}\left(\left(m+sq^{\kappa_1}\right)^2n^2\right)
  f_{\kappa_1,\kappa_2,H}\left(m^2n^2\right)
  .
\end{multline*}

Since $\abs{f_{\kappa_1,\kappa_2}}\leq 1$,
we observe by \eqref{eq:vaaler-approximation-f_kappa1-kappa2}
that $\abs{f_{\kappa_1,\kappa_2,H}}\leq 2$.
Using the identity
\begin{displaymath}
  z_1z_2z_3z_4 - z'_1z'_2z'_3z'_4 
  =
   (z_1-z'_1)z'_2z'_3z'_4 + z_1(z_2-z'_2)z'_3z'_4
   + z_1z_2(z_3-z'_3)z'_4 + z_1z_2z_3(z_4-z'_4)
   ,
\end{displaymath}
and taking, for $0\leq r < q^\rho$ and $0\leq s < q^\rho$,
\begin{displaymath}
  E_{24}(r,s)
  =
  \sum_{q^{\mu-1} \leq m < q^\mu}
  \sum_{q^{\nu-1} \leq n < q^\nu}
  q^{\kappa_2-\kappa_1}
  \sum_{\abs{k}< K}
  \fourier{B_{\alpha,H}}\left(k q^{\kappa_2-\kappa_1}\right)
  \e\left( \frac{k \left(m+sq^{\kappa_1}\right)^2(n+r)^2}{q^{\kappa_1}} \right)
  ,
\end{displaymath}
by \eqref{eq:vaaler-approximation-f_kappa1-kappa2} we obtain
\begin{equation}\label{eq:approximation-S23-S24-initial}
  \abs{S_{23}(r,s) - S_{24}(r,s)}
  \leq
  8 E_{24}(r,s) + 4 E_{24}(r,0) + 2 E_{24}(0,s) + E_{24}(0,0)
  .
\end{equation}
For $k\neq 0$,
by inequality~\eqref{eq:double-exponential-sum-m^2n^2},
applied with $\xi_4$ replaced by $k q^{-\kappa_1}$, observing that
by \eqref{eq:initial-condition-for-K} we have
$\abs{k} q^{-\kappa_1} \leq \frac12$,
we obtain uniformly for $0\leq r < q^\rho$ and $0\leq s < q^\rho$,
\begin{multline*}
  \abs{
    q^{-\mu-\nu}
    \sum_{q^{\mu-1} \leq m < q^\mu}
    \sum_{q^{\nu-1} \leq n < q^\nu}
    \e\left(
      \frac{k \left(m+sq^{\kappa_1}\right)^2(n+r)^2}{q^{\kappa_1}}
    \right)
  }^4
  \\
  \ll
  \abs{k}^{\frac25} q^{-\frac25\kappa_1} 
  +
  q^{-\nu}
  +
  q^{-\mu} \log q^{\kappa_1} 
  +
  q^{-2\mu-2\nu}
  \abs{k}^{-1} q^{\kappa_1}  \log^3 q^{\kappa_1} 
  \\
  +
  \left(
    q^{-\mu-2\nu} \abs{k}^{-\frac35} q^{\frac35\kappa_1} 
    +
    q^{-2\mu-\nu} \abs{k}^{-\frac45} q^{\frac45\kappa_1} 
    +
    q^{-\mu-\nu} \abs{k}^{-\frac25} q^{\frac25\kappa_1} 
  \right)
  \log^3 q^{\kappa_1} 
  ,
\end{multline*}
which, using $\abs{k}\geq 1$, \eqref{eq:inequality-mu-nu}
and \eqref{eq:definition-kappa_1},
is
\begin{displaymath}
  \\
  \ll
  \abs{k}^{\frac25} q^{-\frac25\kappa_1}   
  .
\end{displaymath}
It follows that uniformly for $0\leq r < q^\rho$ and $0\leq s < q^\rho$,
\begin{displaymath}
  E_{24}(r,s)
  \ll
  q^{\kappa_2-\kappa_1}
  \fourier{B_{\alpha,H}}\left(0\right)
  q^{\mu+\nu}
  +
  q^{\kappa_2-\kappa_1}
  \sum_{1\leq\abs{k}< K}
  \fourier{B_{\alpha,H}}\left(k q^{\kappa_2-\kappa_1}\right)
  \abs{k}^{\frac1{10}} q^{\mu+\nu-\frac1{10}\kappa_1}
  ,
\end{displaymath}
which by \eqref{eq:vaaler-coef-majoration},
\eqref{eq:definition-H} and \eqref{eq:definition-lambda}
gives uniformly for $0\leq r < q^\rho$ and $0\leq s < q^\rho$,
\begin{displaymath}
  E_{24}(r,s)
  \ll
  q^{\mu+\nu}
  \left(K^{-1} 
  +
  K^{\frac1{10}} q^{-\frac1{10}\kappa_1}
  \right)
  .
\end{displaymath}
It follows from \eqref{eq:approximation-S23-S24-initial} that
\begin{equation}\label{eq:approximation-S23-S24}
  \abs{S_{23}(r,s) - S_{24}(r,s)}
  \ll
  q^{\mu+\nu}
  \left(K^{-1} 
  +
  K^{\frac1{10}} q^{-\frac1{10}\kappa_1}
  \right)
  .
\end{equation}

By \eqref{eq:definition-f-kappa1-kappa2-H}, we have
\begin{multline*}
  S_{24}(r,s)
  =
  q^{4\lambda}
  \sum_{\abs{h_1}\leq H}
  \sum_{\abs{h_2}\leq H}
  \sum_{\abs{h_3}\leq H}
  \sum_{\abs{h_4}\leq H}
  \fourier{\chi_{\alpha,H}}(h_1)
  \conjugate{\fourier{\chi_{\alpha,H}}}(h_2)
  \conjugate{\fourier{\chi_{\alpha,H}}}(h_3)
  \fourier{\chi_{\alpha,H}}(h_4)
  \\
  F_{\kappa_1,\kappa_2}(h_1)
  \conjugate{F_{\kappa_1,\kappa_2}}(h_2)
  \conjugate{F_{\kappa_1,\kappa_2}}(h_3)
  F_{\kappa_1,\kappa_2}(h_4)
  \
  S_{25}(r,s,h_1,h_2,h_3,h_4)
  ,
\end{multline*}
with
\begin{multline*}
  S_{25}(r,s,h_1,h_2,h_3,h_4)
  =
  \\
  \sum_{q^{\mu-1} \leq m < q^\mu} \sum_{q^{\nu-1} \leq n < q^\nu}
  \e_{q^{\kappa_2}}\left(
    h_1\left(m+sq^{\kappa_1}\right)^2 (n+r)^2
    - h_2 m^2 (n+r)^2
    - h_3 \left(m+sq^{\kappa_1}\right)^2 n^2
    + h_4 m^2 n^2
  \right)
  .
\end{multline*}
Let us write
\begin{equation}
  \label{eq:S24=S26+S27+S28}
  S_{24}(r,s) = S_{26}(r,s) + S_{27}(r,s) + S_{28}(r,s)
  ,
\end{equation}
where $S_{26}(r,s)$ denotes
the contribution in $S_{24}(r,s)$ of the terms such that
$h_1-h_2-h_3+h_4 \neq 0$, 
$S_{27}(r,s)$ denotes the contribution
in $S_{24}(r,s)$ of the terms such that
$h_1-h_2-h_3+h_4 = 0$ and $\abs{h_1-h_2} \geq q^{\rho_3}$,
and
$S_{28}(r,s)$ denotes the contribution
in $S_{24}(r,s)$ of the terms such that
$h_1-h_2-h_3+h_4 = 0$ and $\abs{h_1-h_2} < q^{\rho_3}$,
for some integer $\rho_3>0$ to be chosen later.

\begin{lemma}\label{Lemainestimates}
  Suppose that $q$ is a prime number.
  Then we have the following upper bounds:
  \begin{align}\label{eq:S26-upperbound}
    q^{-\mu-\nu-2\rho}
    & \sum_{1\leq r < q^\rho} \sum_{1\leq s < q^\rho} 
      \abs{S_{26}(r,s)}
      \ll
      \rho_1^4
      \left(
      \sum_{0\leq \ell< q^{\kappa_2-\kappa_1}}
      \abs{F_{\kappa_1,\kappa_2}(\ell)}
      \right)^4 
    \\
    & \qquad \times
      \left(
      q^{-\frac1{10}(\mu-\rho-\rho_1)}
      +
      q^{-\frac14\mu} \log^{\frac14}\left(q^{\mu+\nu}\right)
      +
      q^{-\frac1{20}(\mu+\nu)}
      \log^{\frac34}\left(q^{\mu+\nu}\right)
      \right)
      ,  \nonumber
  \end{align}
  \begin{align}\label{eq:S27-upperbound}
    q^{-\mu-\nu-2\rho}
    & \sum_{1\leq r < q^\rho} \sum_{1\leq s < q^\rho} 
      \abs{S_{27}(r,s)}
      \ll
      \rho_1^2
      \left(
      \sum_{0\leq \ell< q^{\kappa_2-\kappa_1}}
      \abs{F_{\kappa_1,\kappa_2}(\ell)}
      \right)^2 \\  \nonumber
    & \qquad \times
      \left(
      q^{-\frac 14(\rho_3-\widetilde{\rho}) }
      + q^{-\frac 16(\mu-\rho -\widetilde{\rho}-\rho_1)} 
      + q^{-\frac 14 \mu}
      \right)
      ,
  \end{align}
  and
  \begin{align}\label{eq:S28-upperbound}
    q^{-\mu-\nu-2\rho}
    &\sum_{1\leq r < q^\rho} \sum_{1\leq s < q^\rho} 
      \abs{S_{28}(r,s)}
      \ll
      \sum_{0\le k < q^{\rho_3}}
      \abs{F_{\kappa_2-\rho_3,\kappa_2}(k)}  
    \\ \nonumber
    & \qquad \times
      \Biggl( 
      \rho_1 q^{\frac 12 \rho_3}
      \Biggl( 
      \sum_{0\leq \ell < q^{\kappa_2-\kappa'_1}}
      \abs{
      F_{\kappa'_1,\kappa_2}(\ell)
      }
      \Biggr) 
    \\ \nonumber
    &\qquad \qquad 
      \times \left(
      q^{-\frac18(\mu-\rho-\rho_1-\rho_2)} 
      +
      \left(q^{-\frac14(\nu-3\rho)} + q^{-\frac14\mu}\right)(\mu+\nu)
      \right)
    \\  \nonumber
    &\qquad \qquad 
      + q^{\frac 12\rho_3}
      \left(
      q^{-\frac 12 \rho_1} 
      +
      q^{-\frac18(\mu-\rho-\rho_1-\rho_2)}
      \right) 
    \\  \nonumber
    &\qquad \qquad 
      + \sqrt{\rho \rho_2}
      q^{\frac 12 \widetilde{\rho}} 
      \max_{h_1} \abs{F_{\kappa'_2,\kappa_2}\left(h_1\right)}
    \\  \nonumber
    &\qquad \qquad  
      +\sqrt{\kappa_2-\kappa'_1}
      q^{\rho_3-\frac18\rho_5}
      +  q^{\frac 12 (\rho_3-\rho_2)} 
      +  q^{\rho_3-\rho_1}  
      \Biggr)
      ,
  \end{align}
  where
  \begin{align*}
    \kappa_1 &= \mu - \rho, \\
    \kappa_2 &= 2\mu + \nu + \rho + \widetilde{\rho}, \\
    \kappa_1' &=  2\mu- \rho -\rho_2, \\
    \kappa_2' &= 2\mu - \rho - \rho_2 - \rho_5,
  \end{align*}
  and where the appearing parameters
  \begin{displaymath}
    \rho,\, \widetilde{\rho},\, \rho_1,\, \rho_2\, \rho_3,\, \rho_5
  \end{displaymath}
  are positive and satisfy
  \begin{displaymath}
    \rho < \frac{\mu}8, \quad
    \widetilde{\rho} < \frac{\mu}8,\quad
    \rho < \rho_1 < \mu-3\rho, \quad
    \rho+\rho_1+ \rho_2 + 3 \le \mu.
  \end{displaymath}
\end{lemma}

\medskip
\subsection{Upper bounds for $S_{26}(r,s)$}~

If $h_1-h_2-h_3+h_4 \neq 0$ we write
\begin{displaymath}
  \xi_4 = (h_1-h_2-h_3+h_4) \, q^{-\kappa_2} 
  .
\end{displaymath}
By inequality~\eqref{eq:double-exponential-sum-m^2n^2}
we get the upper bound
\begin{multline*}
  \abs{q^{-\mu-\nu} S_{25}(r,s,h_1,h_2,h_3,h_4)}^4
  \\
  \ll
  \norm{\xi_4}^{2/5}
  +
  q^{-\nu}
  +
  q^{-\mu} \log \norm{\xi_4}^{-1}
  +
  q^{-2\mu-2\nu}
  \norm{\xi_4}^{-1} \log^3 \norm{\xi_4}^{-1}
  \\
  +
  \left(
    q^{-\mu-2\nu} \norm{\xi_4}^{-3/5} 
    +
    q^{-2\mu-\nu} \norm{\xi_4}^{-4/5} 
    +
    q^{-\mu-\nu} \norm{\xi_4}^{-2/5}
  \right)
  \log^3 \norm{\xi_4}^{-1}    
  .
\end{multline*}
By \eqref{eq:definition-H}, \eqref{eq:definition-lambda}
and \eqref{eq:definition-kappa_1} we have
\begin{displaymath}
  q^{-2\mu-\nu-\rho-\widetilde{\rho}}
  =
  q^{-\kappa_2}
  \leq \abs{\xi_4}
  \leq 4 H q^{-\kappa_2}
  \leq 4 K q^{\lambda-\kappa_2}
  = 4 K q^{-\kappa_1}  
  \asymp  K q^{-\mu+\rho}
  ,
\end{displaymath}
which gives that
\begin{math}
  \abs{q^{-\mu-\nu} S_{25}(r,s,h_1,h_2,h_3,h_4)}^4
\end{math}
is
\begin{multline*}
  \ll
  (K q^{-\mu+\rho})^{\frac25}
  +
  q^{-\nu}
  +
  q^{-\mu} \log\left(q^{\mu+\nu}\right)
  \\
  +
  \left(
    q^{-\nu+2\rho} 
    +
    q^{-\mu-2\nu+\frac35(2\mu+\nu+\rho+ \widetilde{\rho})} 
    +
    q^{-2\mu-\nu+\frac45(2\mu+\nu+\rho)+ \widetilde{\rho}}
    +
    q^{-\mu-\nu+\frac25(2\mu+\nu+\rho + \widetilde{\rho})} \right)
  \log^3\left(q^{\mu+\nu}\right)
\end{multline*}
\begin{multline*}
  \ll
  (K q^{-\mu+\rho})^{\frac25}
  +
  q^{-\mu} \log\left(q^{\mu+\nu}\right)
  \\
  +
  \left(
    q^{-\nu+\rho+\widetilde{\rho}} 
    +
    q^{\frac15(\mu-7\nu+3\rho+3\widetilde{\rho})} 
    +
    q^{\frac15(-2\mu-\nu+4\rho+4\widetilde{\rho})}
    +
    q^{\frac15(-\mu-3\nu+2\rho+2\widetilde{\rho})} \right)
  \log^3\left(q^{\mu+\nu}\right)
\end{multline*}
and, using \eqref{eq:condition-mu_nu},
observing that
\begin{displaymath}
  \mu-7\nu
  \leq \frac12(\mu+\nu) - \frac72(\mu+\nu)
  = - 3(\mu+\nu)
\end{displaymath}
we get for
\begin{math}
  \abs{q^{-\mu-\nu} S_{25}(r,s,h_1,h_2,h_3,h_4)}^4
\end{math}
the upper bound
\begin{multline*}
  \ll
  (K q^{-\mu+\rho})^{\frac25}
  +
  q^{-\mu} \log\left(q^{\mu+\nu}\right)
  \\
  +
  \left(
    q^{-\frac12(\mu+\nu)+\rho+\widetilde{\rho}} 
    +
    q^{-\frac35(\mu+\nu)+\frac35\rho +\frac35 \widetilde{\rho}} 
    +
    q^{\frac15(-2\mu-\nu+4\rho+4\widetilde{\rho})}
    +
    q^{\frac15(-\mu-3\nu+2\rho+2\widetilde{\rho})} \right)
  \log^3\left(q^{\mu+\nu}\right)
\end{multline*}
which is satisfactory provided that $K$, $\rho$
and $\widetilde{\rho}$ are small enough.
Recall that by assumption 
\begin{equation}\label{eq:condition-rho-restricted}
  0<\rho<\frac{\mu}{8}, \quad 0<\widetilde{\rho}<\frac{\mu}{8}
\end{equation}
which is actually stronger than \eqref{eq:condition-rho} and 
which by \eqref{eq:condition-mu_nu} implies
\begin{math}
  \rho < \frac{1}{16}(\mu+\nu),
\end{math} 
\begin{math}
  \widetilde{\rho} < \frac{1}{16}(\mu+\nu).
\end{math}
We also set
\begin{equation}\label{eq:definition-K}
  K = q^{\rho_1}
\end{equation}
so that by assumption
\begin{equation}
  \label{eq:condition-rho_1}
  \rho < \rho_1 < \mu-3\rho
  ,
\end{equation}
which replace \eqref{eq:initial-condition-for-K},
to be chosen later in order to optimize with the
error term introduced by \eqref{eq:vaaler-approximation-f_kappa1-kappa2},
we get for
\begin{math}
  \abs{q^{-\mu-\nu} S_{25}(r,s,h_1,h_2,h_3,h_4)}^4
\end{math}
the upper bound
\begin{displaymath}
  \ll
  q^{-\frac25(\mu-\rho-\rho_1)}
  +
  q^{-\mu} \log\left(q^{\mu+\nu}\right)
  +
  \left(
    q^{-\frac{3}{8}(\mu+\nu)} 
    +
    q^{-\frac{21}{40}(\mu+\nu)} 
    +
    q^{-\frac15(\mu+\nu)}
    +
    q^{-\frac{7}{20}(\mu+\nu)} \right)
  \log^3\left(q^{\mu+\nu}\right)
  ,
\end{displaymath}
and finally
\begin{displaymath}
  \abs{q^{-\mu-\nu} S_{25}(r,s,h_1,h_2,h_3,h_4)}^4
  \ll
  q^{-\frac25(\mu-\rho-\rho_1)}
  +
  q^{-\mu} \log\left(q^{\mu+\nu}\right)
  +
  q^{-\frac15(\mu+\nu)}
  \log^3\left(q^{\mu+\nu}\right)
  .
\end{displaymath}
It follows that for $h_1-h_2-h_3+h_4 \neq 0$ we get
\begin{multline*}
  q^{-2\rho}
  \sum_{1\leq r < q^\rho} \sum_{1\leq s < q^\rho} 
  \abs{ q^{-\mu-\nu} S_{25}(r,s,h_1,h_2,h_3,h_4) }
  \\
  \ll
  q^{-\frac1{10}(\mu-\rho-\rho_1)}
  +
  q^{-\frac14\mu} \log^{\frac14}\left(q^{\mu+\nu}\right)
  +
  q^{-\frac1{20}(\mu+\nu)}
  \log^{\frac34}\left(q^{\mu+\nu}\right)
  .
\end{multline*}
By the definition of $S_{26}(r,s)$ given by \eqref{eq:S24=S26+S27+S28},
it follows that
\begin{multline*}
  q^{-\mu-\nu-2\rho}
  \sum_{1\leq r < q^\rho} \sum_{1\leq s < q^\rho} 
  \abs{S_{26}(r,s)}
  \\
  \ll
  \left(
    q^\lambda
    \sum_{\abs{h}\leq H}
    \abs{\fourier{\chi_{\alpha,H}}(h)}
    \abs{F_{\kappa_1,\kappa_2}(h)}
  \right)^4
  \left(
    q^{-\frac1{10}(\mu-\rho-\rho_1)}
    +
    q^{-\frac14\mu} \log^{\frac14}\left(q^{\mu+\nu}\right)
    +
    q^{-\frac1{20}(\mu+\nu)}
    \log^{\frac34}\left(q^{\mu+\nu}\right)
  \right)
  ,
\end{multline*}
and by \eqref{eq:L1-mean-chi_H-F_kappa1-kappa2} we
directly obtain (\ref{eq:S26-upperbound}).

\medskip
\subsection{Upper bounds for $S_{27}(r,s)$}~

If $h_1-h_2-h_3+h_4 = 0$, then
\begin{multline*}
  S_{25}(r,s,h_1,h_2,h_3,h_4)
  =
  S_{25}(r,s,h_1,h_2,h_3,-h_1+h_2+h_3)
  \\
  =
  \sum_m \sum_n
  \e_{q^{\kappa_2}}\Big(
    (h_1 - h_2)m^2(2nr+r^2)
    +
    (h_1 - h_3)\left(2msq^{\kappa_1}+s^2q^{2\kappa_1}\right) n^2
    \\
    +
    h_1\left(2msq^{\kappa_1}+s^2q^{2\kappa_1}\right)(2nr+r^2)
  \Big).
\end{multline*}
Therefore
if $h_1-h_2-h_3+h_4 = 0$ and $h_1-h_2 \neq 0$
taking
\begin{displaymath}
  \xi_2 = 2 r (h_1-h_2) q^{-\kappa_2} 
  ,
\end{displaymath}
by \eqref{eq:definition-H}, \eqref{eq:definition-lambda}
and \eqref{eq:definition-kappa_1}
we have
\begin{displaymath}
  r \abs {h_1-h_2} q^{-2\mu-\nu-2\rho}
  \ll
  \abs{\xi_2}
  \leq 2 H q^{\rho-\kappa_2}
  \leq 2 K q^{\rho+\lambda-\kappa_2}
  = 2 K q^{\rho-\kappa_1}
  \asymp K q^{-\mu+\rho+\widetilde{\rho}}
  ,
\end{displaymath}
and using inequality~\eqref{eq:double-exponential-sum-xi_2}
we get the upper bound
\begin{multline*}
  \abs{
    q^{-\mu-\nu} S_{25}(r,s,h_1,h_2,h_3,-h_1+h_2+h_3)
  }^2
  \ll
  \left( K q^{-\mu+\rho+\widetilde{\rho}} \right)^{1/3}
  +
  \frac{q^{\frac 12\rho + \frac 12 \widetilde{\rho}}}{r^{1/2} \abs{h_1-h_2}^{1/2}}
  +
  \frac{1}{q^{\mu/2}}
  +
  \frac{1}{q^\nu}
  ,
\end{multline*}
which leads to the upper bound
\begin{multline*}
  q^{-2\rho}
  \sum_{1\leq r < q^\rho} \sum_{1\leq s < q^\rho} 
  \abs{
    q^{-\mu-\nu}
    S_{25}(r,s,h_1,h_2,h_3,-h_1+h_2+h_3)
  }
  \\
  \ll
  \left( K q^{-\mu+\rho+\widetilde{\rho}} \right)^{1/6}
  +
  \frac{q^{\widetilde{\rho}/4}}{ \abs{h_1-h_2}^{1/4}}
  +
  \frac{1}{q^{\mu/4}}
  +
  \frac{1}{q^{\nu/2}}
  ,
\end{multline*}
so, if, say, $\abs{h_1-h_2} \ge q^{\rho_3}$,
for some $\rho_3$ to be chosen later,
using \eqref{eq:definition-K} and \eqref{eq:inequality-mu-nu},
we get the upper bound
\begin{displaymath}
  \ll
  q^{-\frac16(\mu-\rho-\widetilde{\rho}-\rho_1)} 
  +
  q^{-\frac14(\rho_3- \widetilde{\rho})}
  +
  q^{-\frac14\mu}
  +
  q^{-\frac12\nu}
  \ll
  q^{-\rho_4}
  ,
\end{displaymath}
with
\begin{equation}
  \label{eq:definition-rho_4}
  \rho_4
  =
  \min\left(
    \tfrac14(\rho_3-\widetilde{\rho}),
    \tfrac16(\mu-\rho -\widetilde{\rho}-\rho_1),
    \tfrac14\mu
  \right).
\end{equation}
Note that this requires the relation
\begin{displaymath}
  \rho_3 > \widetilde{\rho}.
\end{displaymath}
By definition,
\begin{multline*}
  S_{27}(r,s)
  =
  q^{4\lambda}
  \sum_{\substack{\abs{h_1}\leq H,\ \abs{h_2}\leq H\\ \abs{h_1-h_2} \ge q^{\rho_3}}}
  \sum_{\abs{h_3}\leq H}
  \fourier{\chi_{\alpha,H}}(h_1)
  \conjugate{\fourier{\chi_{\alpha,H}}}(h_2)
  \conjugate{\fourier{\chi_{\alpha,H}}}(h_3)
  \fourier{\chi_{\alpha,H}}(-h_1+h_2+h_3)
  \\
  F_{\kappa_1,\kappa_2}(h_1) \conjugate{F_{\kappa_1,\kappa_2}}(h_2)
  \conjugate{F_{\kappa_1,\kappa_2}}(h_3) F_{\kappa_1,\kappa_2}(-h_1+h_2+h_3)
  \\
  S_{25}(r,s,h_1,h_2,h_3,-h_1+h_2+h_3)
  ,
\end{multline*}
where we notice that,
by \eqref{eq:vaaler-coef-chi-B} and \eqref{eq:vaaler-coef-chi-star-H}, 
$\fourier{\chi_{\alpha,H}}(-h_1+h_2+h_3) = 0$
if $\abs{-h_1+h_2+h_3}> H$.
This takes care of the condition $\abs{-h_1+h_2+h_3}\leq H$
coming from $\abs{h_4}\leq H$.
It follows that
\begin{multline*}
  q^{-\mu-\nu-2\rho}
  \sum_{1\leq r < q^\rho} \sum_{1\leq s < q^\rho} 
  \abs{S_{27}(r,s)}
  \\
  \ll
  q^{-\rho_4}
  q^{2\lambda}
  \sum_{\abs{h_1}\leq H}
  \sum_{\abs{h_2}\leq H}
  \abs{\fourier{\chi_{\alpha,H}}(h_1) F_{\kappa_1,\kappa_2}(h_1)}
  \abs{\fourier{\chi_{\alpha,H}}(h_2) F_{\kappa_1,\kappa_2}(h_2)}
  \\
  q^{2\lambda}
  \sum_{\abs{h_3}\leq H}
  \abs{
    \fourier{\chi_{\alpha,H}}(h_3)
    \fourier{\chi_{\alpha,H}}(-h_1+h_2+h_3)
    F_{\kappa_1,\kappa_2}(h_3)
    F_{\kappa_1,\kappa_2}(-h_1+h_2+h_3)
  }
  ,
\end{multline*}
and since by Cauchy-Schwarz
and \eqref{eq:L2-mean-chi_H-F_kappa1-kappa2},
\begin{align*}
  &
    q^{2\lambda}
    \sum_{\abs{h_3}\leq H}
    \abs{
    \fourier{\chi_{\alpha,H}}(h_3)
    \fourier{\chi_{\alpha,H}}(-h_1+h_2+h_3)
    F_{\kappa_1,\kappa_2}(h_3)
    F_{\kappa_1,\kappa_2}(-h_1+h_2+h_3)
    }
  \\
  & \qquad
    \leq
    \left(
    q^{2\lambda}
    \sum_{h_3}
    \abs{\fourier{\chi_{\alpha,H}}(h_3) F_{\kappa_1,\kappa_2}(h_3)}^2
    \right)^{1/2}
  \\
  & \qquad\qquad\qquad
  \left(
    q^{2\lambda}
    \sum_{h_3}
    \abs{
      \fourier{\chi_{\alpha,H}}(-h_1+h_2+h_3)
      F_{\kappa_1,\kappa_2}(-h_1+h_2+h_3)
    }^2 
  \right)^{1/2}
  \ll 1,
\end{align*}
we obtain by \eqref{eq:L1-mean-chi_H-F_kappa1-kappa2} the 
proposed upper bound (\ref{eq:S27-upperbound}).

\medskip
\subsection{Upper bounds for $S_{28}(r,s)$}~

We recall that $S_{28}(r,s)$ corresponds to
the case $h_2 = h_1 - \ell$ and $h_4 = h_3 - \ell$
with $\abs{\ell} < q^{\rho_3}$.

As we have seen in the previous section the case, where $h_1 - h_2 = h_3-h_4$
is small, cannot be handled directly by applying estimates for
double exponential sums.
The idea is to use the Fourier transform forth and back for those
$h_1,h_2,h_3,h_4$.
We have
\begin{displaymath}
  S_{28}(r,s)
  =
  \sum_{\abs{\ell}<q^{\rho_3}}   S_{\text{diag}}(r,s,\ell)
\end{displaymath}
with
\begin{multline*}
  S_{\text{diag}}(r,s,\ell)
  =
  q^{4\lambda}
  \sum_{\abs{h_1}\leq H}
  \fourier{\chi_{\alpha,H}}(h_1)
  \conjugate{\fourier{\chi_{\alpha,H}}}(h_1-\ell)
   F_{\kappa_1,\kappa_2}(h_1) \conjugate{F_{\kappa_1,\kappa_2}}(h_1-\ell)
   \\
   \sum_{\abs{h_3}\leq H}
    \conjugate{\fourier{\chi_{\alpha,H}}}(h_3)
  \fourier{\chi_{\alpha,H}}(h_3-\ell)
    \conjugate{F_{\kappa_1,\kappa_2}}(h_3)F_{\kappa_1,\kappa_2}(h_3-\ell)
  \\ 
  \sum_m \sum_n
  \e_{q^{\kappa_2-\kappa_1}}\Big(
  (h_1 - h_3)\left(2ms+s^2q^{\kappa_1}\right) n^2
  +
  h_1\left(2ms+s^2q^{\kappa_1}\right)(2nr+r^2)
  \Big) \\
   \e_{q^{\kappa_2}} \Big( \ell m^2(2nr+r^2) \Bigl)
  .
\end{multline*}
Notice that,
by \eqref{eq:vaaler-coef-chi-B} and \eqref{eq:vaaler-coef-chi-star-H},
$\fourier{\chi_{\alpha,H}}(h_1-\ell) = 0$
if $\abs{h_1-\ell}> H$.
This takes care of the condition $\abs{h_1-\ell}\leq H$.
The same applies to the condition $\abs{h_3-\ell}\leq H$.

By \eqref{eq:vaaler-coef-chi-B}, we have
\begin{multline*}
  \fourier{\chi_{\alpha,H}}(h_1)
  \conjugate{\fourier{\chi_{\alpha,H}}}(h_1-\ell)
  \conjugate{\fourier{\chi_{\alpha,H}}}(h_3)
  \fourier{\chi_{\alpha,H}}(h_3-\ell)
  \\
  = 
  \fourier{\chi_{\alpha,H}^*}(h_1)
  \conjugate{\fourier{\chi_{\alpha,H}^*}}(h_1-\ell)
  \conjugate{\fourier{\chi_{\alpha,H}^*}}(h_3)
  \fourier{\chi_{\alpha,H}^*}(h_3-\ell)
  \e\left(\frac{\alpha}2\left(-h_1+(h_1-\ell)+h_3-(h_3-\ell)\right)\right)
  \\
  = 
  \fourier{\chi_{\alpha,H}^*}(h_1)
  \conjugate{\fourier{\chi_{\alpha,H}^*}}(h_1-\ell)
  \conjugate{\fourier{\chi_{\alpha,H}^*}}(h_3)
  \fourier{\chi_{\alpha,H}^*}(h_3-\ell)
  ,
\end{multline*}
hence we can replace  $\chi_{\alpha,H}$ by  $\chi_{\alpha,H}^*$ in the
above sum,
which is more convenient
since $\fourier{\chi_{\alpha,H}^*}$ is real valued
(see~\eqref{eq:vaaler-coef-chi-star-H}):
\begin{multline*}
  S_{\text{diag}}(r,s,\ell)
  =
  q^{2\lambda}
  \sum_{\abs{h_1}\leq H}
  \fourier{\chi_{\alpha,H}^*}(h_1)
  \conjugate{\fourier{\chi_{\alpha,H}^*}}(h_1-\ell)
  F_{\kappa_1,\kappa_2}(h_1) \conjugate{F_{\kappa_1,\kappa_2}}(h_1-\ell)
  \\
  \sum_m \sum_n
  T(r,s,m,n,\ell)
  \e_{q^{\kappa_2-\kappa_1}}\Big(
  h_1 \left(2ms+s^2q^{\kappa_1}\right) n^2
  +
  h_1\left(2ms+s^2q^{\kappa_1}\right)(2nr+r^2)
  \Big) \\
   \e_{q^{\kappa_2}} \Big( \ell m^2(2nr+r^2) \Bigl)
  ,
\end{multline*}
where $T(r,s,m,n,\ell)$ takes care of the summation over $h_3$:
\begin{displaymath}
  T(r,s,m,n,\ell)
  := q^{2\lambda} 
  \sum_{\abs{h_3}\leq H}
  \conjugate{\fourier{\chi_{\alpha,H}^*}}(h_3)
  \fourier{\chi_{\alpha,H}^*}(h_3-\ell)
  \conjugate{F_{\kappa_1,\kappa_2}}(h_3)
  F_{\kappa_1,\kappa_2}(h_3-\ell)
  \e_{q^{\lambda}}\left(- h_3\left(2ms+s^2q^{\kappa_1}\right) n^2 \right)
  .
\end{displaymath}
By the definition of $F_{\kappa_1,\kappa_2}$ given
in~\eqref{eq:definition-F_kappa1-kappa2}:
\begin{displaymath}
 \conjugate{F_{\kappa_1,\kappa_2}}(h_3)
  F_{\kappa_1,\kappa_2}(h_3-\ell)
  =
    \frac 1{q^{2\lambda}}
  \sum_{0\le u_1,u_2 < q^{\lambda}}
  \conjugate{f_{\kappa_1,\kappa_2}}\left(q^{\kappa_1} u_1\right)
  f_{\kappa_1,\kappa_2}\left(q^{\kappa_1} u_2\right)
  \e_{q^{\lambda}}\left( h_3\left( u_1 - u_2  \right) + \ell u_2 \right)
  ,
\end{displaymath}  
and since
\begin{displaymath}
  \chi_{\alpha,H}^* * (\chi_{\alpha,H}^* \e^\ell )
  \left(\frac{u}{q^{\lambda}}\right)
  =
  \sum_{\abs{h}\leq H} 
  \conjugate{\fourier{\chi_{\alpha,H}^*}}(h)
  \fourier{\chi_{\alpha,H}^*}(h-\ell) 
  \e\left( \frac{h u}{q^{\lambda}} \right)
  ,
\end{displaymath}
$T(r,s,m,n,\ell)$ is equal to
\begin{displaymath}
  \sum_{0\le u_1,u_2 < q^{\lambda}}
  \conjugate{f_{\kappa_1,\kappa_2}}\left(q^{\kappa_1} u_1\right)
  f_{\kappa_1,\kappa_2}\left(q^{\kappa_1} u_2\right) 
  \e_{q^\lambda}(\ell u_2)
  \
  \chi_{\alpha,H}^* * (\chi_{\alpha,H}^* \e^\ell ) 
   \left(
    \frac{u_1-u_2-\left(2ms+s^2q^{\kappa_1}\right) n^2}{q^{\lambda}}
  \right)
  .
\end{displaymath}  
Let
\begin{align*}
  &T_1(r,s,m,n,\ell) \\
  &  =
    \sum_{0\le u_1,u_2 < q^{\lambda}}
    \conjugate{f_{\kappa_1,\kappa_2}}\left(q^{\kappa_1} u_1\right)
    f_{\kappa_1,\kappa_2}\left(q^{\kappa_1} u_2\right) 
    \e_{q^\lambda}(\ell u_2)
    \
    \chi_{\alpha}^* * (\chi_{\alpha}^* \e^\ell )  
    \left(
    \frac{u_1-u_2-\left(2ms+s^2q^{\kappa_1}\right) n^2}{q^{\lambda}}
    \right)
    .
\end{align*}
We have
\begin{multline*}
  \abs{T(r,s,m,n,\ell)-T_1(r,s,m,n,\ell)}
  \\
  \leq
  q^\lambda
  \sum_{0\le u < q^\lambda}
  \abs{
   \chi_{\alpha,H}^* *  (\chi_{\alpha,H}^* \e^\ell)
   \left(\frac{u}{q^\lambda}\right)
    -
    \chi_{\alpha}^* *  (\chi_{\alpha}^* \e^\ell ) 
   \left(\frac{u}{q^\lambda}\right)
  }  
  ,
\end{multline*}
hence by Lemma~\ref{lemma:chi_H-convolution-chi_H-2} and
\eqref{eq:definition-H}, 
\begin{displaymath}
  \abs{T(r,s,m,n,\ell)-T_1(r,s,m,n,\ell)}
  \le \frac{3 q^\lambda}{H+1} 
   =  \frac{3}{K}
  .
\end{displaymath}
Replacing $T(r,s,m,n,\ell)$ by $T_1(r,s,m,n,\ell)$
in $S_{\text{diag}}(r,s,\ell)$
will produce an error term
\begin{displaymath}
  \ll
  q^{2\lambda}
  \sum_{\abs{h_1}\leq H}
  \abs{
  \fourier{\chi_{\alpha,H}^*}(h_1)
  \conjugate{\fourier{\chi_{\alpha,H}^*}}(h_1-\ell)
  F_{\kappa_1,\kappa_2}(h_1) \conjugate{F_{\kappa_1,\kappa_2}}(h_1-\ell)
  }
  \sum_m \sum_n  \frac3K
\end{displaymath}
which is, by Cauchy-Schwarz,
\begin{displaymath}
    \ll
    K^{-1} q^{\mu+\nu}
    q^{2\lambda}
    \left(
      \sum_{\abs{h_1}\leq H}
      \abs{
        \fourier{\chi_{\alpha,H}^*}(h_1)
        F_{\kappa_1,\kappa_2}(h_1) 
      }^2
    \right)^{1/2}
    \left(
      \sum_{\abs{h_1}\leq H}
      \abs{
        \conjugate{\fourier{\chi_{\alpha,H}^*}}(h_1-\ell)
        \conjugate{F_{\kappa_1,\kappa_2}}(h_1-\ell)
      }^2
    \right)^{1/2}
    ,
\end{displaymath}
which by~\eqref{eq:L2-mean-chi_H-F_kappa1-kappa2},
is $O\left( K^{-1} q^{\mu+\nu} \right)$.
It follows that
\begin{multline*}
  S_{\text{diag}}(r,s,\ell)
  = q^{2	\lambda}
  \sum_{\abs{h_1}\leq H}
  \fourier{\chi_{\alpha,H}^*}(h_1)
  \conjugate{\fourier{\chi_{\alpha,H}^*}}(h_1-\ell)
  F_{\kappa_1,\kappa_2}(h_1) \conjugate{F_{\kappa_1,\kappa_2}}(h_1-\ell)
  \\ 
   \times 
  \sum_m \sum_n  T_1(r,s,m,n,\ell)
  \e_{q^{\lambda}}
  \Big(
      h_1 
       \left( 2ms+s^2q^{\kappa_1} \right) 
       (n+r)^2
    \Big) 
    \e_{q^{\kappa_2}} \Big( \ell m^2(2nr+r^2) \Bigl)
  \\
  + O\left( K^{-1} q^{\mu+\nu} \right)
    .
\end{multline*}
By \eqref{eq:definition-alpha-lambda}and
\eqref{eq:chi-star-convolution-chi-star-non-zero-2},
for
\begin{math}
  u_1 \not\equiv u_2 + \left(2ms+s^2q^{\kappa_1}\right) n^2 \bmod q^\lambda
  ,
\end{math}
we have
\begin{displaymath}
  \chi_{\alpha}^* * (\chi_{\alpha}^* \e^\ell )  
  \left(
    \frac{u_1-u_2-\left(2ms+s^2q^{\kappa_1}\right) n^2}{q^{\lambda}}
  \right)
  = 0
  .
\end{displaymath}
Hence only the term
\begin{math}
  u_1 \equiv u_2 + \left(2ms+s^2q^{\kappa_1}\right) n^2 \bmod q^\lambda
\end{math}
contributes to $T_1(r,s,m,n,\ell)$
and, recalling by \eqref{eq:chi-star-convolution-chi-star-of-zero-2}
that
\begin{math}
  \chi_{\alpha}^* * (\chi_{\alpha}^* \e^\ell )  (0)
  =
  \fourier{\chi_\alpha^*}(\ell),
\end{math}
we have
\begin{displaymath}
  T_1(r,s,m,n,\ell)
  =
  \fourier{\chi_\alpha^*}(\ell)
  \sum_{0\le u_2 < q^{\lambda}}
  \conjugate{f_{\kappa_1,\kappa_2}}\left(
    q^{\kappa_1} u_2 + q^{\kappa_1}\left(2ms+s^2q^{\kappa_1}   \right) n^2
  \right)
  f_{\kappa_1,\kappa_2}  \left( q^{\kappa_1} u_2 \right) 
   \e_{q^\lambda}(\ell u_2)
  .
\end{displaymath}

We now set 
\begin{multline*}
  T_2(r,s;h_1,u_2, \ell) 
  :=
  \sum_{m} \sum_n
  \conjugate{f_{\kappa_1,\kappa_2}}\left(
    q^{\kappa_1} u_2 + q^{\kappa_1} \left(2ms+s^2q^{\kappa_1}\right) n^2
  \right) \\
  \e_{q^{\lambda}}
  \Big(
  h_1 
  \left( 2ms+s^2q^{\kappa_1} \right) 
  (n+r)^2
  \Big)
  \e_{q^{\kappa_2}} \Big( \ell m^2(2nr+r^2) \Bigl) 
\end{multline*}
so that $ S_{\text{diag}}(r,s,\ell)$ rewrites to
\begin{align}\label{eq:Sdiag-to-T2}
  S_{\text{diag}} (r,s,\ell) 
  &
    =
    q^{\lambda} \
    \fourier{\chi_\alpha^*}(\ell)
    \sum_{\abs{h_1}\leq H}
    q^{\lambda} \fourier{\chi_{\alpha,H}^*}(h_1)
    q^{\lambda}\fourier{\chi_{\alpha,H}^*}(h_1-\ell)
    F_{\kappa_1,\kappa_2}(h_1) \conjugate{F_{\kappa_1,\kappa_2}}(h_1-\ell)
  \\ \nonumber
  &
  \qquad \qquad\qquad \qquad\qquad 
  \frac 1{q^{\lambda}}
  \sum_{0\le u_2 < q^{\lambda}}
  f_{\kappa_1,\kappa_2}\left(q^{\kappa_1} u_2\right)
  \e_{q^\lambda}(\ell u_2) \, 
  T_2(r,s;h_1,u_2,\ell)  
  \\ \nonumber
  & \qquad + O\left( K^{-1} q^{\mu+\nu}  \right)
  .
\end{align}
We take the sum over $\abs{\ell} \le q^{\rho_3}$ and apply the
Cauchy-Schwarz inequality:
\begin{align*}
   \sum_{\abs{\ell} \le q^{\rho_3}} \abs{S_{\text{diag}} (r,s,\ell) }
   & \le  
   \sum_{\abs{h_1}\leq H}
      \abs{q^\lambda \fourier{\chi_{\alpha,H}^*}(h_1)
        F_{\kappa_1,\kappa_2}(h_1) } 
        \\
        &\qquad \qquad
        \frac 1{q^{\lambda}}
       \sum_{0\le u_2 < q^{\lambda}}
       \sum_{\abs{\ell} \le q^{\rho_3}}
       \abs{ q^\lambda  \fourier{\chi_{\alpha,H}^*}(h_1-\ell) 
          F_{\kappa_1,\kappa_2}(h_1-\ell)}
        \abs{T_2(r,s;h_1,u_2,\ell)}
         \\
         & \qquad + O\left( K^{-1} q^{\mu+\nu+\rho_3}  \right)
         \\
         &\le 
         \sum_{\abs{h_1}\leq H}
         \abs{q^\lambda \fourier{\chi_{\alpha,H}^*}(h_1)
        F_{\kappa_1,\kappa_2}(h_1) } 
        \\
       &\qquad
     \frac 1{q^{\lambda}}
       \sum_{0\le u_2 < q^{\lambda}} 
        \left(   
       \sum_{\abs{\ell} \le q^{\rho_3}}
       \abs{ q^\lambda  \fourier{\chi_{\alpha,H}^*}(h_1-\ell) 
          F_{\kappa_1,\kappa_2}(h_1-\ell)}^2
       \right)^{\frac 12}
       \\
       & \qquad \qquad \qquad
       \left(
       \sum_{\abs{\ell} \le q^{\rho_3}}
       \abs{T_2(r,s;h_1,u_2,\ell)}^2
       \right)^{\frac 12} 
       \\
       &\qquad 
        + O\left( K^{-1} q^{\mu+\nu+\rho_3}  \right)
       \\
       &= 
       \frac 1{q^{\lambda}}
       \sum_{0\le u_2 < q^{\lambda}} 
       \sum_{\abs{h_1}\leq H}
         \abs{q^\lambda \fourier{\chi_{\alpha,H}^*}(h_1)
        F_{\kappa_1,\kappa_2}(h_1) } 
        \\
        &\qquad\qquad
        \left(   
       \sum_{\abs{\ell} \le q^{\rho_3}}
       \abs{ q^\lambda  \fourier{\chi_{\alpha,H}^*}(h_1-\ell) 
          F_{\kappa_1,\kappa_2}(h_1-\ell)}^2
       \right)^{\frac 12}
       \\
       & \qquad \qquad \qquad
       \left(
       \sum_{\abs{\ell} \le q^{\rho_3}}
       \abs{T_2(r,s;h_1,u_2,\ell)}^2
       \right)^{\frac 12}
       \\
        & \qquad + O\left( K^{-1} q^{\mu+\nu+\rho_3}  \right)
        . 
\end{align*}
The main goal is now to estimate the last sum
$\sum_\ell \abs{T_2(r,s;h_1,u_2,\ell)}^2$
uniformly.
For the remaining part we can then apply Lemma~\ref{Lemixedestimate}.

The sum $T_2(r,s;h_1,u_2,\ell)$ is actually a sum that is quite
similar to the original sum.
However, the degree in $m$ is reduced from $2$ to $1$.

We apply again Cauchy-Schwarz inequality for the sum over $n$,
followed by the variant of the van der Corput inequality for the sum
over $m$ given in Lemma~\ref{lemma:van-der-corput}, with the
corresponding parameters $N=q^\mu$, $N'= q^{\mu-\rho_2}$ and
$R = q^{\rho_2}$, for some $\rho_2 < \mu$ to be chosen later.  This
leads to
\begin{align}\label{eq:from-T2-to-T3}
  \abs{T_2(r,s;h_1,u_2,\ell)}^2
  & \ll
  q^{2\mu+2\nu-\rho_2}
  \\
  & + 
  q^{\mu+\nu-\rho_2}
  \sum_{1\le r_2 < q^{\rho_2} }
  \left(1-\frac{r_2}{q^{\rho_2}}\right)
  \Re\left(T_3(r,s,r_2;h_1,u_2,\ell) \right)
  \nonumber
\end{align}
and consequently to 
\begin{align}\label{eq:from-T2-to-T3-2}
  \sum_{\abs{\ell}\le q^{\rho_3}}
  \abs{T_2(r,s;h_1,u_2,\ell)}^2
  & \ll
  q^{2\mu+2\nu +\rho_3-\rho_2}
  \\
  & + 
  q^{\mu+\nu-\rho_2}
  \sum_{1\le r_2 < q^{\rho_2} }
  \left(1-\frac{r_2}{q^{\rho_2}}\right)
  \sum_{\abs{\ell}\le q^{\rho_3}}
  \Re\left(T_3(r,s,r_2;h_1,u_2,\ell) \right),
  \nonumber
\end{align}
where 
\begin{align*}
  T_3(r,s,r_2;h_1,u_2,\ell)
  &
    =
    \sum_m \sum_n 
    \conjugate{f_{\kappa_1,\kappa_2}}
    \left(
    q^{\kappa_1} u_2
    + q^{\kappa_1} \left(2(m + r_2q^{\mu-\rho_2})s+s^2q^{\kappa_1}\right) n^2 
    \right) \\
  & \qquad \qquad \quad   
    f_{\kappa_1,\kappa_2} 
    \left(
    q^{\kappa_1} u_2 + q^{\kappa_1} \left(2ms+s^2q^{\kappa_1}\right) n^2
    \right)
  \\
  & \qquad \qquad \quad 
    \e_{q^{\lambda}}
    \Big(  h_1 2 r_2 s q^{\mu-\rho_2}  (n+r)^2 \Big)
  \\
  & \qquad \qquad \quad 
    \e_{q^{\kappa_2}}
    \Big( \ell(2mr_2 q^{\mu-\rho_2} + r_2^2q^{2\mu-2\rho_2})(2nr+r^2) \Big)    
.
\end{align*}
For convenience we set
\begin{align*}
  \kappa'_1 = \kappa_1 + \mu- \rho_2.
\end{align*}
We immediately observe that 
\begin{align*}
  &
    \conjugate { f_{\kappa_1,\kappa_2}}
    \left(
    q^{\kappa_1} u_2
    + q^{\kappa_1} \left(2(m + r_2q^{\mu-\rho_2})s+s^2q^{\kappa_1}\right) n^2 
    \right)   
    f_{\kappa_1,\kappa_2}
    \left(
    q^{\kappa_1} u_2 + q^{\kappa_1} \left(2ms+s^2q^{\kappa_1}\right) n^2
    \right) 
  \\
  &= 
    \conjugate { f_{\kappa'_1,\kappa_2}}
    \left(
    q^{\kappa_1} u_2
    + q^{\kappa_1} \left(2(m + r_2q^{\mu-\rho_2})s+s^2q^{\kappa_1}\right) n^2 
    \right)   
     f_{\kappa'_1,\kappa_2}
    \left(
    q^{\kappa_1} u_2 + q^{\kappa_1} \left(2ms+s^2q^{\kappa_1}\right) n^2
    \right) 
    .
\end{align*} 
Thus, $T_3(r,s,r_2;h_1,u_1)$ rewrites to 
\begin{align*}
  T_3(r,s,r_2;h_1,u_2,\ell)
  &
    =
    \sum_m \sum_n 
    \conjugate {f_{\kappa_1',\kappa_2}}
    \left(
    q^{\kappa_1} u_2
    + q^{\kappa_1} \left(2(m + r_2q^{\mu-\rho_2})s+s^2q^{\kappa_1}\right) n^2 
    \right) \\
  & \qquad \qquad \quad   
     f_{\kappa_1',\kappa_2}
    \left(
    q^{\kappa_1} u_2 + q^{\kappa_1} \left(2ms+s^2q^{\kappa_1}\right) n^2
    \right)  \\
  & \qquad \qquad \quad 
    \e_{q^{\lambda}}
    \Big(  h_1 2 r_2 s q^{\mu-\rho_2}  (n+r)^2 \Big)  \\  
   & \qquad \qquad \quad 
    \e_{q^{\kappa_2}}
    \Big( \ell(2mr_2 q^{\mu-\rho_2} + r_2^2q^{2\mu-2\rho_2})(2nr+r^2) \Big)    
.
\end{align*}
At this level we do another Fourier analysis,
applying~\eqref{eq:vaaler-approximation-f_kappa1-kappa2}  with a
substantially shorter interval of digits (between $\kappa'_1$ and
$\kappa_2$) that we have to detect,
taking:
\begin{equation}\label{eq:definition-lambda'}
  \lambda' = \kappa_2-\kappa'_1 = \nu + 2\rho + \widetilde{\rho} + \rho_2
  ,
\end{equation}
\begin{displaymath}
  \alpha' = q^{-\lambda'},
\end{displaymath}
and
\begin{displaymath}
  H' = q^{\lambda'} K - 1.
\end{displaymath}

Using \eqref{eq:definition-f-kappa1-kappa2-H}
with $\alpha'$ in place of $\alpha$
and $H'$ in place of $H$, let
\begin{align*}
  T'_3(r,s,r_2;h_1,u_2,\ell)
  &  =
  \sum_m \sum_n 
  \conjugate {f_{\kappa_1',\kappa_2,H'}}
  \left(
  q^{\kappa_1} u_2
  + q^{\kappa_1} \left(2(m + r_2q^{\mu-\rho_2})s+s^2q^{\kappa_1}\right) n^2 
  \right) 
  \\
  & \qquad \qquad \quad 
  f_{\kappa_1',\kappa_2,H'}
  \left(
    q^{\kappa_1} u_2 + q^{\kappa_1} \left(2ms+s^2q^{\kappa_1}\right) n^2
  \right)
  \\
  & \qquad \qquad \quad 
  \e_{q^{\lambda}}
  \Big(  h_1 2 r_2 s q^{\mu-\rho_2}  (n+r)^2 \Big) 
  \\
  & \qquad \qquad \quad 
  \e_{q^{\kappa_2}}
  \Big( \ell(2mr_2 q^{\mu-\rho_2} + r_2^2q^{2\mu-2\rho_2})(2nr+r^2) \Big)    
  .
\end{align*}
Since $\abs{f_{\kappa'_1,\kappa_2}}\leq 1$,
we observe by \eqref{eq:vaaler-approximation-f_kappa1-kappa2}
that $\abs{f_{\kappa'_1,\kappa_2,H}}\leq 2$.
Using the identity
\begin{displaymath}
  z_1z_2 - z'_1z'_2
  =
   (z_1-z'_1)z'_2 + z_1(z_2-z'_2)
   ,
\end{displaymath}
by~\eqref{eq:vaaler-approximation-f_kappa1-kappa2} we have
\begin{displaymath}
  \abs{T_3(r,s,r_2;h_1,u_2,\ell) - T'_3(r,s,r_2;h_1,u_2,\ell)}
  \leq
    E_3(r_2) +   2 E_3(0)
\end{displaymath}
with
\begin{align*}
  E_3(r_2)
  &
    =
    \sum_m \sum_n
    q^{\kappa_2-\kappa'_1}
    \sum_{\abs{k}< K}
    \fourier{B_{\alpha',H'}}\left(k q^{\kappa_2-\kappa'_1}\right)
    \e_{q^{\kappa'_1}}\left(
    k ( q^{\kappa_1} u_2
    +
    q^{\kappa_1}
    \left(2(m + r_2q^{\mu-\rho_2})s+s^2q^{\kappa_1}\right) n^2 )
    \right)
  \\
  &
    =
    q^{\kappa_2-\kappa'_1}
    \sum_{\abs{k}< K}
    \fourier{B_{\alpha',H'}}\left(k q^{\kappa_2-\kappa'_1}\right)
    \sum_m \sum_n
    \e_{q^{\mu-\rho_2}}\left(
    k ( u_2
    +
    \left(2(m + r_2q^{\mu-\rho_2})s+s^2q^{\kappa_1}\right) n^2 )
    \right)
    .
\end{align*}
For $k\neq 0$, provided $K \leq q^{\mu-\rho-\rho_2-2}$,
we apply Lemma~\ref{lemma:double-exponential-sum-mn^2}
with
\begin{math}
  \xi_3 = 2 k s q^{-\mu+\rho_2}
\end{math}
which satisfies
\begin{displaymath}
  0 < \norm{\xi_3} = \abs{\xi_3} \leq \frac12,
\end{displaymath}
and this gives, uniformly for $r_2\in\{0,\ldots,q^{\rho_2}-1\}$,
\begin{multline*}
  \abs{
    q^{-\mu-\nu}
    \sum_m \sum_n
    \e_{q^{\mu-\rho_2}}\left(
      k ( u_2
      +
      \left(2(m + r_2q^{\mu-\rho_2})s+s^2q^{\kappa_1}\right) n^2 )
    \right)
  }
  \\
  \ll
  K^{\frac14}q^{-\frac14(\mu-\rho-\rho_2)}
  + (\mu+\nu) q^{-\nu-\frac12\rho_2}
  + q^{-\frac12\nu}
  + (\mu+\nu) q^{-\frac12\mu}
  \ll
  K^{\frac14}q^{-\frac14(\mu-\rho-\rho_2)}
  .
\end{multline*}
It follows that, uniformly for $r_2\in\{0,\ldots,q^{\rho_2}-1\}$,
\begin{displaymath}
  E_3(r_2)
  \ll
  q^{\kappa_2-\kappa'_1}
  \fourier{B_{\alpha',H'}}\left(0\right)
  q^{\mu+\nu}
  +
  q^{\kappa_2-\kappa'_1}
  \sum_{1\leq \abs{k}< K}
  \fourier{B_{\alpha',H'}}\left(k q^{\kappa_2-\kappa'_1}\right)
  q^{\mu+\nu}
  K^{\frac14}q^{-\frac14(\mu-\rho-\rho_2)}
  ,
\end{displaymath}
which, by the definitions of $\alpha'$, $H'$ and by
\eqref{eq:vaaler-coef-majoration}, gives
\begin{displaymath}
  E_3(r_2)
  \ll
  q^{\mu+\nu}
  \left(
    K^{-1} 
    +
    K^{\frac14}q^{-\frac14(\mu-\rho-\rho_2)}
  \right)
  .
\end{displaymath}
This means that, up to a negligible error term,
we can replace $T_3(r,s,r_2;h_1,u_2,\ell)$
by $T'_3(r,s,r_2;h_1,u_2,\ell)$,
and \eqref{eq:from-T2-to-T3} is replaced by
\begin{multline}\label{eq:from-T2-to-T'3}
  \abs{T_2(r,s;h_1,u_2,\ell)}^2
  \ll
  q^{2\mu+2\nu-\rho_2}
  + 
  q^{2\mu+2\nu}
  \left(
    K^{-1} 
    +
    K^{\frac14}q^{-\frac14(\mu-\rho-\rho_2)}
  \right)  
  \\
  + 
  q^{\mu+\nu-\rho_2}
  \sum_{1\le r_2 < q^{\rho_2} }
  \left(1-\frac{r_2}{q^{\rho_2}}\right)
  \Re\left(T'_3(r,s,r_2;h_1,u_2,\ell) \right)
  .
\end{multline}
By~\eqref{eq:definition-f-kappa1-kappa2-H}
we can write
\begin{align*}
  & T_3'(r,s,r_2;h_1,u_2,\ell)
  \\
  &
    =  
    q^{2\lambda'} 
    \sum_{\abs{h_5}\le H'} \sum_{\abs{h_6}\le H'}  
    \fourier{\chi_{\alpha',H'}}(h_5)
    \conjugate{\fourier{\chi_{\alpha',H'}}}(h_6) 
    F_{\kappa'_1,\kappa_2}(h_5)
    \conjugate{F_{\kappa'_1,\kappa_2}}(h_6)   \\
  &
    \qquad
    \sum_{m,n} \e_{q^{\kappa_2}} 
    \left(
    h_5 q^{\kappa_1}
    \left( u_2 + \left(2ms+s^2q^{\kappa_1}\right) n^2 \right)  
    -
    h_6 q^{\kappa_1}
    \left(
    u_2 + \left(2(m+r_2 q^{\mu-\rho_2})s+s^2q^{\kappa_1}\right) n^2
    \right)
    \right) \\ 
  &\qquad \qquad
    \e_{q^{\lambda}}\left( 2 h_1 r_2 q^{\mu-\rho_2} s (n+r)^2 \right) 
    \e_{q^{\kappa_2}}
    \Big( \ell(2mr_2 q^{\mu-\rho_2} + r_2^2q^{2\mu-2\rho_2})(2nr+r^2) \Big)
  \\ 
  &
    = 
    q^{2\lambda'} 
    \sum_{\abs{h_5}\le H'} \sum_{\abs{h_6}\le H'}  
    \fourier{\chi_{\alpha',H'}}(h_5)
    \conjugate{\fourier{\chi_{\alpha',H'}}}(h_6) 
    F_{\kappa'_1,\kappa_2}(h_5) \conjugate{F_{\kappa'_1,\kappa_2}}(h_6)   \\
  &\qquad
    \sum_{m,n} \e_{q^{\lambda}} 
    \left( h_5 \left( u_2 + \left(2ms+s^2q^{\kappa_1}\right) n^2 \right)  
    -h_6\left( u_2 + \left(2(m+r_2 q^{\mu-\rho_2})s+s^2q^{\kappa_1}\right) n^2  \right)\right) \\
  &\qquad \qquad
    \e_{q^{\lambda}}\left( 2h_1 sr_2 q^{\mu-\rho_2}(n+r)^2 \right)
    \e_{q^{\kappa_2}}
    \Big( \ell(2mr_2 q^{\mu-\rho_2} + r_2^2q^{2\mu-2\rho_2})(2nr+r^2) \Big)
   .
\end{align*}
Let us write
\begin{equation}
  \label{eq:T'3-to-T4-T5}
  T_3'(r,s,r_2;h_1,u_2,\ell)
  =
  T_4(r,s,r_2;h_1,u_2,\ell)
  +
  T_5(r,s,r_2;h_1,u_2,\ell)
\end{equation}
where $T_4(r,s,r_2;h_1,u_2,\ell)$ denotes the contribution in
$T_3'(r,s,r_2;h_1,u_2,\ell)$ of the terms such that $h_5\neq h_6$,
and  $T_5(r,s,r_2;h_1,u_2,\ell)$ denotes the contribution in
$T_3'(r,s,r_2;h_1,u_2,\ell)$ of the terms such that $h_5= h_6$.

We first focus on $T_4(r,s,r_2;h_1,u_2,\ell)$.
The appearing exponential sums are of the form
\begin{displaymath}
  \sum_{m,n} a_m b_n  \e\left( \xi_3 mn^2 + \xi_1 mn \right),
\end{displaymath}
where
\begin{displaymath}
  \xi_3 = \frac {2s(h_5-h_6)}{q^{\lambda}}.
\end{displaymath}
Note that by assumption  $K = q^{\rho_1} \leq q^{\mu-\rho-\rho_2-3}$.
Consequently
\begin{displaymath}
  \abs{\xi_3}
  =
  \abs{\frac {2s(h_5-h_6)}{q^{\lambda}}}
  \leq
  \abs{\frac {4 q^\rho q^{\kappa_2-\kappa'_1} K}{q^{\kappa_2-\kappa_1}}}
  = 4 K q^{-\mu+\rho+\rho_2} \leq \frac12
  ,
\end{displaymath}
hence for $h_5\neq h_6$, we have
$2s (h_5-h_6) \not\equiv 0 \bmod q^\lambda$,
and
\begin{displaymath}
  q^{-\mu-\nu-3\rho} = q^{-\lambda}
  \leq \norm{\xi_3}
  = \abs{\xi_3}
  \ll K q^{-\mu+\rho+\rho_2}
  .
\end{displaymath}
It follows that if $h_5 \ne h_6$ we can apply
Lemma~\ref{lemma:double-exponential-sum-mn^2} and obtain
\begin{align*}
  q^{-\mu-\nu} 
  \sum_{m,n} a_m b_n  \e\left( \xi_3 mn^2 + \xi_1 mn\right) 
  &
    \ll 
    K^{\frac14} q^{-\frac14(\mu-\rho-\rho_2)} 
    +
    \left(q^{-\frac12(\nu-3\rho)} + q^{-\frac12\mu}\right)(\mu+\nu)
    .
\end{align*}
It follows that
\begin{multline*}
  T_4(r,s,r_2;h_1,u_2,\ell)
  \ll
  q^{\mu+\nu}
  \left(
    q^{\lambda'} 
    \sum_{\abs{h}\le H'} 
    \abs{
      \fourier{\chi_{\alpha',H'}}(h)
      F_{\kappa'_1,\kappa_2}(h)
    }
  \right)^2
  \\
  \times
  \left(
    K^{\frac14} q^{-\frac14(\mu-\rho-\rho_2)} 
    +
    \left(q^{-\frac12(\nu-3\rho)} + q^{-\frac12\mu}\right)(\mu+\nu)
  \right),
\end{multline*}
and by \eqref{eq:L1-mean-chi_H-F_kappa1-kappa2} we get
\begin{multline}\label{eq:estimate-T4}
  T_4(r,s,r_2;h_1,u_2,\ell)
  \ll
  q^{\mu+\nu}
  (\log K)^2
  \left(
    \sum_{0\leq \ell < q^{\kappa_2-\kappa'_1}}
    \abs{
      F_{\kappa'_1,\kappa_2}(\ell)
    }
  \right)^2
  \\
  \times
  \left(
    K^{\frac14} q^{-\frac14(\mu-\rho-\rho_2)} 
    +
    \left(q^{-\frac12(\nu-3\rho)} + q^{-\frac12\mu}\right)(\mu+\nu)
  \right)
\end{multline}
and consequently
\begin{align*}
  \sum_{\abs{\ell}\le q^{\rho_3}}
  \abs{T_4(r,s,r_2;h_1,u_2,\ell)}
  &\ll 
   q^{\mu+\nu+\rho_3}
  (\log K)^2
  \left(
    \sum_{0\leq \ell < q^{\kappa_2-\kappa'_1}}
    \abs{
      F_{\kappa'_1,\kappa_2}(\ell)
    }
  \right)^2
  \\
  &\qquad \times
  \left(
    K^{\frac14} q^{-\frac14(\mu-\rho-\rho_2)} 
    +
    \left(q^{-\frac12(\nu-3\rho)} + q^{-\frac12\mu}\right)(\mu+\nu)
  \right)
   .
\end{align*}

We now concentrate on $T_5(r,s,r_2;h_1,u_2,\ell)$.
We have
\begin{multline*}
  T_5(r,s,r_2;h_1,u_2,\ell)
  \\
  = 
  q^{2\lambda'} 
  \sum_{\abs{h_5}\le H'} 
  \abs{\fourier{\chi_{\alpha',H'}}(h_5) F_{\kappa'_1,\kappa_2}(h_5)}^2
  \sum_{m,n} 
  \e_{q^{\lambda}}\left(
    2 h_1 s r_2 q^{\mu-\rho_2}(n+r)^2
    -h_5 \left(2r_2s q^{\mu-\rho_2}\right) n^2
  \right)
  \\
  \times
  \e_{q^{\kappa_2}}
  \Big( \ell(2mr_2 q^{\mu-\rho_2} + r_2^2q^{2\mu-2\rho_2})(2nr+r^2) \Big)
  \\= 
  q^{2\lambda'} 
  \sum_{\abs{h_5}\le H'} 
  \abs{\fourier{\chi_{\alpha',H'}}(h_5) F_{\kappa'_1,\kappa_2}(h_5)}^2
  S_5(r,s,r_2;h_1,h_5;\ell)
  +
  O\left( q^{\nu-\rho+\rho_3+\rho_2} \right)
  ,
\end{multline*}
where,
since $\lambda' = \lambda-(\mu-\rho_2) = \nu + 2\rho + \widetilde{\rho} + \rho_2$,
\begin{align}\label{eq:definition-S5}
  & 
  S_5(r,s,r_2;h_1,h_5,\ell)
  \\
  &:=  \sum_m \sum_n \e_{q^{\lambda'}}
  \left( 2h_1 sr_2 (n+r)^2  - 2h_5sr_2 n^2 \right)
    \e_{q^{\kappa_2}}
  \Big( \ell(2mr_2 q^{\mu-\rho_2} + r_2^2q^{2\mu-2\rho_2})(2nr+r^2) \Big)
  .  
  \nonumber
\end{align}
After summing over $\abs{\ell}\le q^{\rho_3}$ we, thus, have
\begin{align*}
  \sum_{\abs{\ell}\le q^{\rho_3}}
  T_5(r,s,r_2;h_1,u_2,\ell)
  &= 
   \sum_{\abs{h_5}\le H'} 
  \abs{ q^{\lambda'}  \fourier{\chi_{\alpha',H'}}(h_5) F_{\kappa'_1,\kappa_2}(h_5)}^2
  \sum_{\abs{\ell}\le q^{\rho_3}}
  S_5(r,s,r_2;h_1,h_5;\ell) \\
  &+
  O\left( q^{\nu-\rho+2\rho_3+\rho_2} \right)
  .
\end{align*}

By Lemma~\ref{lemma:incomplete-gauss-sum-2} 
we have
\begin{displaymath}
  \abs{S_5(r,s,r_2;h_1,h_5)}
  \ll
  q^\mu
  \left(
  q^{\nu-\lambda'} + \lambda' \log q \right)
  q^{\frac 34\lambda'} \gcd\left(2sr_2(h_1-h_5),q^{\lambda'}\right)^{\frac 14}
  .
\end{displaymath}
which by \eqref{eq:definition-lambda'} gives
\begin{align*}
  \abs{S_5(r,s,r_2;h_1,h_5)}
  &
    \ll
    q^\mu
    \left(
    q^{-3\rho-\rho_2} + \lambda' \log q \right)
    q^{\frac 34(\nu+2\rho +\widetilde{\rho} +\rho_2)}
    \gcd\left(2sr_2(h_1-h_5),q^{\lambda'}\right)^{\frac 14}
  \\
  &
    \ll
    \lambda' 
    q^{\mu+\nu}
        q^{-\frac 14 (\nu-6\rho-3\widetilde{\rho}-3\rho_2)}
    \gcd\left(2sr_2(h_1-h_5),q^{\lambda'}\right)^{\frac 14}
    .
\end{align*}
Let $\rho_5$ be an integer such that
$0\leq \rho_5 \leq \nu-6\rho-3\widetilde{\rho} -3\rho_2$,
to be chosen later.
If
\begin{displaymath}
  \gcd\left( h_1-h_5, q^{\lambda'} \right)
  \leq q^{\nu-6\rho-3\widetilde{\rho} -3\rho_2-\rho_5}
  ,
\end{displaymath}
then 
\begin{displaymath}
  \abs{S_5(r,s,r_2;h_1,h_5)}
  \ll
  \lambda' \left( \log q \right)
  q^{\mu+\nu-\frac14\rho_5}
  \sqrt{
    \gcd\left(s,q^{\lambda'}\right)
    \gcd\left(r_2,q^{\lambda'}\right)
  },
\end{displaymath}
and, using \eqref{eq:L2-mean-chi_H-F_kappa1-kappa2},
get a contribution to $ \sum_{\ell} T_5(r,s,r_2;h_1,u_2,\ell)$
\begin{displaymath}
  \ll
  \lambda'
  q^{\mu+\nu+\rho_3-\frac14\rho_5}
  \sqrt{
    \gcd\left(s,q^{\lambda'}\right)
    \gcd\left(r_2,q^{\lambda'}\right)
  }
  .
\end{displaymath}
We will see later that averaging over $s$ and $r_2$
using~\eqref{eq:gcd-sum} and~\eqref{eq:sigma-x-q-lambda},
the $\gcd$'s will produce a only constant factor.

Thus, it remains to study the case
\begin{displaymath}
  \gcd\left( h_1-h_5, q^{\lambda'} \right)
  >
  q^{\nu-2\rho-\widetilde{\rho} -\rho_2-\rho_5}.
\end{displaymath}
Since $q$ is a prime number, this condition implies that
\begin{displaymath}
h_5 \equiv h_1 \bmod q^{\nu-2\rho-\widetilde{\rho}-\rho_2-\rho_5}.
\end{displaymath}
Notice that
\begin{math}
  \nu-2\rho-\widetilde{\rho} -\rho_2-\rho_5 < \kappa_2-\kappa'_1 = \nu+2\rho+\widetilde{\rho} +\rho_2.
\end{math}
In this case we sum over $\ell$ and estimate then trivially:
\begin{multline*} 
  \sum_{\abs{\ell}\le q^{\rho_3} } 
  S_5(r,s,r_2;h_1,h_5,\ell)  
  = 
  \sum_m \sum_n 
  \e_{q^{\lambda'}}
  \left( 2h_1 sr_2 (n+r)^2  - 2h_5sr_2 n^2 \right)
  \\
  \sum_{\abs{\ell}\le q^{\rho_3} } 
  \e_{q^{\kappa_2}}
  \Big( \ell(2mr_2 q^{\mu-\rho_2} + r_2^2q^{2\mu-2\rho_2})(2nr+r^2) \Big).
\end{multline*}
Note that 
\begin{displaymath}
  2nr + r^2 \gg q^\nu r 
  \quad\mbox{and}\quad
  2mr_2 q^{\mu-\rho_2} + r_2^2q^{2\mu-2\rho_2} \gg q^{2\nu-\rho_2} r_2.
\end{displaymath}
Hence, we have
\begin{displaymath}
  \sum_{\abs{\ell}\le q^{\rho_3} } 
  \e_{q^{\kappa_2}}
  \Big( \ell(2mr_2 q^{\mu-\rho_2} + r_2^2q^{2\mu-2\rho_2})(2nr+r^2) \Big)
  \ll 
  \frac{q^{\kappa_2} }{ q^{2\mu+\nu -\rho_2   } r r_2  }
  = 
  \frac{q^{\rho + \widetilde{\rho} + \rho_2} }{ r r_2  }
\end{displaymath}
which leads to 
\begin{displaymath}
  \abs{ \sum_{\abs{\ell}\le q^{\rho_3} } S_5(r,s,r_2;h_1,h_5,\ell) }
  \ll 
  \frac{ q^{\mu+\nu + \widetilde{\rho} + \rho + \rho_2} }{r r_2}
\end{displaymath}

Thus, if $h_5 \equiv h_1 \bmod q^{\nu-2\rho-\widetilde{\rho}-\rho_2-\rho_5}$
then we get a contribution to $T_5(r,s,r_2;h_1,u_2,\ell)$ that is upper bounded by
\begin{displaymath}
  \ll
  \frac{ q^{\mu+\nu + \widetilde{\rho} + \rho + \rho_2} }{r r_2}
    \sum_{\substack{\abs{h_5}\leq H'\\h_5 \equiv h_1 \bmod q^{\nu-3\rho-\rho_2-\rho_5} }}
  \abs{ q^{\lambda'} \fourier{\chi_{\alpha',H'}}(h_5) F_{\kappa'_1,\kappa_2}(h_5)}^2 
\end{displaymath}
which, since
$F_{\kappa'_1,\kappa_2}$ is $q^{\kappa_2-\kappa'_1}$-periodic
by \eqref{eq:periodicity-F_kappa1-kappa2},
is
\begin{displaymath}
  \ll
  \frac{ q^{\mu+\nu + \widetilde{\rho} + \rho + \rho_2} }{r r_2}
  \sum_{
    \substack{0\leq \ell< q^{\kappa_2-\kappa'_1}\\
      \ell \equiv h_1 \bmod q^{\nu-3\rho-\rho_2-\rho_5} }}
  \abs{F_{\kappa'_1,\kappa_2}(\ell)}^2
  \sum_{h_7\in\Z}
  q^{2\lambda'}
  \abs{
    \fourier{\chi_{\alpha',H'}}(\ell+h_7 q^{\kappa_2-\kappa'_1})
  }^2
  ,
\end{displaymath}
which is, by Lemma~\ref{lemma:chi-convolution-chi-average},
\begin{displaymath}
  \ll
  \frac{ q^{\mu+\nu + \widetilde{\rho} + \rho + \rho_2} }{r r_2}
  \sum_{
    \substack{0\leq \ell< q^{\kappa_2-\kappa'_1}\\
      \ell \equiv h_1 \bmod q^{\kappa_2-\kappa'_2} }}
  \abs{F_{\kappa'_1,\kappa_2}(\ell)}^2
  ,
\end{displaymath}
where $\kappa'_2 = \kappa_2-(\nu-2\rho-\widetilde{\rho} -\rho_2-\rho_5)$.
By \eqref{eq:FT-product-formula}, we have
\begin{displaymath}
    F_{\kappa'_1,\kappa_2}(\ell)
    =
    F_{\kappa'_1,\kappa'_2}\left(\frac{\ell}{q^{\kappa_2-\kappa'_2}}\right)
    F_{\kappa'_2,\kappa_2}(\ell)
    ,
\end{displaymath}
and, since 
$F_{\kappa'_2,\kappa_2}$ is $q^{\kappa_2-\kappa'_2}$-periodic
by \eqref{eq:periodicity-F_kappa1-kappa2},
using \eqref{eq:quadratic-mean-F_kappa1-kappa2},
we obtain
\begin{align*}
  \sum_{
  \substack{0\leq \ell< q^{\kappa_2-\kappa'_1}\\
  \ell \equiv h_1 \bmod q^{\kappa_2-\kappa'_2} }}
  \abs{F_{\kappa'_1,\kappa_2}(\ell)}^2
  &
    =
    \sum_{0\leq \ell' < q^{\kappa'_2-\kappa'_1}}
    \abs{
    F_{\kappa'_1,\kappa'_2}\left(
    \frac{h_1+\ell' q^{\kappa_2-\kappa'_2}}{q^{\kappa_2-\kappa'_2}}
    \right)
    F_{\kappa'_2,\kappa_2}\left(h_1+\ell' q^{\kappa_2-\kappa'_2}\right)
    }^2
  \\
  &
    =
    \abs{F_{\kappa'_2,\kappa_2}\left(h_1\right)}^2
    \sum_{0\leq \ell' < q^{\kappa'_2-\kappa'_1}}
    \abs{
    F_{\kappa'_1,\kappa'_2}\left(
    \frac{h_1}{q^{\kappa_2-\kappa'_2}} + \ell' 
    \right)
    }^2
  \\
  &
    =
    \abs{F_{\kappa'_2,\kappa_2}\left(h_1\right)}^2
    .
\end{align*}
Summing up we get a contribution that is upper bounded by
\begin{displaymath}
  \ll
  \frac{ q^{\mu+\nu + \widetilde{\rho} + \rho + \rho_2} }{r r_2} 
  \max_{h_1} \abs{F_{\kappa'_2,\kappa_2}\left(h_1\right)}^2.
\end{displaymath}
Summing up we have
\begin{align*}
  \sum_{\abs{\ell}\le q^{\rho_3}}
  \abs{T_5(r,s,r_2;h_1,u_2,\ell)} 
  & \ll
    \lambda'
  q^{\mu+\nu+\rho_3-\frac14\rho_5}
  \sqrt{
    \gcd\left(s,q^{\lambda'}\right)
    \gcd\left(r_2,q^{\lambda'}\right)
  }
  \\
  & +
  \frac{ q^{\mu+\nu + \widetilde{\rho} + \rho + \rho_2} }{r r_2} 
  \max_{h_1} \abs{F_{\kappa'_2,\kappa_2}\left(h_1\right)}^2
\end{align*}
and consequently
\begin{align}\nonumber
  &
  q^{-\mu-\nu-2\rho}
  \sum_{1\leq r < q^\rho} \sum_{1\leq s < q^\rho} 
  \abs{S_{28}(r,s)}
   \\ \label{eq:S28-upperbound}
   & 
  \ll
   \sum_{0\le k < q^{\rho_3}}
   \abs{F_{\kappa_2-\rho_3,\kappa_2}(k)}
   \Biggl( 
  q^{\frac 12 \rho_3}
  (\log K)
  \left(
    \sum_{0\leq \ell < q^{\kappa_2-\kappa'_1}}
    \abs{
      F_{\kappa'_1,\kappa_2}(\ell)
    }
  \right)
  \\ \nonumber
  &\qquad \qquad \qquad \qquad \quad
  \left(
    K^{\frac18} q^{-\frac18(\mu-\rho-\rho_2)} 
    +
    \left(q^{-\frac14(\nu-3\rho)} + q^{-\frac14\mu}\right)(\mu+\nu)
  \right)
  \\  \nonumber
  &\qquad \qquad \qquad \qquad 
  + q^{\frac 12\rho_3}
  \left(
    K^{-\frac 12} 
    +
    K^{\frac18}q^{-\frac18(\mu-\rho-\rho_2)}
  \right) 
  \\  \nonumber
  &\qquad \qquad \qquad \qquad 
  + \sqrt{\rho \rho_2}
  q^{\frac 12 \widetilde{\rho}} 
  \max_{h_1} \abs{F_{\kappa'_2,\kappa_2}\left(h_1\right)}
   \\  \nonumber
   &\qquad \qquad \qquad \qquad  
  +\sqrt{\lambda'}
  q^{\rho_3-\frac18\rho_5}
 +  q^{\frac 12 (\rho_3-\rho_2)} 
+ K^{-1} q^{\rho_3}  
   \Biggr)
  .
\end{align}

\medskip
\subsection{Choosing all parameters}~

Finally we will choose all parameters in a way that we get
a non-trivial type II sum estimate.

We recall that $\beta_1 = \frac 15$, that is, by symmetry we
have to consider the range \eqref{eq:inequality-mu-nu}.

%This will only effect the range for the type I sum computations.

Furthermore we observe that the right choice of $\rho_3$ is crucial
to obtain non-trivial upper bounds. From (\ref{eq:S27-upperbound}) 
it follows that $\rho_3-\widetilde{\rho}$ 
has to be large enough so that it beats the factor
\begin{displaymath}
  \left( 
    \sum_{0\leq \ell< q^{\kappa_2-\kappa_1}}
    \abs{F_{\kappa_1,\kappa_2}(\ell)} \right)^2 \ll q^{2 \eta(f) \lambda } 
    = q^{2\eta(f)(\mu+\nu + 2 \rho + \widetilde{\rho})}
  ,
\end{displaymath}
that is, we certainly need the estimate
\begin{displaymath}
  \rho_3 \ge d_1' \eta(f)\lambda  \ge d_1' \eta(f) \nu 
\end{displaymath}
for a proper chosen constant $d_1'> 0$. 
On the other hand in (\ref{eq:S28-upperbound}) we have the term
\begin{displaymath}
  \sum_{0\le k < q^{\rho_3}}
  \abs{F_{\kappa_2-\rho_3,\kappa_2}(k)} 
  \max_{h_1} \abs{F_{\kappa'_2,\kappa_2}\left(h_1\right)}
  \ll
  q^{\rho_3 \eta_q - c(f) (\nu-2\rho-\widetilde{\rho} -\rho_2-\rho_5)}.
\end{displaymath}
In order to ensure that this term is negligible it is necessary that
\begin{displaymath}
  \rho_3 \eta(f) \le d_2' c(f) \nu
\end{displaymath}
for some constant $d_2' > 0$. This means that we should have
\begin{equation}\label{eqkappacrelation}
d_1' \eta(f)^2 \le d_2' c(f).
\end{equation}
Here we recall that in the case of $f(n) = \e(\gamma s_q(n))$,
we have
$\eta(f) = \eta_q \le c'' \log\log q/\log q$ and 
$c(f) = c_q  \ge c'\|(q-1)\gamma\|^2  /\log q$.
Hence, if $q$ is sufficiently large 
(and $\|(q-1)\gamma\|$ is not too small) then
(\ref{eqkappacrelation}) is satisfied. 
%Consequently it is possible to choose $\rho_3$ accordingly.
Actually, since $\mu$ and $\nu$ are proportional and $\mu$
appears in almost all estimates we will set
$\rho_3 = d_3'\eta(f)\mu$ for some proper constant $d_3'>0$. 

By assumption we know that $\eta(f)$ is sufficiently small:
\begin{displaymath}
  \eta(f) \le \frac 1{2.000}.
\end{displaymath}
Since $c(f) \le \eta(f)$, the same upper bound applies to $c(f)$.
We then set
\begin{displaymath}
  \rho_3 = 128 \eta(f)\mu  < \frac 1{15} \mu
\end{displaymath}
and 
\begin{align*}
  \rho & =  \frac{\rho_3}2, \\
  \widetilde{\rho} &= c(f) \mu, \\
  \rho_1 &= \rho_2 = 2 \rho_3, \\
  \rho_5 &= 10 \rho_3.
\end{align*}

With this choice we analyze the upper bounds
(\ref{eq:S26-upperbound})--(\ref{eq:S28-upperbound})
in detail. Since we have
\begin{displaymath}
  \kappa_2-\kappa_1
  = \lambda
  = \mu + \nu + 2 \rho+\widetilde{\rho} \le 6 \mu
\end{displaymath}
it follows that
\begin{displaymath}
  \left( 
    \sum_{0\leq \ell< q^{\kappa_2-\kappa_1}}
    \abs{F_{\kappa_1,\kappa_2}(\ell)} \right)^4
  \ll q^{4\eta(f)\lambda } 
  \le q^{16 \eta(f)\mu } 
  \le q^{\mu/120}.
\end{displaymath} 
Furthermore we have
\begin{displaymath}
  \mu - \rho - \rho_1 \ge \frac{\mu}2
\end{displaymath}
which implies that
\begin{displaymath}
  q^{-\mu-\nu-2\rho}
  \sum_{1\leq r < q^\rho} \sum_{1\leq s < q^\rho} \abs{S_{26}(r,s)}
  \ll \mu^5 q^{\mu\left( \frac 1{120} - \frac 1{20} \right)} 
  \ll q^{-\frac 1{30} \mu}
  .
\end{displaymath}

In order to handle (\ref{eq:S27-upperbound})
we observe (using $c(f)\le \eta(f)$) that
\begin{displaymath}
  \frac 14(\rho_3 - \widetilde{\rho} )
  \ge \frac{\rho_3}8 
  = 16 \eta(f)\mu 
\end{displaymath}
and
\begin{displaymath}
 \frac {\mu} 4 > \frac 16(\mu - \rho-\widetilde{\rho} - \rho_1) 
  \ge \frac {\mu}{12} > \frac{\rho_3}4 > \frac 14(\rho_3 - \widetilde{\rho} )
\end{displaymath}
which implies that
\begin{displaymath}
  q^{-\mu-\nu-2\rho}
  \sum_{1\leq r < q^\rho} \sum_{1\leq s < q^\rho} 
  \abs{S_{27}(r,s)}
  \ll
  \mu^2 q^{2\lambda \eta_q} q^{-16 \eta(f)\mu } 
  \le
  \mu^2 q^{8 \mu \eta_q} q^{-16 \eta(f)\mu } 
  = 
  \mu^2 q^{- 8 \eta(f)\mu }
  .
\end{displaymath}

The most involved part is the upper bound (\ref{eq:S28-upperbound}).
Here we have
\begin{align*}
  \sum_{0\le k < q^{\rho_3}} \abs{F_{\kappa_2-\rho_3,\kappa_2}(k)}
  &\ll
    q^{\rho_3 \eta_q} \le q^{128 \eta(f)^2 \mu}
  \\
  q^{\frac 12 \rho_3} 
  \sum_{0\le \ell < q^{\kappa_2-\kappa_1'}} \abs{F_{\kappa_1',\kappa_2}(\ell)}
  &\ll 
    q^{\frac 12 \rho_3 + \lambda' \eta_q}
    \le
    q^{69 \mu \eta_q} 
  \\
  \frac 18 ( \mu-\rho- \rho_1-\rho_2) 
  &\ge \frac{\mu}{16},
  \\
  \frac 14(\nu- 3\rho) 
  &\ge \frac {\mu}8 
  \\
  q^{\frac 12 \rho_3} \left( q^{-\frac 12 \rho_1}
  + q^{-\frac 18(\mu-\rho-\rho_1-\rho_2)} \right)
  & \ll q^{\frac 12 \rho_3} \left( q^{-\rho_3} + q^{-\frac 1{16}\mu} \right)
    \ll q^{-64 \mu \eta_q} + q^{-\frac 1{33} \mu} 
  \\
  q^{\frac 12 \widetilde{\rho}} 
  \max_{h_1} \abs{F_{\kappa_2',\kappa_2}(h_1)}
  &\ll 
    q^{\frac 12 c(f) \mu} q^{-c(f)(\kappa_2-\kappa_2')} 
    \le q^{-\frac 12 c(f) \mu}
  \\
  \rho_3-\frac{\rho_5}8 
  &= - \frac {\rho_3}4 
  \\
  \frac 12(\rho_3-\rho_2) 
  &= - \frac{\rho_3}2 
  \\
  \rho_3-\rho_1 
  &= - \rho_3
    .
\end{align*}
This leads to 
\begin{multline*}
  q^{-\mu-\nu-2\rho}
  \sum_{1\leq r < q^\rho} \sum_{1\leq s < q^\rho} 
  \abs{S_{28}(r,s)}
  \\
  \ll
  \mu q^{128 \eta(f)^2 \mu}
  \left( 
    q^{69 \eta(f) \mu} q^{-\frac 1{16} \mu} 
    + q^{-64 \eta(f) \mu} + q^{-\frac 1{33} \mu } 
    + q^{-\frac 12 c(f) \mu }  + q^{-128 \eta(f) \mu } 
    \right) 
    .
\end{multline*}
Since the term $q^{-128 \eta(f) \mu}$ dominates the first three ones we gets
\begin{displaymath}
  q^{-\mu-\nu-2\rho}
  \sum_{1\leq r < q^\rho} \sum_{1\leq s < q^\rho} 
  \abs{S_{28}(r,s)}
  \ll
  \mu q^{\left( 128 \eta(f)^2 - \frac 12 c(f) \right)  \mu} +
  q^{- 128 \eta(f) \mu }.
\end{displaymath}

Summing up, this directly leads to 
\begin{multline*}
  q^{-\mu-\nu-2\rho}
  \sum_{1\leq r < q^\rho} \sum_{1\leq s < q^\rho} 
  \abs{S_{24}(r,s)}
  \\
  \ll
  q^{-\frac 1{30} \mu} +
  \mu^2 q^{- 8 \eta(f)\mu } +
  \mu q^{128 \eta(f)^2 \mu}
  \left( 
    q^{69 \eta(f) \mu} q^{-\frac 1{16} \mu} 
    + q^{-64 \eta(f) \mu} + q^{-\frac 1{33} \mu } 
    + q^{-\frac 12 c(f) \mu }  + q^{-128 \eta(f) \mu } 
    \right) 
    .
\end{multline*}

\medskip
\subsection{Upper bounds for $S_{20}(\theta)$}~

By applying (\ref{eq:S22-to-S23}) and (\ref{eq:approximation-S23-S24}) 
and by observing that
\begin{displaymath}
q^{-\rho_1} + q^{\rho_1+\rho-\mu} \ll q^{-256 \eta(f)\mu } 
\end{displaymath}
we derive
\begin{multline*}
  q^{-\mu-\nu-2\rho}
  \sum_{1\leq r < q^\rho} \sum_{1\leq s < q^\rho} 
  \abs{S_{22}(r,s)}
  \\
  \ll
  q^{-256 \eta(f)\mu }  + 
  q^{-\frac 1{30} \mu} +
  \mu^2 q^{- 8 \eta(f)\mu } \\
  +
  \mu q^{128 \eta(f)^2 \mu}
  \left( 
    q^{69 \eta(f) \mu} q^{-\frac 1{16} \mu} 
    + q^{-64 \eta(f) \mu} + q^{-\frac 1{33} \mu } 
    + q^{-\frac 12 c(f) \mu }  + q^{-128 \eta(f) \mu } 
  \right) \\
  % \ll \mu^2 q^{- 8 \eta(f)\mu }  + 
    \mu^2 q^{(128 \eta(f)^2 - c(f))\mu }
        .
\end{multline*}
Finally by using (\ref{eqS20est}) we obtain
\begin{proposition}\label{proposition:TypeII-sums-final-estimate}
  Uniformly in the range \eqref{eq:inequality-mu-nu} and $\theta\in\R$,
  we have
  \begin{displaymath}
    q^{-\mu-\nu} \abs{S_{20}(\theta)}
    \ll  %\mu^{\frac 12} q^{- 2 \eta(f)\mu }  + 
    \mu^{\frac 12} q^{(32 \eta(f)^2 - \frac 14 c(f))\mu },
  \end{displaymath}
  provided $f$ satisfies $\eta(f) \le 1/2000$.
\end{proposition}
\begin{remark}
  If $\eta(f) > 1/2000$ then we do not get a proper upper bound.
\end{remark}

%\color{red}
%(The next lines have to be rewritten: we must average on $r_2$ and $s$
%before introducing $\eta$ (if we like to do so)). 

\section{Sums of Type I}\label{section:typeIsums}

We use the abbreviation $M= q^\mu$ and $N = q^\nu$.

In order to estimate the Type I sums,
it is sufficient to estimate the sums
\begin{displaymath}
  S_I(\theta) = 
  \sum_{q^{\mu-1} \leq m < q^\mu} 
  \abs{ \sum_{n \in I_\nu(m)} f(m^2n^2) \e(\theta m n) }
  ,
\end{displaymath}
uniformly over all intervals $I_\nu(m) \subseteq [q^{\nu-1}, q^\nu[$
under the condition that
\begin{displaymath}
  \frac{\mu}{\mu+\nu} \leq \beta_1
  .
\end{displaymath}
Splitting $q$-adically,
the general case can be reduced to this case.

We recall that we have chosen $\beta_1 = \frac 15$, see
(\ref{eq:choice-beta_1}), so we have to consider the range
\begin{displaymath}
  \mu \le \frac 14 \nu.
\end{displaymath}
As in the case of type II sums we will assume a slightly 
more general condition
\begin{equation}
  \label{eq:assumption-mu-nu-typeI}
  \mu \le \frac 14 \nu + C
\end{equation}
for some absolute constant $C > 0$.

After Cauchy-Schwarz, van der Corput inequality
(with $R = q^\rho$ where $1\leq \rho < \nu$)
and using the carry property decribed in Lemma~\ref{Le:carryproperty}, taking
\begin{equation}
  \label{eq:definition-lambda-type-I}
  \lambda_1 = 2\mu + \nu + 2\rho
\end{equation}
and
\begin{math}
  I_\nu(m,r) =   I_\nu(m) \cap   \left( I_\nu(m) - r \right)
  ,
\end{math}
we get 
\begin{displaymath}
  \abs{S_I(\theta)}^2 \le  
 \frac{M^2N^2}{R}
  +
  \frac{M N}{R}
  \Re
  \sum_{m} 
  \sum_{1\leq r < R}
  \left(1-\frac{r}{R}\right)
  \sum_{n \in I_\nu(m,r)}
  f_{\lambda_1}(m^2(n+r)^2)
  \conjugate{f_{\lambda_1}(m^2n^2)}
  + O(M^2N^2 q^{-\rho})
  .
\end{displaymath}
By Fourier inversion we have
\begin{displaymath}
  f_{\lambda_1}(m^2(n+r)^2)
  \conjugate{f_{\lambda_1}(m^2n^2)}
  =
  \sum_{0\leq h < q^{\lambda_1}}
  \sum_{0\leq k < q^{\lambda_1}}
  F_{\lambda_1}(h)
  \conjugate{F_{\lambda_1}(k)}
  \e_{q^{\lambda_1}}
  \left(
    h m^2(n+r)^2 - k m^2 n^2
  \right)
  .
\end{displaymath}

We will partition $h$ and $k$ according to the condition
\begin{displaymath}
  \gcd(h-k,q^{\lambda_1}) = q^\delta,
\end{displaymath}
where 
\begin{displaymath}
  0\le \delta \le \lambda_1.
\end{displaymath}
We have 
\begin{displaymath}
  k = h + \ell q^\delta,
\end{displaymath}
where we can assume by periodicity that $0\le \ell < q^{\lambda_1 - \delta}$ 
and that $\gcd(\ell, q) = 1$. 

For convenience we set
\begin{equation}\label{eq:definition-S_delta-typeI}
  S_\delta(r)
  :=
  \sum_{0\leq h < q^{\lambda_1}}
  \sum_{\substack{0\leq \ell < q^{\lambda_1-\delta}\\ \gcd(\ell,q) = 1}}
  F_{\lambda_1}(h)
  \conjugate{F_{\lambda_1}(h + \ell q^\delta)}
  \sum_{m}  \sum_{n \in I_\nu(m,r)}
  \e_{q^{\lambda_1}}
  \left( \ell q^\delta m^2 n^2 + h m^2(2nr+r^2) \right)
  .
\end{equation}
We have then
\begin{equation}\label{eqSIest}
  \abs{S_I(\theta)} 
  \ll 
  q^{\mu+\nu- \frac 12 \rho} 
  + q^{\frac 12(\mu+\nu)}
  \left( 
    \frac 1R
    \sum_{1 \le r \le R} 
    \sum_{0\le \delta\le \lambda_1} 
    \abs{S_\delta(r)} 
  \right)^{\frac 12}
  .
\end{equation}

\medskip
\subsection{Upper bounds for $\abs{S_\delta(r)}$
  for large difference $\lambda_1-\delta$}~ 

The exponential sums appearing in $S_\delta(r)$
defined in \eqref{eq:definition-S_delta-typeI}
are of the form
\begin{displaymath}
  \sum_{n \in I_\nu(m,r)}
  \e
  \left( 
    \frac{\ell}{ q^{\lambda_1-\delta} } m^2 n^2 + 
    \frac{h}{q^{\lambda_1}} m^2(2n r + r^2)
  \right),
\end{displaymath}
where $\ell$ and $q^{\lambda_1-\delta}$ are coprime.
By Lemma~\ref{lemma:incomplete-gauss-sum-2}
with $m$ replaced by
\begin{math}
  q^{\lambda_1-\delta} /   \gcd\left( m^2, q^{\lambda_1-\delta} \right)
\end{math}
we, thus, obtain 
\begin{multline*}
  \abs{
    \sum_{n \in I_\nu(m,r)}
    \e\left( 
      \frac{\ell}{ q^{\lambda_1-\delta} } m^2 n^2 + 
      \frac{h}{q^{\lambda_1}} m^2(2n r + r^2)
    \right)
  }
  \\
  \ll 
  \gcd\left(m^2, q^{\lambda_1-\delta} \right)^{\frac 12}
  q^{\nu - \frac 12 (\lambda_1-\delta)}
  +
  \nu^{1/2} q^{\frac12 \nu}
  +
  \nu^{1/2} q^{\frac 12 (\lambda_1-\delta)}
  .
\end{multline*}
Since by Lemma~\ref{lemma:gcd-sum}
and \eqref{eq:tau-q-lambda},
remembering that $q$ is a prime number,
\begin{displaymath}
  \sum_m \gcd\left(m^2, q^{\lambda_1-\delta}\right)^{\frac 12}
  \leq
  \sum_m \gcd\left(m, q^{\lambda_1-\delta}\right)
  \ll 
  \tau\left( q^{\lambda_1-\delta} \right) q^\mu
  \ll \nu^{\omega(q)} q^\mu
  = \nu q^\mu
\end{displaymath}
we deduce
\begin{multline*}
  \sum_{m}
  \abs{
    \sum_{n \in I_\nu(m,r)}
    \e\left( 
      \frac{\ell}{ q^{\lambda_1-\delta} } m^2 n^2 + 
      \frac{h}{q^{\lambda_1}} m^2(2n r + r^2)
    \right)
  }
  \\
  \ll
  \nu q^{\mu+\nu} q^{- \frac 12 (\lambda_1-\delta)}
  +
  \nu^{1/2} q^{\mu+\frac12 \nu}
  +
  \nu^{1/2} q^{\mu+\frac 12 (\lambda_1-\delta)}
  .
\end{multline*}

Next we write $h$ as $h = h_1 + h_2 q^\delta$ with
$0\le h_1 < q^\delta$ and $0\le h_2 < q^{\lambda_1-\delta}$. 
By \eqref{eq:strongly-multiplicative-FT-product-formula} we have 
for $k = h + \ell q^\delta = h_1 + (h_2+\ell) q^\delta $ 
\begin{displaymath}
  F_{\lambda_1}(h)
  \conjugate{F_{\lambda_1}(k)} 
  = 
  \abs{F_{\delta}(h_1)}^2
  F_{\lambda_1-\delta}(h_1q^{-\delta} + h_2)
  \conjugate{F_{\lambda_1-\delta}(h_1q^{-\delta} + h_2 + \ell)}.
\end{displaymath}  
By applying twice Lemma~\ref{Le1} we have
\begin{displaymath}
  \sum_{0\le h_2 <  q^{\lambda_1-\delta}}
  \sum_{0\le \ell <  q^{\lambda_1-\delta}} 
  \abs{
    F_{\lambda_1-\delta}(h_1q^{-\delta} + h_2)
    \conjugate{F_{\lambda_1-\delta}(h_1q^{-\delta} + h_2 + \ell)}
  }
  \ll q^{2\eta(f) (\lambda_1-\delta)}
  ,
\end{displaymath}
and, since, by \eqref{eq:quadratic-mean-F_kappa1-kappa2},
\begin{displaymath}
  \sum_{0\le h_1 < q^\delta} \abs{F_{\delta}(h_1)}^2 = 1,
\end{displaymath}
we deduce
\begin{equation}\label{eqSdelta1}
  \abs{S_\delta(r)}
  \ll 
  \nu q^{\mu+\nu} q^{  - \left(\frac 12 - 2\eta(f) \right)  (\lambda_1-\delta)}
  +
  \nu^{1/2} q^{\mu+\nu} q^{-\frac12 \nu + 2\eta(f) (\lambda_1-\delta)}
  +
  \nu^{1/2}
  q^{\mu+\nu} q^{-\nu + \left( \frac 12 + 2\eta(f)\right)(\lambda_1-\delta)} 
  . 
\end{equation}
The first term provides a non-trivial bound if $\eta(f) < \frac 14$
and the difference $\lambda_1-\delta$ is sufficiently large.
The second term is non-trivial if $\eta(f) < \frac{1}{20}$
%(this will be remembered in \eqref{eq:assumption-eta_q}).
and the third term is non-trivial if $\eta(f) < \frac{1}{20}$
and \eqref{eq:assumption-mu-nu-typeI} hold. 
Actually we use the general assumption $\eta(f) \le 1/2000$. So
we are on the safe side.

\medskip
\subsection{Upper bounds for $\abs{S_\delta(r)}$ for small difference $\lambda_1-\delta$}~
In a second step we consider the case, where the 
difference
\begin{equation}\label{eq:definition-kappa-type-I}
  \kappa := \lambda_1-\delta
\end{equation}
is (relatively) small.
It will be made precise later but it is convenient
to assume immediately that
\begin{equation}
  \label{eq:initial-condition-kappa-typeI}
  \kappa < \nu.
\end{equation}
For example, if $\kappa = 0$ then we only have
the case $h=k$ which can be called the diagonal case.
So we consider now a sort of bigger diagonal.

We start with the exponential sum
\begin{displaymath}
  \sum_{n \in I_\nu(m,r)}
  \e\left(   
    \frac{\ell}{q^\kappa} m^2 n^2 + 
    \frac{2rh}{q^{\lambda_1}} m^2 n + 
    \frac{r^2h}{q^{\lambda_1}} m^2 
  \right)
\end{displaymath}
and split up $n$ in residue classes modulo $q^\kappa$:
$n = n_1 + n_2 q^\kappa$ with $0\le n_1 < q^\kappa$
and $q^{\nu-\kappa-1}\le n_2 < q^{\nu - \kappa}$
such that $n_1 + n_2 q^\kappa \in I_\nu(m,r)$.
This leads to $q^\kappa$ exponential sums,
where the exponent is just linear in $n_2$:
\begin{multline*}
  \sum_{n \in I_\nu(m,r)}
  \e\left(   
  \frac{\ell}{q^\kappa} m^2 n^2 + 
  \frac{2rh}{q^{\lambda_1}} m^2 n + 
  \frac{r^2h}{q^{\lambda_1}} m^2 
  \right)
  \\
  =
  \sum_{0\le n_1 < q^\kappa} 
  \e \left(
  \frac{\ell}{q^\kappa} m^2 n_1^2 + 
  \frac{2rh}{q^{\lambda_1}} m^2 n_1
  \right)
  \sum_{\substack{n_2\\ n_1 + n_2 q^\kappa \in I_\nu(m,r)}}
  \e \left(
  \frac{2rh}{q^{\delta}} m^2 n_2 + 
  \frac{r^2h}{q^{\lambda_1}} m^2 
  \right)
\end{multline*}
and this can be estimated,
using \eqref{eq:estimate-geometric-series}, by:
\begin{displaymath}
  \ll 
  q^\kappa
  \min\left( 
    q^{\nu-\kappa}, 
    \norm{ \frac{2rh m^2}{q^{\delta}} }^{-1}
  \right)
  ,  
\end{displaymath}
which leads to the upper bound 
\begin{equation}\label{eqSdeltaest2}
  \abs{S_\delta(r)}
  \ll
  q^{2\kappa} \sum_m 
  \max_\ell
  \sum_{0\le h < q^{\lambda_1}}
  \abs{ F_{\lambda_1}(h)
    \conjugate{F_{\lambda_1}(h+ \ell q^\delta)} 
  }
  \min
  \left( 
    q^{\nu-\kappa}, 
    \norm{ \frac{2rh m^2}{q^{\delta}} }^{-1}
  \right)
  .
\end{equation}

We now set
\begin{equation}\label{eq:definition-A-type-I}
  A
  :=
  \frac{q^\delta}{2rm^2}
  \leq
  \frac{q^{\lambda_1}}{2rm^2}
  ,
\end{equation}
and,
using \eqref{eq:definition-lambda-type-I} and
\eqref{eq:definition-kappa-type-I},
\begin{equation}\label{eq:definition-B-type-I}
  B := \frac{A}{q^{\nu-\kappa}}
  = \left( \frac M m \right)^2 \frac{R^2}{2r}
  \geq \frac{R}{2} =  \frac{q^\rho}{2}  
  .
\end{equation} 
Then the mapping
\begin{displaymath}
  t \mapsto \norm{ \frac tA } = 
  \norm{ \frac{2r t m^2}{q^{\delta}} }
\end{displaymath}
is a periodic function with period $A$ and that increases linearly for
$0\le t \le A/2$ (from $0$ to $\frac 12$) and decreases linearly for
$A/2 \le t \le A$ (from $\frac 12$ to $0$). Note that $A$ is usually
not an integer. Therefore we have to be careful because we
only evaluate at integer values $h$.

We first consider those $0\leq h < q^{\lambda_1}$ for which
\begin{equation}\label{eqhrestriction}
  \norm{ \frac hA } < \frac{1}{q^{\nu-\kappa}}.
\end{equation}
For those $h$, there exist $k\in\Z$ such that
\begin{displaymath}
  \frac{h}{A} - \frac{1}{q^{\nu-\kappa}}
  < k <
  \frac{h}{A} + \frac{1}{q^{\nu-\kappa}}
  .
\end{displaymath}
Remembering that $k$ is an integer
and $0\leq h < q^{\lambda_1}$,
this implies that
\begin{displaymath}
  0\leq k < \frac{q^{\lambda_1}}{A} +1
\end{displaymath}
and
\begin{math}
  -B < k A - h < B
  ,
\end{math}
hence
\begin{displaymath}
  -B-1 < \floor{k A} - h < B
  .
\end{displaymath}
Thus the whole set of $0\leq h < q^{\lambda_1}$
satisfying \eqref{eqhrestriction} is covered by
the following union of {\it almost arithmetic progressions}
\begin{displaymath}
  H_1 :=
  \bigcup_{\abs{h_0} < B + 1}
  \left\{
    h_0 + \floor{ k A } : 0\le k < \frac{q^{\lambda_1}}{A} + 1
  \right\}
  .
\end{displaymath}
If $h\in H_1$ then we use
\begin{displaymath}
  \min\left( 
    q^{\nu-\kappa}, 
    \norm{ \frac{2rh m^2}{q^{\delta}} }^{-1}
  \right)
  \leq q^{\nu-\kappa}
  ,
\end{displaymath}
and otherwise we will use the other term.
Hence we have
\begin{equation}\label{eq:S_delta-H1-H2}
  \abs{S_\delta(r)}
  \ll
  S_{\delta,H_1}(r)
  +
  S_{\delta,H_2}(r)
\end{equation}
with
\begin{math}
  H_2 = \{0,\ldots,q^{\lambda_1}-1\} \setminus H_1
  ,
\end{math}
\begin{displaymath}
  S_{\delta,H_1}(r)
  =
  q^{\nu+\kappa}
  \sum_m 
  \max_\ell
  \sum_{h\in H_1}
  \abs{ F_{\lambda_1}(h)
    \conjugate{F_{\lambda_1}(h+ \ell q^\delta)} 
  }
  ,
\end{displaymath}
\begin{displaymath}
  S_{\delta,H_2}(r)
  =
  q^{2\kappa} \sum_m 
  \max_\ell
  \sum_{h\in H_2}
  \abs{ F_{\lambda_1}(h)
    \conjugate{F_{\lambda_1}(h+ \ell q^\delta)} 
  }
  \norm{ \frac{2rh m^2}{q^{\delta}} }^{-1}
  .
\end{displaymath}
This leads to consider the sum
\begin{displaymath}
  \sum_{0\le k < \frac{q^{\lambda_1}}{A} + 1}
  \abs{
    F_{\lambda_1}(h_0 + \floor{k A})
    \conjugate{F_{\lambda_1}(h_0 + \floor{k A} + \ell q^\delta)}
  }
  ,
\end{displaymath}  
which, by Cauchy-Schwarz, is
\begin{displaymath}
  \le  
  \left( 
  \sum_{0\le k < \frac{q^{\lambda_1}}{A} + 1}
  \abs{ F_{\lambda_1}(h_0 + \floor{ k A } ) }^2
  \right)^{\frac 12} 
  \left(
  \sum_{0\le k < \frac{q^{\lambda_1}}{A} + 1}
  \abs{F_{\lambda_1}(h_0 + \floor{ k A } + \ell q^\delta)}^2
  \right)^{\frac 12}
  .
\end{displaymath}
By Lemma~\ref{lemma:L2-along-almost-arithmetic-progressions} 
with $\alpha = \floor{\log A/\log q}$,
for any $h'_0\in\Z$ we have
\begin{displaymath}
  \sum_{0\le k < \frac{q^{\lambda_1}}{A}}
  \abs{ F_{\lambda_1}\left( h'_0 + \floor{ k A } \right) }^2
  \ll
  \max_{t\in\R} \abs{F_\alpha(t)}^2
  .
\end{displaymath}
Here we have an additional term in the sums
for $k=\ceil{\frac{q^{\lambda_1}}{A}}$,
which is
\begin{displaymath}
  \abs{
    F_{\lambda_1}\left(
      h'_0 + \floor{ \ceil{\frac{q^{\lambda_1}}{A}} A }
    \right)
  }^2
  \leq
  \max_{t\in\R} \abs{F_{\lambda_1}(t)}^2
  \leq
  \max_{t\in\R} \abs{F_\alpha(t)}^2
  .
\end{displaymath}
Since
\begin{displaymath}
  \max_{t\in\R} \abs{F_\alpha(t)} \ll q^{-c(f) \alpha} \ll A^{-c(f)},
\end{displaymath}
we obtain
\begin{equation}\label{eq:correlation-along-quasi-periodic-sequence}
  \sum_{0\le k < \frac{q^{\lambda_1}}{A} + 1}
  \abs{ F_{\lambda_1}(h_0 + \floor{ k A })
    \conjugate{F_{\lambda_1}(h_0 + \floor{ k A } + \ell q^\delta)}} 
  \ll
  A^{-2c(f)} 
  ,
\end{equation}
which, by \eqref{eq:definition-B-type-I}, leads to
\begin{displaymath}
  \sum_{\abs{h_0} < B + 1} 
  \sum_{0\le k < \frac{q^{\lambda_1}}{A} + 1}
  \abs{ F_{\lambda_1}(h_0 + \floor{ k A })
    \conjugate{F_{\lambda_1}(h_0 + \floor{ k A } + \ell q^\delta)}}
  \ll
  B A^{-2c(f)} 
  =
  B^{1-2c(f)} 
  q^{-2c(f)(\nu-\kappa)}
  .
\end{displaymath}
Thus, multiplying by $q^{\nu+\kappa}$ and summing over $m$,
remembering \eqref{eq:definition-B-type-I}
we get
\begin{displaymath}
  S_{\delta,H_1}(r)
  \ll
  q^{\mu+\nu+\kappa} \frac{R^{2-4c(f)}}{r^{1-2c(f)}} q^{-2c(f)(\nu-\kappa)}
  .
\end{displaymath}

It remains to consider $h\in H_2$.
For those $h$ we have $0\leq h < q^{\lambda_1}$
and
\begin{displaymath}
  \norm{\frac{h}{A}} \geq \frac{1}{q^{\nu-\kappa}}
  ,
\end{displaymath}
which implies that there exist $k\in\Z$ such that
\begin{displaymath}
  \frac{1}{q^{\nu-\kappa}}
  \leq \abs{\frac{h}{A} -k}
  \leq \frac12
  .
\end{displaymath}
It follows that
\begin{displaymath}
  -\frac12
  \leq
  k < \frac{q^{\lambda_1}}{A} +\frac12
\end{displaymath}
and remembering that $k$ is an integer, that
\begin{displaymath}
  0\leq k < \frac{q^{\lambda_1}}{A} + 1
  .
\end{displaymath}
We also have
\begin{displaymath}
  B = \frac{A}{q^{\nu-\kappa}} \leq \abs{h - kA} \leq \frac{A}{2}
\end{displaymath}
which implies that
\begin{displaymath}
  B - 1 \leq \abs{h - \floor{kA}} \leq \frac{A}{2} + 1
  .
\end{displaymath}
Therefore we have
\begin{displaymath}
  H_2
  \subseteq
  \bigcup_{B-1 \le \abs{h_0} \le \frac{A}{2}+1}
  \left\{
    h_0 + \floor{kA}:\
    0\leq k < \frac{q^{\lambda_1}}{A} + 1
  \right\}
  .
\end{displaymath}
Furthermore, for $h = h_0 + \floor{kA}$
with $B-1 \le \abs{h_0} \le \frac{A}{2}+1$, we have
\begin{displaymath}
  \norm{ \frac{2rh m^2}{q^{\delta}} }
  =
  \norm{ \frac{h}{A} }
  =
  \norm{ \frac{h_0+ \floor{kA}}{A} }
  =
  \norm{ \frac{h_0 - \frp{kA}}{A} }
  ,
\end{displaymath}
hence
\begin{displaymath}
  \norm{ \frac{2rh m^2}{q^{\delta}} }
  \geq
  \norm{ \frac{h_0}{A} } - \frac{1}{A}
  \geq
  \norm{ \frac{\abs{h_0}-1}{A} } - \frac{2}{A}
  =
  \frac{\abs{h_0}-3}{A}
  .
\end{displaymath}

Using \eqref{eq:correlation-along-quasi-periodic-sequence},
this leads to an upper bound of the form
\begin{displaymath}
  S_{\delta,H_2}(r)
  \ll
  q^{2\kappa} 
  \sum_m \sum_{B-1 \le \abs{h_0} \le \frac{A}{2}+1}  
  A^{-2c(f)} 
  \frac{A}{\abs{h_0}-3}
  ,
\end{displaymath}
hence by \eqref{eq:definition-B-type-I},
\begin{align*}
  S_{\delta,H_2}(r)
  \ll
  (\log A)
  q^{\mu+2\kappa}
  \left( \frac{M^2}{m^2} \frac{R^2}r \right)^{1-2c(f)}
  q^{(1-2c(f))(\nu-\kappa)}
  \ll
  (\log A)
  q^{\mu+\nu+\kappa}
  \frac{R^{2-4c(f)}}{r^{1-2c(f)}}
  q^{-2c(f)(\nu-\kappa)}
  .
\end{align*}
Thus,
remembering \eqref{eq:definition-kappa-type-I}
and \eqref{eq:S_delta-H1-H2},
we get
\begin{equation}\label{eqSdelta3}
  \abs{S_\delta(r)}
  \ll
  \nu
  q^{\mu+\nu+\lambda_1-\delta}
  \frac{R^{2-4c(f)}}{r^{1-2c(f)}}
  q^{-2c(f)(\nu-\lambda_1+\delta)}
  \ll 
  \nu
  \frac{R^{2-4c(f)}}{r^{1-2c(f)}}
  q^{\mu+\nu} q^{-2c(f)\nu+(1+2c(f))(\lambda_1-\delta)}
  .
\end{equation}

\medskip
\subsection{Upper bound for $S_I(\theta)$}~

We distinguish between $0\le \delta \le \lambda_1 - \rho$ and 
$\lambda_1 - \rho < \delta \le \lambda_1$.

In the first case we can use the upper bound (\ref{eqSdelta1}) and
get
\begin{displaymath}
  \sum_{0\le \delta \le \lambda_1 - \rho} \abs{S_\delta(r)}
  \ll
  \nu q^{\mu+\nu} q^{  - \left(\frac 12 - 2\eta(f) \right) \rho}
  +
  \nu^{1/2} q^{\mu+\nu} q^{-\frac12 \nu + 2\eta(f) \lambda_1}
  +
  \nu^{1/2}
  q^{\mu+\nu} q^{-\nu + \left( \frac 12 + 2\eta(f)\right)\lambda_1} 
\end{displaymath}
whereas in the second case
the condition \eqref{eq:initial-condition-kappa-typeI} is fulfilled and
we apply the upper bound (\ref{eqSdelta3}) and get
\begin{displaymath}
  \sum_{\lambda_1 - \rho< \delta \le \lambda_1} \abs{S_\delta(r)}
  \ll 
  \nu
  \frac{R^{2-4c(f)}}{r^{1-2c(f)}}
  q^{\mu+\nu} q^{-2c(f)\nu+(1+2c(f))\rho}
\end{displaymath}
We now have to choose the parameters.
Additionally to the assumption
\eqref{eq:assumption-mu-nu-typeI}
we also assume that
\begin{equation}
  \label{eq:assumption-rho-nu}
  \rho \leq \frac1{20} \nu
   .
\end{equation}
By \eqref{eq:definition-lambda-type-I} this leads to
\begin{displaymath}
  \lambda_1 = 2\mu+\nu+2\rho
  \leq
  \frac12 \nu + \nu + \frac{1}{10} \nu
  =
  \frac{8}{5} \nu
\end{displaymath}
hence
\begin{displaymath}
  -\frac12 \nu + 2\eta(f) \lambda_1
  \leq
  -\frac12 \nu + \frac1{10} \cdot \frac85 \nu 
  = - \frac{17}{50} \nu
  < - \frac12 \rho
\end{displaymath}
and
\begin{displaymath}
  -\nu + \left( \frac 12 + 2\eta(f)\right)\lambda_1
  \leq
  -\nu + \frac6{10} \cdot \frac85 \nu = -\frac1{25}\nu
  < - \frac12 \rho
  ,
\end{displaymath}
so that
\begin{align*}
  \frac 1R
  \sum_{1 \le r \le R} 
  \sum_{0 \leq \delta \le \lambda_1}
  \abs{S_\delta(r)}
  &
    \ll 
    \nu q^{\mu+\nu} q^{  - \left(\frac 12 - 2\eta(f) \right) \rho}
    +
    \nu
    R^{1-2c(f)}
    q^{\mu+\nu} q^{-2c(f)\nu+(1+2c(f))\rho}  
  \\
  &
    \ll 
    \nu q^{\mu+\nu} q^{  - \left(\frac 12 - 2\eta(f) \right) \rho}
    +
    \nu
    q^{\mu+\nu} q^{-2c(f)\nu + 2\rho}  
    .
\end{align*}
We finally set
\begin{equation}\label{eqrhofinal}
  \rho = \floor{\frac{2 c(f) \nu}{\frac{5}{2}-2\eta(f)}}
  .
\end{equation}
By assumption $c(f) \le \eta(f) \le 1/2000$. Thus, 
\eqref{eq:assumption-rho-nu} is clearly satisfied.
By (\ref{eqSIest}) we obtain
\begin{equation}\label{eqSIest2}
  \abs{S_I(\theta)} 
  \ll
  q^{\mu+\nu - \rho}
  +
  \nu^{\frac 12}
  q^{\mu+\nu}
  q^{  - \frac12 \left(\frac 12 - 2\eta(f) \right) \rho}
  \ll
  \nu^{\frac 12}
  q^{\mu+\nu}
  q^{-\left(\frac 12 - 2\eta(f) \right) \frac{c(f)\nu}{\frac{5}{2}-2\eta(f)}}
  ,
\end{equation}
which gives
\begin{proposition}\label{proposition:TypeI-sums-final-estimate}
  Uniformly in the range \eqref{eq:assumption-mu-nu-typeI}
  and $\theta\in\R$,
  we have
  \begin{equation}\label{eqSIest2-2}
   q^{-\mu-\nu} 
   \abs{S_I(\theta)}
    \ll 
    \nu^{\frac 12}
        q^{- \frac{10}{51} c(f)\nu}
    .
  \end{equation}  
\end{proposition}

\section{Proof of the main results}

Finally, we will use
Proposition~\ref{proposition:TypeII-sums-final-estimate}
and
Proposition~\ref{proposition:TypeI-sums-final-estimate}
to check the conditions of Lemma~\ref{lemma:combinatorial-identity}.

We start with type II sums, that is, we consider sums of the form
\begin{equation}\label{eqtypeIIsum-2}
\sum_{M/q < m \le M} \ \sum_{x/(qm) < n \le x/m} 
a_m b_n f(m^2n^2) \e(\theta m n) 
\end{equation}
for $x^{\beta_1} \le M \le x^{\beta_2}$.
 We recall that $\beta_1 = \frac 15$.
By symmetry in $m$ and $n$
it is sufficient to consider $x>0$ and $M>0$ with
\begin{equation}\label{eqxMrelation}
x^{\beta_1} \le M \le x^{\frac12}.
\end{equation}
We now define positive integers $\mu$ and $\nu$ by the conditions
\begin{equation}\label{eqxxMrel}
  q^{\mu-1} \le M < q^\mu
  \quad \mbox{and} \quad
  q^{\nu-1} \le \frac xM < q^\nu.
\end{equation}
From \eqref{eqxxMrel} we can write
\begin{math}
  q^{\nu-1} M \le x
\end{math}
and by \eqref{eqxMrelation} we have
\begin{math}
  x \leq M^5
  ,
\end{math}
hence
\begin{displaymath}
  (q^{\nu-1})^{1/4} \leq M.
\end{displaymath}
Since $M \le q^\mu$ this gives 
\begin{displaymath}
  \frac{\nu}4 - \frac 14 \leq \mu.
\end{displaymath}
Similarly, by \eqref{eqxMrelation} we have $M^2 \leq x$
and by \eqref{eqxxMrel} we can write
\begin{math}
  x < q^\nu M,
\end{math}
hence
\begin{math}
  M < q^\nu
\end{math}
and since $q^{\mu-1} \leq M$, we get the property
\begin{displaymath}
\mu \le \nu + 1.
\end{displaymath}
We are now considering the sum (\ref{eqtypeIIsum-2}) (see also (\ref{definition-sums-type-II})),
where we have to sum over $m$ with $M/q \le m < M$ which implies
\begin{displaymath}
q^{\mu-2} \le \frac Mq \le m < M \le q^\mu.
\end{displaymath}
We can cover this range by the union of the two intervals 
$[q^{\mu-2},q^{\mu-1})$ and $[q^{\mu-1},q^\mu)$.
Similarly the sum over $n$ with $x/(qm) \le n < x/m$
can be covered, using \eqref{eqxxMrel}, by
\begin{displaymath}
  q^{\nu-2} \le \frac x{qM} \le \frac x{qm}
  \le n < \frac xm \le \frac {qx}M \le q^{\nu+1}
\end{displaymath}
or by the union of three intervals 
$[q^{\nu-2},q^{\nu-1})$, $[q^{\nu-1},q^\nu)$, and $[q^{\nu},q^{\nu+1})$.
Thus, we can cover the summation range in (\ref{eqtypeIIsum-2}) by
six rectangles of the form $[q^{\mu-1},q^\mu) \times [q^{\nu-1},q^\nu)$, where we
have to replace $(\mu,\nu)$ by 
$(\mu-1,\nu-1)$, $(\mu,\nu-1)$, $(\mu-1,\nu)$,
$(\mu,\nu)$, $(\mu-1,\nu+1)$, $(\mu,\nu+1)$.
Since $\mu$ and $\nu$ satisfy
\begin{math}
   \frac{\nu}4 - \frac 14 \leq \mu \le \nu+1
   ,
 \end{math}
it follows that 
the replaced pairs satisfy
\begin{align*}
  &
    \frac{\nu-1}4 - 1 \le \mu-1 \le \nu-1 +1, \\
  &
    \frac{\nu-1}4 \le \mu \le \nu-1 +2, \\
  &
    \frac{\nu}4 - \frac 54 \le \mu-1 \le \nu, \\ 
  &
    \frac{\nu}4 - \frac 14 \le \mu \le \nu+1, \\
  &
    \frac{\nu+1}4 - \frac 32 \le \mu-1 \le \nu+1 -1,\\
  &
    \frac{\nu+1}4  - \frac 12 \le \mu \le \nu+1.
\end{align*}
Thus, we can apply
Proposition~\ref{proposition:TypeII-sums-final-estimate}
with $C = 2$ and obtain (by applying it 
for every choice of the six pairs 
$(\mu-1,\nu-1)$, $(\mu,\nu-1)$, $(\mu-1,\nu)$, $(\mu,\nu)$,
$(\mu-1,\nu+1)$, $(\mu,\nu+1)$) 
\begin{align*}
   \abs{ \sum_{q^{\mu-2} \le m < q^\mu } \sum_{q^{\nu-2}\le n < q^{\nu+1}  } 
    a_m b_n f(m^2n^2) \e(\theta m n)   } 
&\ll q^{\mu+\nu} \left(
  %  \mu^{\frac 12} q^{- 2 \eta(f)\mu }  + 
   \mu^{\frac 12} q^{(32 \eta(f)^2 - \frac 14 c(f))\mu } \right)
  \\
    &\ll % x (\log x)^{\frac 12}  q^{-  2 \eta(f) (\mu+\nu)/5 } 
      x (\log x)^{\frac 12} q^{(32 \eta(f)^2 - \frac 14
      c(f))(\mu+\nu)/5 }
  \\
     &\ll  (\log x)^{\frac 12}
      x^{1-(\frac 14 c(f) - 32 \eta(f)^2)/5 } 
       .
\end{align*}
By setting some of the $a_m$ and $b_m$ artificially to zero we obtain also an upper bound
when we sum over any rectangle that is contained in
$[q^{\mu-2},q^\mu) \times [q^{\nu-1},q^{\nu+1})$.

In order to cover precisely the range $M/q \le m < M$,
 $x/(qm)\le n < x/q$ we
split up the range $q^{\mu-2} \le m < q^\mu$ into $K$ subintervals of the form
\[
I_j = \left[  
q^{\mu-2}\left( 1 + j \frac{q^2-1}K \right), 
q^{\mu-2}\left( 1 + (j+1) \frac{q^2-1}K \right)
 \right)
\]
with $0\le j < K$, where
\begin{displaymath}
K = \left\lfloor x^{\frac 12(\frac 14 c(f) - 32 \eta(f)^2)/5 }\right\rfloor,
\end{displaymath}
and by setting $a_m = 0$ if $m$ is not in this interval or not in 
$(M/q, M]$. More precisely we consider
sums of the form
\begin{align*}
S_{j}
% &= 
%  \sum_{q^{\mu-2} \leq m < q^{\mu}}
%  \sum_{q^{\nu-1} \leq n < q^{\nu+1}}
%  a_m b_n f(m^2n^2) \e(\theta m n) 
%  \\
  &=
  \sum_{ m \in I_j \cap (M/q, M ]  } \ 
  \sum_{q^{\nu-1} \leq n < q^{\nu+1}}
  a_m b_n f(m^2n^2) \e(\theta m n) 
\end{align*}
Furthermore in the corresponding
interval for $n$ (that is, $[q^{\nu-2},q^{\nu+1})$)
we also set some $b_n$ to zero so that we approximate the range $M/q \le m < M$, $x/(qm)\le n < x/m$ 
(after summing over $K$ rectangles) in a proper way:
\begin{align*}
S_{j}' 
%&= 
%  \sum_{ m \in I_j \cap (M/q, M ]  } \
%   \sum_{q^{\nu-1} \leq n < q^{\nu+1}}
%  a_m b_n f(m^2n^2) \e(\theta m n) 
%  \\
  &=
  \sum_{ m \in I_j \cap (M/q, M ]  } \
  \sum_{x/\left( q \max\{M/q, q^{\mu-2}\left( 1 + j \frac{q^2-1}K \right) \}  \right) \leq n 
< x / \left( \min\{ M, q^{\mu-2}\left( 1 + j \frac{q^2-1}K \right)  \} \right) }
  a_m b_n  f(m^2n^2) \e(\theta m n).
\end{align*}
So for each $j\in \{0,1,\ldots K-1\}$ we have
\[
S_j' \ll
(\log x)^{\frac 12}
      x^{1-(\frac 14 c(f) - 32 \eta(f)^2)/5 } 
      .
\]

The error term that comes from this approximation 
is trivially upper bounded (for each rectangle) by
\begin{align*}
&\abs{
S_j' - 
 \sum_{ m \in I_j \cap (M/q, M ]  } \
 \sum_{x/(qm) \leq n < x/m}
 a_m b_n f(m^2n^2) \e(\theta m n)
  } \\
&\qquad \ll  \frac{q^{\mu+\nu}}{K^2} \ll \frac{x}{K^2}.
\end{align*}
Thus, the total error is upper bounded by $\ll K x/K^2 = x/K$ and, consequently, the
sum (\ref{eqtypeIIsum-2}) can be upper bounded by
\begin{align*}
  \ll
  &
    K (\log x)^{\frac 12} 
    x^{1- (\frac 14 c(f) - 32 \eta(f)^2)/5 ) } 
    + \frac xK
  \\
  &
    \ll
    (\log x)^{\frac 12}
    x^{1-\frac 12(\frac 14 c(f) - 32 \eta(f)^2)/5 }.
\end{align*}

\medskip

For type I sums we proceed similiarly. 
For $M \le x^{\beta_1}$ we define $\mu$ and $\nu$ by (\ref{eqxxMrel}) and
obtain
\begin{displaymath}
  \mu \le \frac{\nu}4 + \frac 34.
\end{displaymath}
We again cover the range $M/q \le m < M$, $x/(qm)\le n < x/q$
by six instances for $\mu$ and $\nu$ (as above) so that we can apply 
Proposition~\ref{proposition:TypeI-sums-final-estimate}
with $C=2$. Thus, we obtain the upper bound
\begin{displaymath}
  \ll q^{\mu+\nu} \nu^{\frac 12} q^{-\frac {10}{51} c(f) \nu} 
  \ll (\log x)^{\frac 12} x^{1 - \frac{8}{51} c(f) }.
\end{displaymath}
Note that this estimate is uniform in the intervals $I_\nu(m)$.
Clearly they can be adjusted so that we describe any interval of
the form $(t,x/m]$ with $x/(qm)\le t \le x/m$. 

\medskip

Summing up, we have checked the assumptions of
Lemma~\ref{lemma:combinatorial-identity} 
with 
\begin{displaymath}
  U \ll 
  (\log x)^{\frac 12} 
x^{1-\min \left( \frac{8}{51} c(f), \, (\frac 18 c(f) - 16 \eta(f)^2)/5 \right) }
  \ll
    (\log x)^{\frac 12} 
x^{1-(\frac 18 c(f) - 16 \eta(f)^2)/5 }
.
\end{displaymath}
This directly leads to 
\begin{displaymath}
  \sum_{x/q< n \le x} \Lambda(n) f(n^2)\e(\theta n)
  \ll (\log x)^{\frac 52} 
   x^{1-(\frac 18 c(f) - 16 \eta(f)^2)/5 }
  ,
\end{displaymath}
and finally to
\begin{displaymath}
  \sum_{n\le x} \Lambda(n) f(n^2)\e(\theta n)
  \ll
    (\log x)^{\frac 52}
    x^{1-(\frac 18 c(f) - 16 \eta(f)^2)/5 }
   .
\end{displaymath}
This proves Theorem~\ref{Thmainresult}.

% \pagebreak
% \bibliographystyle{siam}
% \bibliography{articles,books,theses}
% \bigskip

\end{document}